\documentclass[reqno]{amsart}

\usepackage{mathrsfs}
\usepackage{amsmath}
\usepackage{amssymb}
\usepackage{cite}
\usepackage{latexsym}
\usepackage{graphicx}
\usepackage{float}

\usepackage{color}
\usepackage{comment}

\theoremstyle{plain}
\newtheorem{thm}{Theorem}[section]
\newtheorem{lem}[thm]{Lemma}

\newtheorem{prop}[thm]{Proposition}
\newtheorem{rmk}[thm]{Remark}

\def\D{\mathrm{D}}
\def\M{\mathscr{M}}
\def\N{\mathscr{N}}

\def\c{\mathrm{c}}
\def\d{\mathrm{d}}
\def\h{\mathrm{h}}

\def\Rset{\mathbb{R}}
\def\Sset{\mathbb{S}}
\def\Zset{\mathbb{Z}}

\def\hot{\mathrm{h.o.t}}

\def\epsilon{\varepsilon}

\DeclareMathOperator{\sign}{sign}

\DeclareMathOperator{\sech}{sech}
\DeclareMathOperator{\csch}{csch}

\DeclareMathOperator{\sn}{sn}
\DeclareMathOperator{\cn}{cn}
\DeclareMathOperator{\dn}{dn}

\makeatletter
 \@addtoreset{equation}{section}
\makeatother
\def\theequation{\arabic{section}.\arabic{equation}}

\begin{document}


\title[
Perturbations of codimension-two bifurcations]%
{Periodic perturbations of codimension-two bifurcations with a double zero eigenvalue
 in dynamical systems with symmetry}

\author{Kazuyuki Yagasaki}

\address{Department of Applied Mathematics and Physics, Graduate School of Informatics,
Kyoto University, Yoshida-Honmachi, Sakyo-ku, Kyoto 606-8501, JAPAN}
\email{yagasaki@amp.i.kyoto-u.ac.jp}

\date{\today}

\begin{abstract}
We study bifurcation behavior
 in periodic perturbations of two-dimensional symmetric systems
 exhibiting codimension-two bifurcations with a double eigenvalue
 when the frequencies of the perturbation terms are small.
We transform the periodically perturbed system to a simpler one
 which is a periodic perturbation of the normal form
 for codimension-two bifurcations with a double zero eigenvalue and symmetry,
 and apply the subharmonic and homoclinic Melnikov methods
 to analyze bifurcations occurring in the system.
In particular, we show that there exist transverse homoclinic or heteroclinic orbits,
 which yield chaotic dynamics, in wide parameter regions.
These results can be applied to three or higher-dimensional systems
 and even to infinite-dimensional systems
 with the assistance of center manifold reduction and the invariant manifold theory.
We illustrate our theory for a pendulum
 subjected to position and velocity feedback control
 when the desired position is periodic in time.
We also give numerical computations by the computer tool \texttt{AUTO}
 to demonstrate the theoretical results.
\end{abstract}

\maketitle

\section{Introduction}
Codimension-two bifurcations
 are fundamental and interesting phenomena in dynamical systems
 and have been studied extensively,
 since the seminal papers of Arnold \cite{A72} and Takens \cite{T74}.
Such bifurcations can occur not only in finite-dimensional systems
 but also in infinite-dimensional systems.
See, e.g., \cite{C81,CLW94,GH83,HI11,K04,W03} and references therein
 for mathematical explanations on the theory and applications.
Two parameters are needed at least
 for a codimension-two bifurcation point to be generically contained
 in the parameter space.
Among of them, codimension-two bifurcations occurring
 when the Jacobian matrices of the vector fields at equilibria have a double zero eigenvalue
 are especially important and now well understood.
In particular, different bifurcation diagrams are obtained,
 depending on whether the systems are symmetric or not \cite{C81,CLW94,GH83}.
Vast literature on applications exhibiting this type of bifurcation is also found
 (see, e.g., Section~20.6A of \cite{W03}).
Similar codimension-two bifurcations
 can occur for periodic orbits or diffeomorphisms \cite{BRS96,K04}.

Moreover, we notice
 that infinitely many bifurcation points of this type can exist
 in the two-dimensional parameter space.
Actually, it was shown in \cite{Y99} that
 this behavior occurs in the three-dimensional system
\begin{equation}
\begin{split}
&
\dot{z}_1=z_2,\\
&
\dot{z}_2=-\sin z_1-\delta_0 z_2+z_3,\\
&
\dot{z}_3=-\alpha z_3+\gamma(\theta_\d-z_1)-\delta_1z_2,
\end{split}
\quad
z=(z_1,z_2,z_3)\in\Sset^1\times\Rset\times\Rset,
\label{eqn:ex}
\end{equation}
where $\alpha,\delta_0>0$ and $\delta_1,\gamma,\theta_\d$ are constants
 and $\Sset^1=\Rset/2\pi\Zset$.
Here the system~\eqref{eqn:ex} represents a pendulum
 driven by a servo-motor and subjected to position and velocity feedback control
 with the desired position and velocity given by $z_1=\theta_\d$ and $z_2=0$,
 where $z_1$ and $z_2$, respectively,
 represent the position and angular velocity of the pendulum;
 $z_3$ the torque servo-motor;
 $\alpha$ the reciprocal of the time constant of the servo-motor;
 $\delta_0$ the damping constant of the pendulum
 resulting from servo-motor friction and other sources;
 and $\gamma$ and $\delta_1$, respectively,
 the position and velocity gain constants of the feedback control.

In this paper, we study bifurcation behavior
 in periodic perturbations of two-dimensional symmetric systems
 exhibiting codimension-two bifurcations with a double zero eigenvalue.
Specifically, we consider two-dimensional nonautonomous systems of the form
\begin{align}
\dot{x}=f(x;\mu)+\epsilon g(x,\omega t),\quad
x
\in\Rset^2,
\label{eqn:sys}
\end{align}
where $\epsilon$ is a small parameter such that $0<\epsilon\ll 1$,
 $\omega>0$ is a constant,
 and $\mu\in\Rset^2$ represents a control parameter vector.
Here $f:\Rset^2\times\Rset^2\to\Rset^2$ is $C^4$ and symmetric around $x=0$, i.e.,
\begin{equation}
f(-x;\mu)=-f(x;\mu),
\label{eqn:sym}
\end{equation}
which yields
\begin{equation}
f(0;\mu)=0,
\label{eqn:f0}
\end{equation}
so that the origin $x=0$ is an equilibrium when $\epsilon=0$.
Moreover, the Jacobian matrix of $f(x;\mu)$ at $x=0$ with $\mu=0$
 is assumed to be the $2\times 2$ Jordan normal form with a double zero eigenvalue,
\begin{equation}
\D_x f(0;0)=
\begin{pmatrix}
0 & 1\\
0 & 0
\end{pmatrix}=:J.
\label{eqn:J}
\end{equation}
In such a situation, bifurcation behavior occurring in \eqref{eqn:sys} when $\epsilon=0$
 has been studied extensively and is well described in several textbooks such as \cite{C81,CLW94,GH83,W03}.

We also assume that $g:\Rset^2\times\Rset\to\Rset^2$ is $C^2$,
 and $g(x,\phi)$ is $2\pi$-periodic in $\phi$ and has mean zero in $\phi$ for any $x\in\Rset^2$, i.e.,
\begin{align}
\int_0^{2\pi} g(x,\phi)\,\d\phi=0.
\label{eqn:mean}
\end{align}
We are interested in the case in which the frequency $\omega$ is sufficiently small
 such that $\omega=O(\epsilon^{1/4})$.
We transform \eqref{eqn:sys} to a simpler system
 which is a periodic perturbation of the normal form
 for codimension-two bifurcations with a double zero eigenvalue and symmetry,
 and apply the subharmonic and homoclinic Melnikov methods \cite{GH83,W03,Y96}
 to analyze bifurcations occurring in \eqref{eqn:sys}.
About the subharmonic Melnikov method,
 we use the version of \cite{Y96},
 which was further sophisticated from the classical theory of \cite{GH83}
 and especially enables us to easily determine the stability of detected subharmonic orbits.
The reader should consult \cite{Y96} for its details
 if he or she is not familiar with this version of the theory.
Among the  theoretical results presented here,
 we show that there exist transverse homoclinic or heteroclinic orbits
 to periodic orbits, which yield chaotic dynamics, in wide parameter regions.
Moreover, we show that many subharmonic orbits can be born at saddle-node bifurcations
 due to the periodic perturbations.
We remark that when $\omega=O(1)$,
 such codimension-two bifurcations without symmetry,
 called the Bogdanov-Takens bifurcations, were studied in \cite{Y02}
 and transverse homoclinic orbits were proven to exist only
 in exponentially small parameter regions.

These results can be applied to three or higher-dimensional systems
 and even to infinite-dimensional systems
 with the assistance of center manifold reduction \cite{C81,GH83,HI11,K04,W03}
 and the invariant manifold theory \cite{BLZ98,BLZ99,E13,F71,F74,W94}.
We illustrate our theory for the three-dimensional system \eqref{eqn:ex}
 with the desired position given by
\begin{equation}
\theta_\d=\epsilon\beta\cos\omega t+\theta_0,\quad
\mbox{$\theta_0=0$ or $\pi$,}
\label{eqn:td}
\end{equation}
where $\beta>0$ is a constant.
Here the reader may think that the third equation of \eqref{eqn:sys} should also be changed to
\[
\dot{z}_3=-\alpha z_3+\gamma(\theta_\d-z_1)+\delta_1(\dot{\theta}_\d-z_2),
\]
but its right hand side is rewritten as
\begin{align*}
-\alpha z_3+\gamma(\epsilon\tilde{\beta}\cos(\omega t+\tilde{\phi})+\theta_0-z_1)-\delta_1z_2,
\end{align*}
where
\[
\tilde{\beta}=\sqrt{\beta^2+\omega^2\beta^2},\quad
\tilde{\phi}=\arctan\omega,
\]
so that we only obtain the same equation with different value of $\beta$
 after the time shift $t\to t+\tilde{\phi}/\omega$.
In addition, we give numerical computations by the computer tool \texttt{AUTO} \cite{DO12}
 to demonstrate the theoretical results.
Periodic and chaotic motions in a similar system
 with a negligibly small time constant of the servo motor,
 i.e., $\alpha,\gamma,\delta_1\gg 1$, were studied
 by using the subharmonic and homoclinic Melnikov methods
 as well as the averaging method \cite{SVM07,Y96b} in \cite{Y94,Y96a}.

The rest of this paper is organized as follows:
In Section~2, we transform \eqref{eqn:sys} to the simpler system
\begin{equation}
\dot{y}_1=y_2,\quad
\dot{y}_2={\nu}_1y_1+{\nu}_2y_2+s_1y_1^3+s_2y_1^2y_2
+\epsilon h(\bar{\omega} t),
\label{eqn:pnf}
\end{equation}
where $\nu_1,\nu_2$ represent new control parameters instead of $\mu$,
 $h(\phi)$ is $2\pi$-periodic and the higher-order terms are ignored.
See \eqref{eqn:nu} and \eqref{eqn:h} in Section~2
 for definitions of $\nu_1,\nu_2,\bar{\omega},s_j$, $j=1,2$, and $h(\phi)$.
In particular, $s_j=1$ or $-1$, $j=1,2$.
Using the subharmonic and homoclinic Melnikov methods,
 we analyze bifurcations occurring in \eqref{eqn:pnf} for $s_1=1$ and $s_1=-1$
 in Sections~3 and 4, respectively.
We discuss three or higher-dimensional symmetric systems
 in which the linear parts have a double zero eigenvalue,
 and describe a treatment to apply our theory
 with assistance of center manifold reduction and the invariant manifold theory in Section~5.
Finally, we apply our theory to the example \eqref{eqn:ex} with \eqref{eqn:td}
 and give numerical results to demonstrate the validity of the theoretical ones in Section~6.

\section{Periodic Perturbation of the Normal Form}

In this section we show
 that the system \eqref{eqn:sys} is transformed
 into the periodic perturbation \eqref{eqn:pnf} of the normal form
 when higher-order terms are ignored.
We carry out a transformation used
 to obtain the normal form of codimensio-two bifurcations
 with a double zero eigenvalue and symmetry for the autonomous part.

Let $x_j$ and $\mu_j$ be the $j$th elements of $x$ and $\mu$, respectively, for $j=1,2$ and let
\begin{align*}
&
a_{jkl}=\frac{\partial ^3f_j}{\partial x_1^k\partial x_2^l}(0;0),\quad
l=3-k,\quad
k=0,1,2,3,\\
&{b}_{jkl}=\frac{\partial ^2 f_j}{\partial x_k\partial\mu_l}(0;0),\quad
k,l=1,2,
\end{align*}
so that
\begin{align}
f(x;\mu)=Jx+&
\begin{pmatrix}
\frac{1}{6}a_{130}x_1^3+\frac{1}{2}a_{121}x_1^2x_2+\frac{1}{2}a_{112}x_1x_2^2
+\frac{1}{6}a_{103}x_2^3\\[1ex]
\frac{1}{6}a_{230}x_1^3+\frac{1}{2}a_{221}x_1^2x_2+\frac{1}{2}a_{212}x_1x_2^2
+\frac{1}{6}a_{203}x_2^3\end{pmatrix}\notag\\
&+
\begin{pmatrix}
(b_{111}\mu_1+b_{112}\mu_2)x_1+(b_{121}\mu_1+b_{122}\mu_2)x_2\\
(b_{211}\mu_1+b_{212}\mu_2)x_1+(b_{221}\mu_1+b_{222}\mu_2)x_2
\end{pmatrix}+\hot,
\label{eqn:f}
\end{align}
where `$\hot$' represents higher-order terms than $O(|x|^3)$ and $O(|\mu|\,|x|)$.
Note that by \eqref{eqn:sym} and \eqref{eqn:f0}
 no constant and even-order terms with respect to $x$ appear in \eqref{eqn:f}.
Assume that
\begin{align}
c:=\tfrac{1}{6}a_{230}\neq 0,\quad
d:=\tfrac{1}{2}(a_{221}+a_{130})\neq 0.
\label{eqn:cd}
\end{align}
Using the transformation
\[
x=\xi+
\begin{pmatrix}
(\tfrac{1}{6}a_{121}+\tfrac{1}{12}a_{212})\xi_1^3
 +(\tfrac{1}{4}a_{112}+\tfrac{1}{12}a_{203})\xi_1^2\xi_2\\[1ex]
-\tfrac{1}{6}a_{130}\xi_1^3+\tfrac{1}{4}a_{212}\xi_1^2\xi_2
 +\tfrac{1}{6}a_{203}\xi_1\xi_2^2-\tfrac{1}{6}a_{103}\xi_2^3
\end{pmatrix},\quad
\xi=(\xi_1,\xi_2)\in\Rset^2,
\]
in \eqref{eqn:sys}, we have
\begin{align}
\dot{\xi}=&J\xi+
\begin{pmatrix}
0\\
c\xi_1^3+d\xi_1^2\xi_2
\end{pmatrix}\notag\\
&+
\begin{pmatrix}
(b_{111}\mu_1+b_{112}\mu_2)\xi_1+(b_{121}\mu_1+b_{122}\mu_2)\xi_2\\
(b_{211}\mu_1+b_{212}\mu_2)\xi_1+(b_{221}\mu_1+b_{222}\mu_2)\xi_2
\end{pmatrix}
+\epsilon g_0(\omega t)+\hot,
\label{eqn:sys1}
\end{align}
where `$\hot$' represents higher-order terms than $O(|\xi|^3)$, $O(|\mu|\,|\xi|)$ and $O(\epsilon)$,
 and $g_0(\phi)=g(0,\phi)$.
Moreover, we carry out the transformation
\begin{align*}
&
\xi=\eta-
\begin{pmatrix}
0\\
(b_{111}\mu_1+b_{112}\mu_2)\eta_1+(b_{121}\mu_1+b_{122}\mu_2)\eta_2
\end{pmatrix},\quad
\eta=(\eta_1,\eta_2)\in\Rset^2,
\end{align*}
in \eqref{eqn:sys1} to obtain
\begin{equation}
\begin{split}
\dot{\eta}_1=&\eta_2+\epsilon g_{01}(\omega t)+\hot,\\
\dot{\eta}_2=&\tilde{\mu}_1\eta_1+\tilde{\mu}_2\eta_2+c\eta_1^3+d\eta_1^2\eta_2
 +\epsilon g_{02}(\omega t)+\hot,
\end{split}
\label{eqn:sys2}
\end{equation}
where $g_{0j}(\phi)$ is the $j$th component of $g_0(\phi)$ and
\[
\tilde{\mu}_1=b_{211}\mu_1+b_{212}\mu_2,\quad
\tilde{\mu}_2=(b_{111}+b_{221})\mu_1+(b_{112}+b_{222})\mu_2.
\]
Finally, changing the time variable as $t\to|c/d|\,t$ and letting
\begin{equation}\label{TransEta}
y_1=\frac{|d|}{\sqrt{|c|}}\eta_1,\quad
y_2=\frac{|d|^2}{|c|^{3/2}}(\eta_2+\epsilon g_{01}(\omega t))+\hot
\end{equation}
in \eqref{eqn:sys2}, we obtain
\[
\dot{y}_1=y_2,\quad
\dot{y}_2=\nu_1 y_1+\nu_2 y_2+s_1y_1^3+s_2y_1^2y_2+\epsilon h(\bar{\omega}t)+\hot,
\]
where
\begin{equation}
\begin{split}
&
\nu_1=\left|\frac{d}{c}\right|^2\tilde{\mu}_1
 =\left|\frac{d}{c}\right|^2(b_{211}\mu_1+b_{212}\mu_2),\\
&
\nu_2=\left|\frac{d}{c}\right|\tilde{\mu}_2
 =\left|\frac{d}{c}\right|((b_{111}+b_{221})\mu_1+(b_{112}+b_{222})\mu_2),\\
&
s_1=\sign c,\quad
s_2=\sign d,\quad
\bar{\omega}=\left|\frac{d}{c}\right|\omega
\end{split}
\label{eqn:nu}
\end{equation}
and
\begin{equation}
h(\phi)=\frac{|d|^3}{|c|^{5/2}}g_{02}(\phi). 
\label{eqn:h}
\end{equation}
Here we have assumed that $\omega=o(1)$.
We obtain the following proposition.

\begin{prop}
\label{prop:2a}
Assume that \eqref{eqn:cd} holds.
Then the sytem \eqref{eqn:sys} is transformed into \eqref{eqn:pnf}
 when the higher-order terms are ignored.
Moreover,
\begin{equation}
\int_0^{2\pi}h(\phi)\d\phi=0.
\label{eqn:prop2a}
\end{equation}
\end{prop}

\begin{proof}
It remains to prove \eqref{eqn:prop2a} but it is obvious from \eqref{eqn:mean}.
\end{proof}

\begin{figure}[t]
\includegraphics[width=50mm]{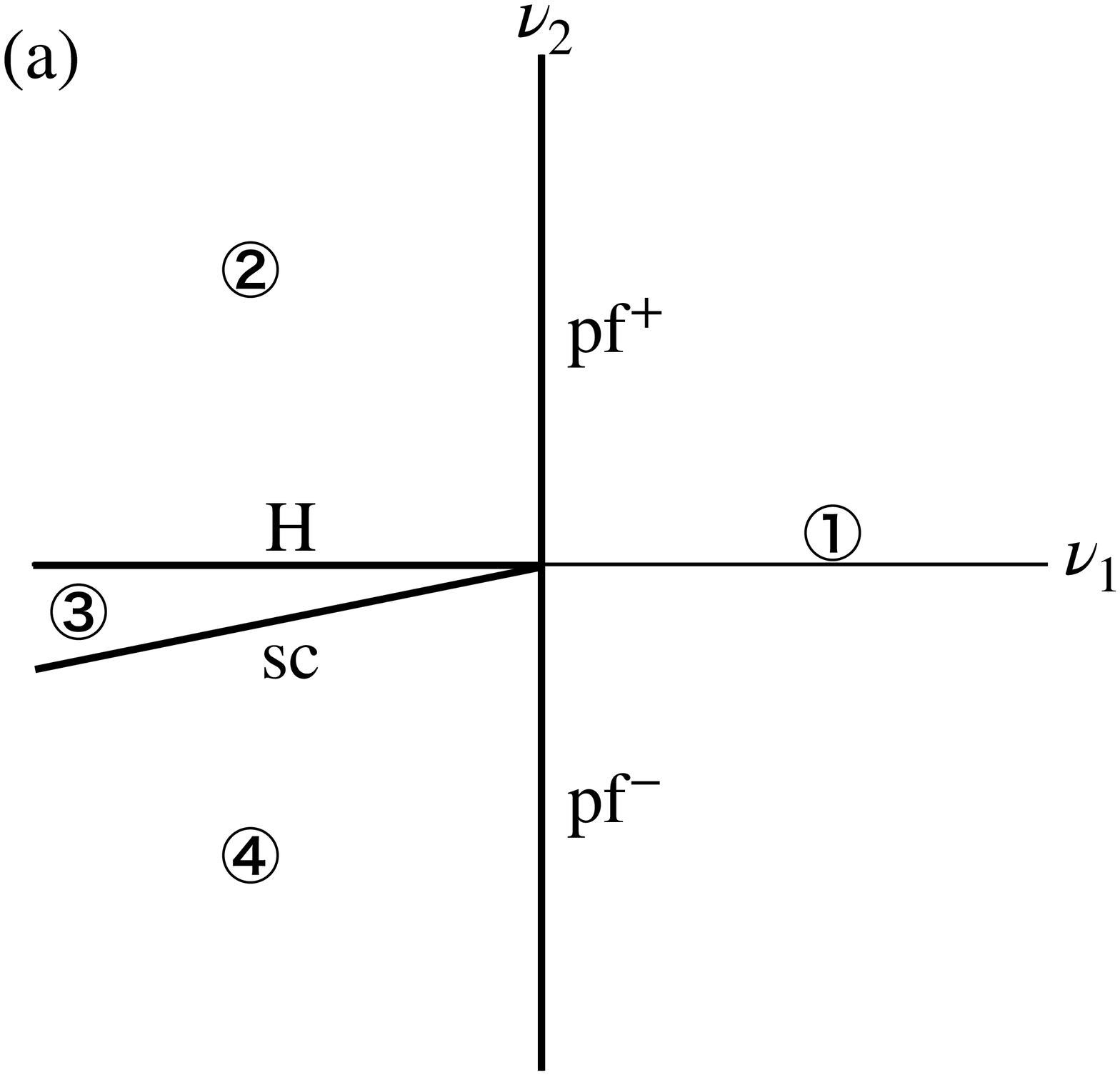}\\[2ex]
\includegraphics[width=100mm]{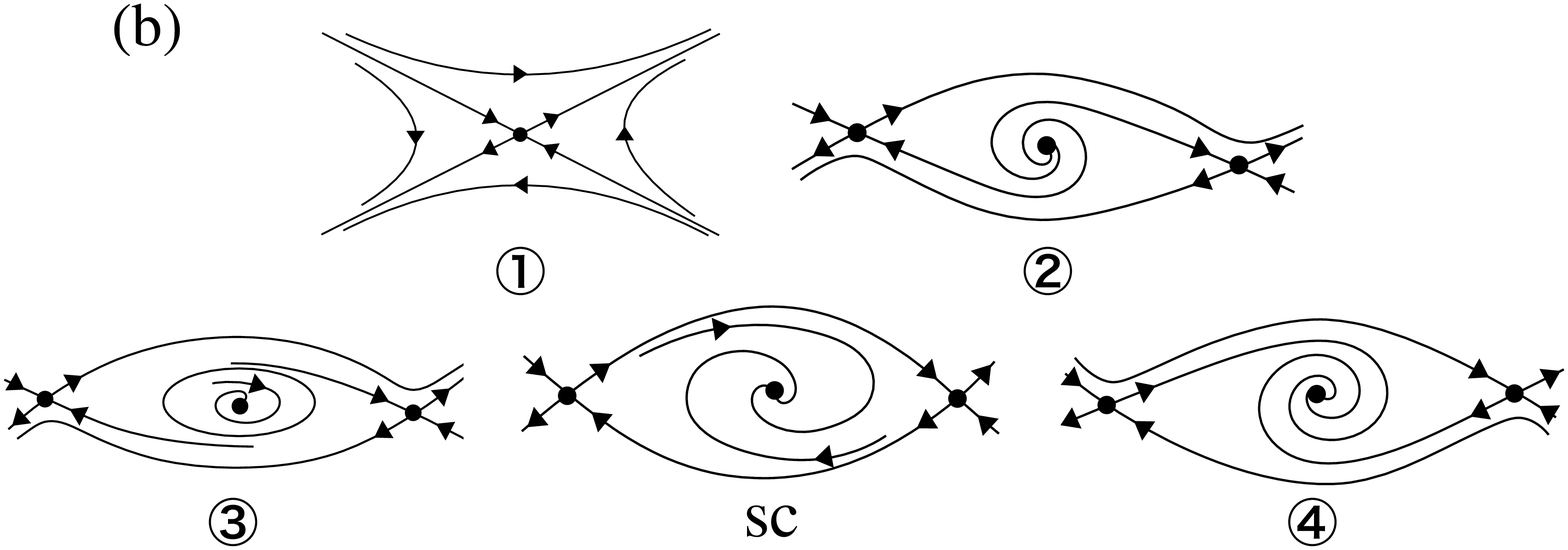}
\caption{Bifurcations of \eqref{eqn:pnf} with $s_1=s_2=1$ for $\epsilon=0$:
(a) Bifurcation sets; (b) Phase portraits.
$\mathrm{pf}^{\pm}$, H and sc
 represent pitchfork, Hopf and heteroclinic bifurcations, respectively.}
\label{fig:2a}
\end{figure}

\begin{figure}[t]
\includegraphics[width=50mm]{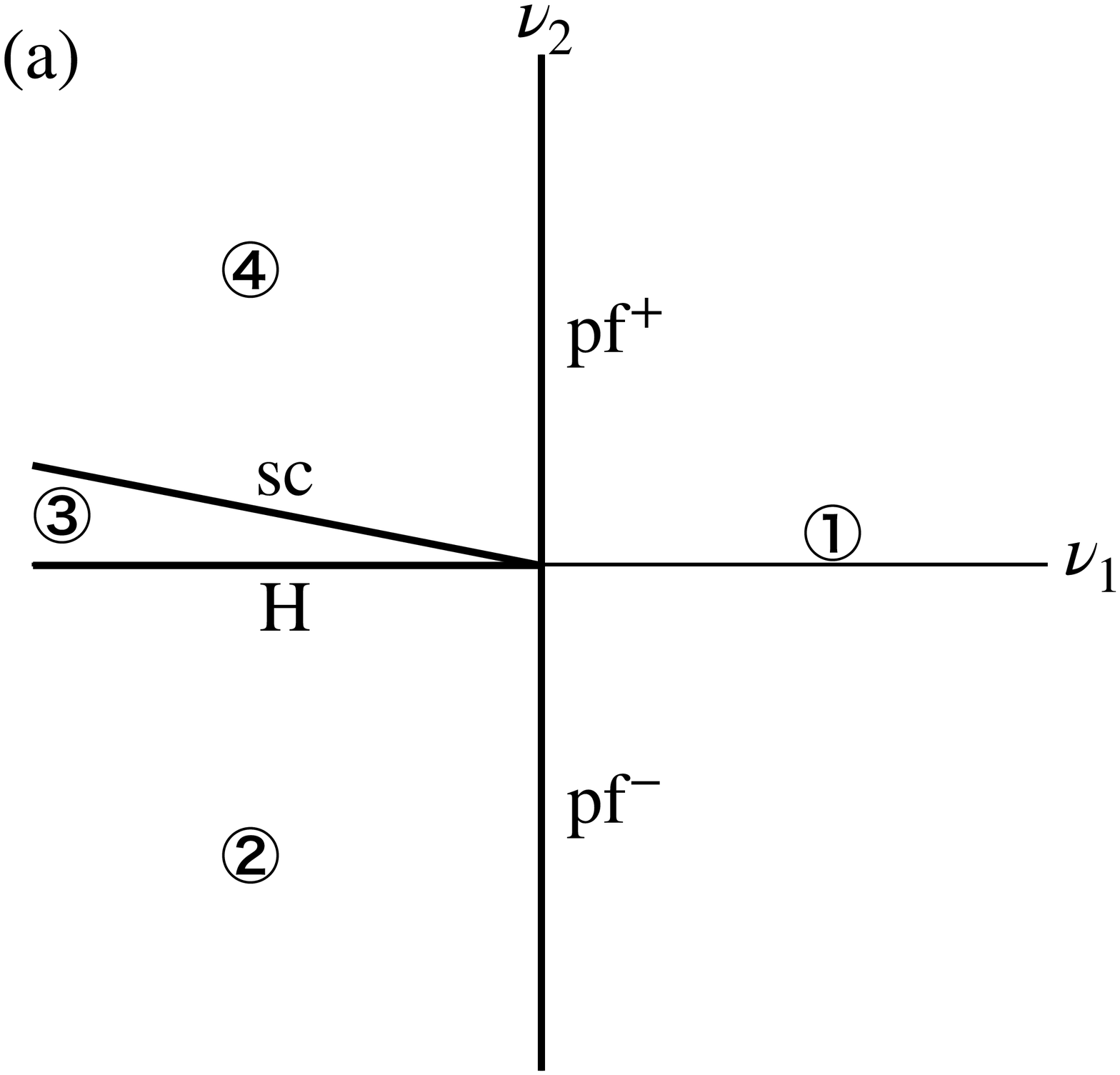}\\[2ex]
\includegraphics[width=100mm]{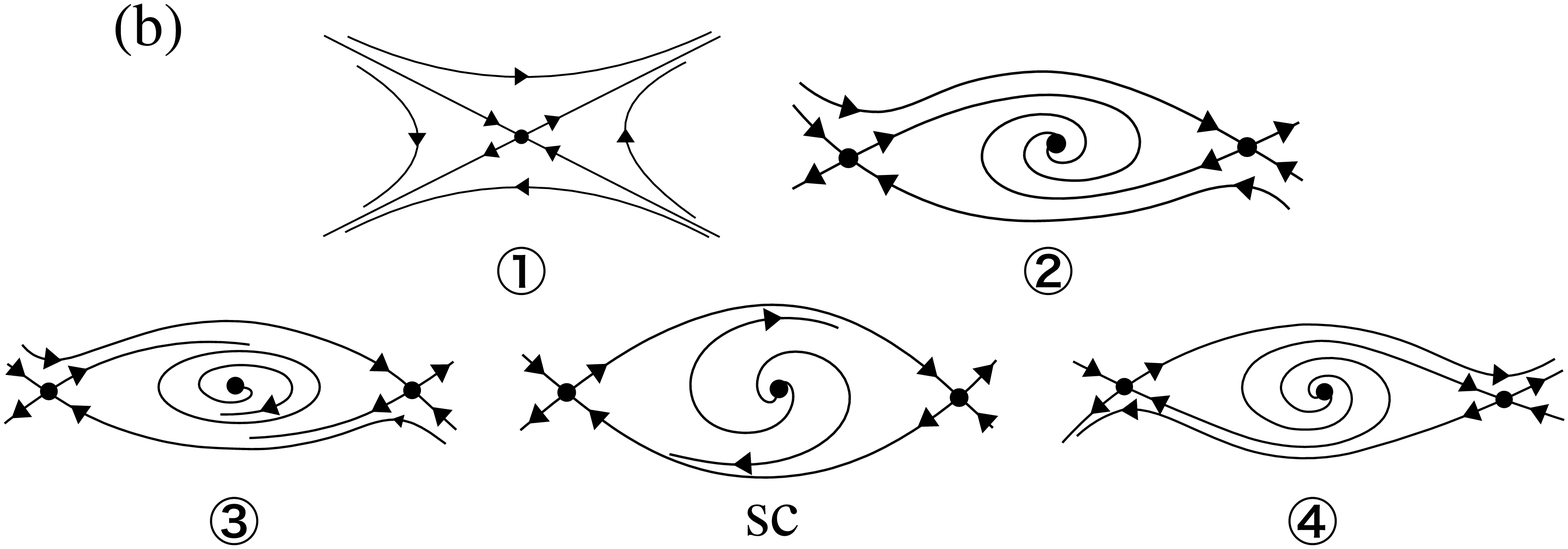}
\caption{Bifurcations of \eqref{eqn:pnf} with $s_1=1$ and $s_2=-1$ for $\epsilon=0$:
(a) Bifurcation sets; (b) Phase portraits.
See the caption of Fig.~\ref{fig:2a} for the meaning of the symbols.}
\label{fig:2b}
\end{figure}

\begin{figure}[t]
\begin{minipage}{45mm}
\includegraphics[width=50mm]{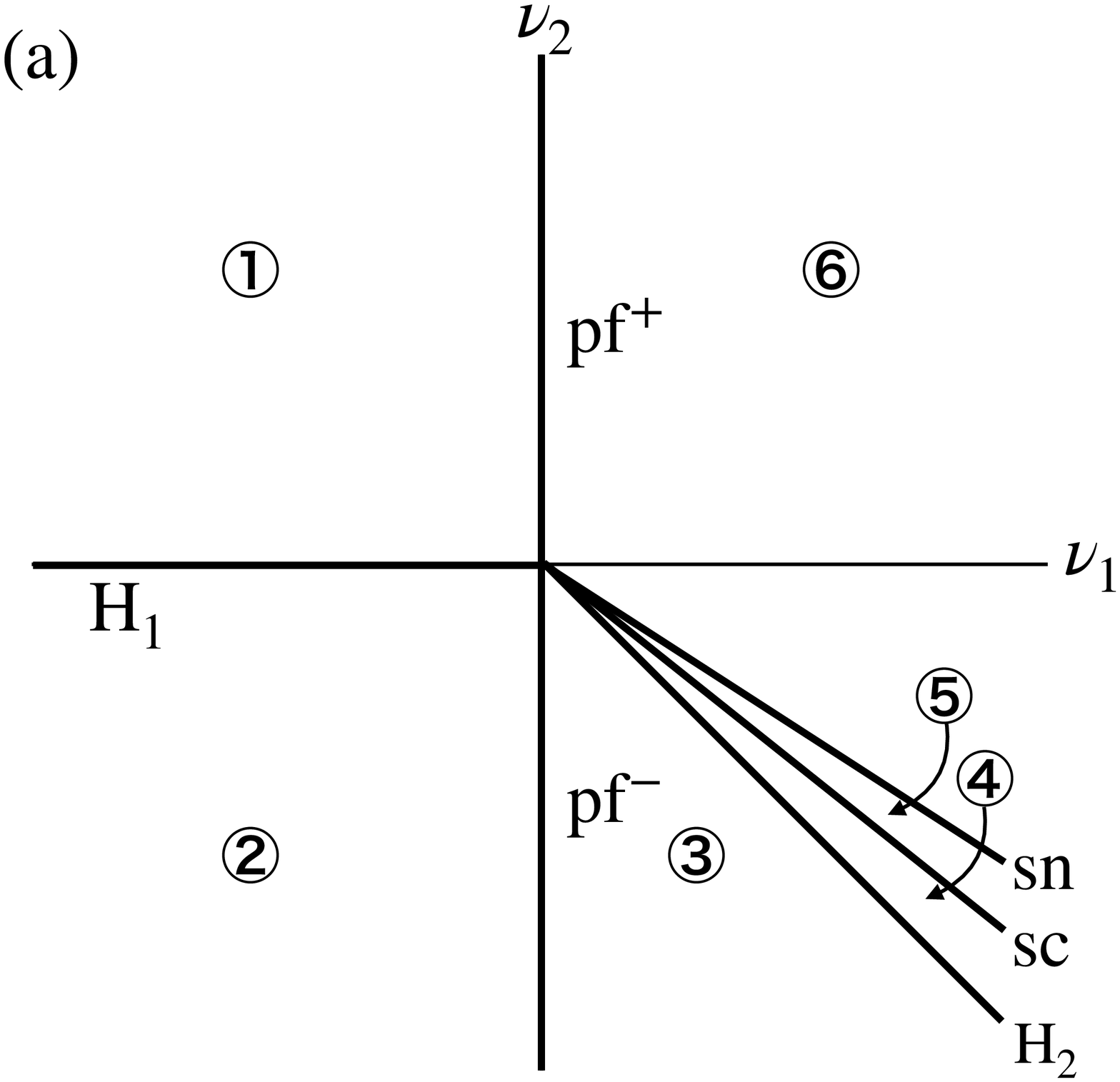}\\[2ex]
\end{minipage}
\begin{minipage}{100mm}
\includegraphics[width=100mm]{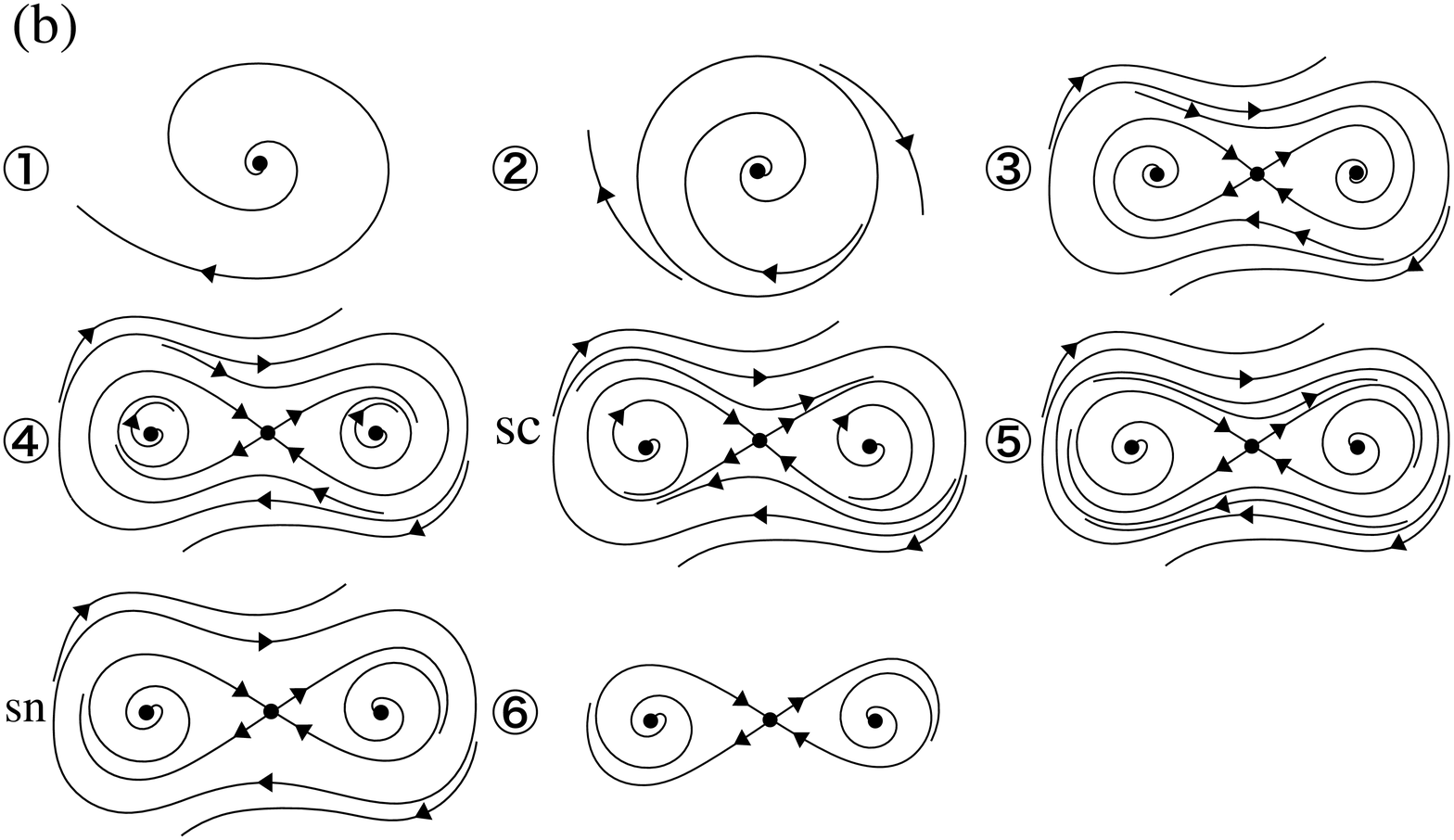}
\end{minipage}
\caption{Bifurcations of \eqref{eqn:pnf} with $s_1=-1$ and $s_2=1$ for $\epsilon=0$:
(a) Bifurcation sets; (b) Phase portraits.
$\mathrm{pf}^{\pm}$, $\mathrm{H}_j$, $j=1,2$, and sc
 represent pitchfork, Hopf and homoclinic bifurcations, respectively,
 while sn represents saddle-node bifurcations of periodic orbits.}
\label{fig:2c}
\end{figure}

We display bifurcation sets and phase portraits of \eqref{eqn:pnf}
 with $\epsilon=0$ and $s_1=1$ (resp. $s_1=-1$) for $s_2=1$ and $-1$, respectively,
 in Figs.~\ref{fig:2a} and \ref{fig:2b} (resp. in Figs.~\ref{fig:2c} and \ref{fig:2d}).
For $s_1=1$, when $s_2=1$ (resp. $s_2=-1$),
(i) a pitchfork bifurcation occurs on $\nu_1=0$ and $\nu_2\neq 0$;
(ii) a Hopf bifurcation occurs on $\nu_2=0$ and $\nu_1<0$;
(iii) a (double) heteroclinic bifurcation (saddle connection) occurs
 near $\nu_2=\nu_1/5<0$ (resp. near $\nu_2=-\nu_1/5>0$).
In region \textcircled{\scriptsize 1}
 there is a saddle;
in region \textcircled{\scriptsize 2}
 there are one source (resp. sink) and two saddles;
in region \textcircled{\scriptsize 3}
 there are one sink (resp. source), two saddles and one unstable (resp. stable) periodic orbit;
in region \textcircled{\scriptsize 4}
 there are one sink (resp. source) and two saddles.
On the curve sc, there are a pair of heteroclinic orbits.

\begin{figure}[t]
\begin{minipage}{45mm}
\includegraphics[width=50mm]{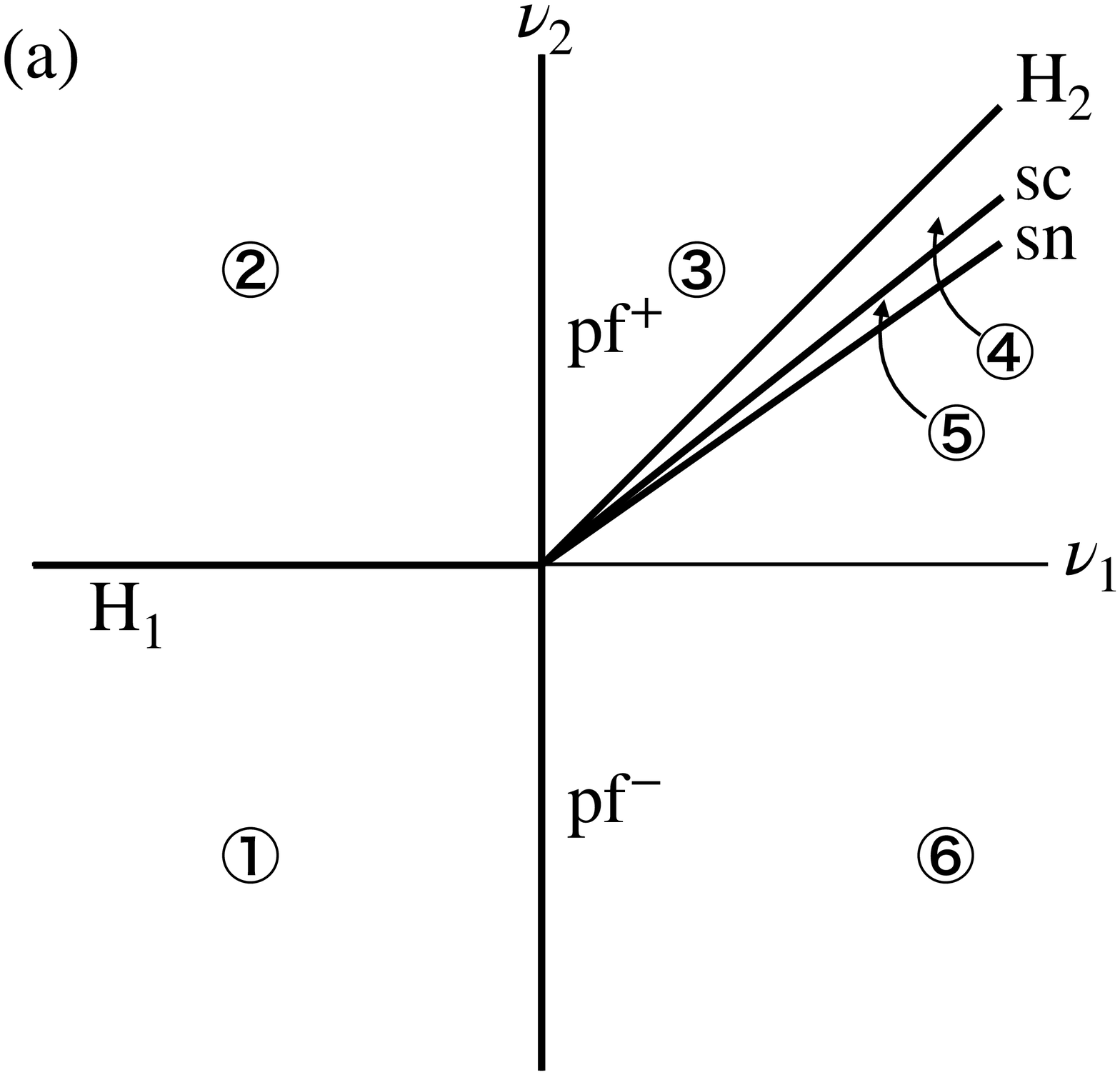}\\[2ex]
\end{minipage}
\begin{minipage}{100mm}
\includegraphics[width=100mm]{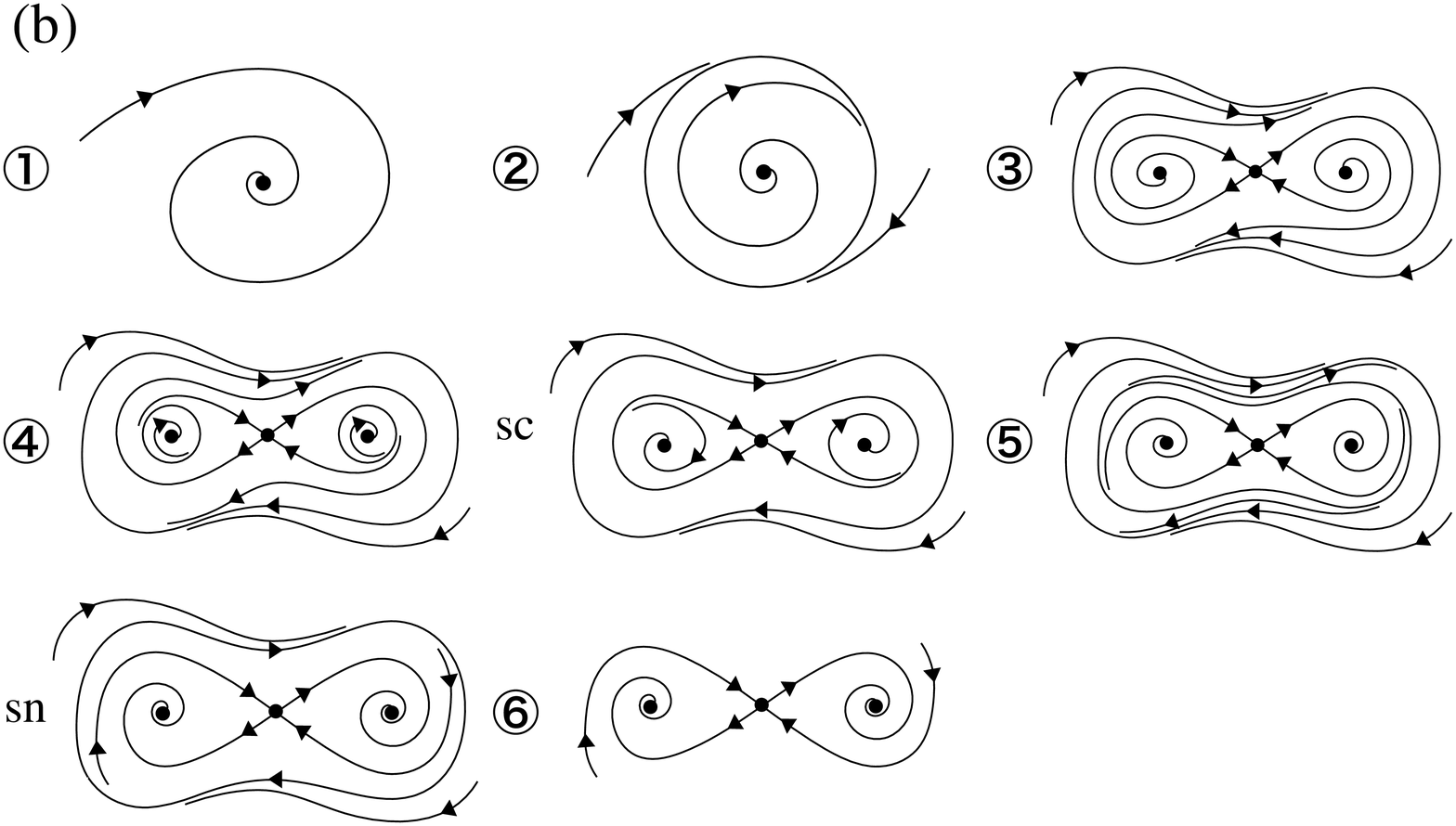}
\end{minipage}
\caption{Bifurcations of \eqref{eqn:pnf} with $s_1=s_2=-1$ for $\epsilon=0$:
(a) Bifurcation sets; (b) Phase portraits.
See the caption of Fig.~\ref{fig:2c} for the meaning of the symbols.}
\label{fig:2d}
\end{figure}

On the other hand, for $s_1=-1$, when $s_2=1$ (resp. $s_2=-1$),
(i) a pitchfork bifurcation occurs on $\nu_1=0$ and $\nu_2\neq 0$;
(ii) a Hopf bifurcation occurs on $\nu_2=0$ and $\nu_1<0$;
(iii) a (double) Hopf bifurcation occurs at $\nu_2=-\nu_1<0$ (resp. $\nu_2=\nu_1>0$);
(iv) a (double) homoclinic bifurcation occurs near $\nu_2=-4\nu_1 /5<0$
 (resp. near $\nu_2=4\nu_1 /5>0$);
(v) a saddle-node bifurcation of periodic orbits occurs
 near $\nu_2=-c\nu_1<0$ (resp. near $\nu_2=c\nu_1>0$) with $c\approx 0.752$.
In region \textcircled{\scriptsize 1}
 there is a source (resp. sink);
in region \textcircled{\scriptsize 2}
 there are a sink (resp. source) and unstable (resp. stable) periodic orbit;
in region \textcircled{\scriptsize 3}
 there are one saddle, two sinks (resp. sources) and one unstable (resp. stable) periodic orbit;
in region \textcircled{\scriptsize 4}
 there are one saddle, two sources (resp. sinks), and one stable and unstable periodic orbits;
in region \textcircled{\scriptsize 5}
 there are one saddle, two sources (resp. sinks) and one unstable (resp. stable) periodic orbit;
in region \textcircled{\scriptsize 6}
 there are one saddle and two sources (resp. sinks).
On the curve sc, there are a pair of homoclinic orbits.
See \cite{C81,GH83} for more details.

\section{Melnikov Analyses for $s_1=1$}\label{section3}

In this and the next sections,
 using the subharmonic and homoclinic Melnikov methods \cite{GH83,W03,Y96},
 we analyze the periodically perturbed system \eqref{eqn:pnf}
 to describe bifurcations occurring in \eqref{eqn:sys} near $(x,\mu)=(0,0)$.
We begin with the case of $s_1=1$.

\subsection{Preliminaries}
Assume that $|\nu_1|,|\nu_2|\ll 1$, $\nu_1<0$, $\nu_2=O(\nu_1)$ and $\epsilon=O(\nu_1^2)$.
Let $\hat{\epsilon}$ be a small positive parameter such that $\nu_1=-\hat{\epsilon}^2$,
 and introduce the new state variables $\zeta=(\zeta_1,\zeta_2)$
 and the new parameter $\hat{\nu}=O(1)$ as
\begin{align}
y_1=\sqrt{-\nu_1}\zeta_1,\quad
y_2=-\nu_1\zeta_2,\quad
\hat{\nu}=-\nu_2/\nu_1.
\label{eqn:sc1}
\end{align}
Rescaling the time variable $t$ as $t\to\sqrt{-\nu_1}\,t$,
 we rewrite \eqref{eqn:pnf} as
\begin{align}
\dot{\zeta}_1=\zeta_2,\quad
\dot{\zeta}_2=-\zeta_1+\zeta_1^3
 +\hat{\epsilon}(\hat{\nu}\zeta_2+s_2\zeta_1^2\zeta_2+\Delta h(\hat{\omega}t)),
\label{eqn:psys1}
\end{align}
where
\begin{equation}
\hat{\omega}
 =\frac{\bar{\omega}}{\hat{\epsilon}}=\left|\frac{d}{c}\right|\frac{\omega}{\hat{\epsilon}}=O(1),\quad
\Delta=\frac{\epsilon}{\hat{\epsilon}^4}=O(1).
\label{eqn:hat}
\end{equation}

When $\hat{\epsilon}=0$, the system~\eqref{eqn:psys1} becomes
\begin{equation}
\dot{\zeta}_1=\zeta_2,\quad
\dot{\zeta}_2=-\zeta_1+\zeta_1^3,
\label{eqn:psys10}
\end{equation}
which is a planar Hamiltonian system with the Hamiltonian
\[
H(\zeta)=\frac{1}{2}\zeta_2^2+\frac{1}{2}\zeta_1^2-\frac{1}{4}\zeta_1^4.
\]
The unperturbed system \eqref{eqn:psys10} has hyperbolic saddles at $\zeta=(\pm 1,0)$
 to which there exists a pair of heteroclinic orbits
\begin{align}
\zeta_\pm^{\h}(t)=
\left(\pm\tanh\frac{t}{\sqrt{2}},
\pm\frac{1}{\sqrt{2}}\sech^2\frac{t}{\sqrt{2}} \right).
\label{eqn:ho1}
\end{align}
Moreover, inside the pair of heteroclinic orbits,
 there exists a one-parameter family of periodic orbits,
\begin{align}
\zeta^k(t)=&\left(\frac{\sqrt{2}\,k}{\sqrt{k^2+1}}\sn\left(\frac{t}{\sqrt{k^2+1}}
\right)\right.,\notag\\
&\quad\left.\frac{\sqrt{2}\,k}{k^2+1}\cn\left(\frac{t}{\sqrt{k^2+1}}\right)
\dn\left(\frac{t}{\sqrt{k^2+1}}\right)\right),\quad k\in(0,1),
\label{eqn:po1}
\end{align}
where $\sn,\cn,\dn$ represent Jacobi's elliptic functions with elliptic modulus $k$.
The period of $\zeta^k(t)$ is given by
\begin{align*}
T^k=4K(k)\sqrt{k^2+1},
\end{align*}
where $K(k)$ is the complete elliptic integral of the first kind,
\[
K(k)=\int_0^{\pi/2}\frac{\d\psi}{\sqrt{1-k^2\sin^2\psi}}.
\]
See \cite{BF54} for necessary information on elliptic functions and complete elliptic integrals.
The phase portrait of \eqref{eqn:psys10} is displayed in Fig.~\ref{fig:3a}.

\begin{figure}[t]
\begin{center}
\includegraphics[scale=0.6]{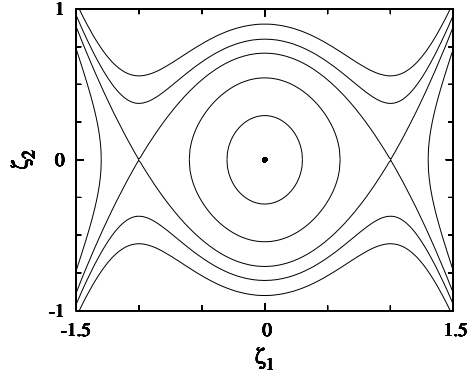}
\caption{Phase portraits of \eqref{eqn:psys10}.}
\label{fig:3a}
\end{center}
\end{figure}

\subsection{Heteroclinic orbits}
We apply the homoclinic Melnikov methods \cite{GH83,W03}.
For $\hat{\epsilon}>0$ sufficiently small
 we easily see that the system~\eqref{eqn:psys1} has hyperbolic periodic orbits,
 of which the stable and unstable manifolds may intersect, near $\zeta=(\pm 1,0)$.
See \cite{GH83,W03}.
We compute the Melnikov functions for $\zeta_\pm^{\h}(t)$ as
\begin{align*}
M_\pm(\phi;\hat{\nu})
=& \int_{-\infty}^\infty
 \zeta_{2\pm}^{\h}(t)\left[\hat{\nu}\zeta_{2\pm}^{\h}(t)
 +s_2\zeta_{1\pm}^{\h}(t)^2\zeta_{2\pm}^{\h}(t)
 +\Delta h(\hat{\omega} t+\phi)\right]\d t\notag\\
=&\frac{4}{3\sqrt{2}}\hat{\nu}+\frac{2\sqrt{2}}{15}s_2\pm\Delta\hat{h}(\phi),
\end{align*}
where
\begin{equation}
\hat{h}(\phi)=\int_{-\infty}^\infty\zeta_{2+}^{\h}(t)h(\hat{\omega} t+\phi)\d t.
\label{eqn:hath}
\end{equation}
We have the following property for the function $\hat{h}(\phi)$.

\begin{lem}
\label{lem:3a}
If $h(\phi)\not\equiv 0$, then
 \begin{equation}
\hat{h}_{\mathrm{max}}=\max_{\phi\in\Sset^1}\hat{h}(\phi)>0
\quad\mbox{and}\quad
\hat{h}_{\mathrm{min}}=\min_{\phi\in\Sset^1}\hat{h}(\phi)<0,
\label{eqn:lem3a}
\end{equation}
where $\Sset^1$ is the circle of length $2\pi$.
\end{lem}

\begin{proof}
Noting that it is $C^2$,
 we expand $h(\phi)$ in a Fourie series as
\[
h(\phi)=\sum_{j=-\infty}^\infty h_{j}e^{i j\phi},\quad
h_j=\frac{1}{2\pi}\int_0^{2\pi}h(\phi)e^{-ij\phi}\d\phi,
\]
so that
\[
\hat{h}(\phi)
 =\sum_{j=-\infty}^\infty h_j\tilde{\zeta}_{2+}(j\hat{\omega})e^{i j\phi},
\]
where
\begin{equation}
\tilde{\zeta}_{2+}(\chi)
 =\int_{-\infty}^\infty\zeta_{2+}^{\h}(t)e^{i\chi t}\d t
 =
\begin{cases}
\sqrt{2}\pi\chi\csch(\sqrt{2}\pi\chi/2) & \mbox{for $\chi\neq 0$;}\\
2 & \mbox{for $\chi=0$,}
\end{cases}
\label{eqn:lem3ap}
\end{equation}
which is not zero for any $\chi\in\Rset$.
Here we have used the fact that
 the integral in \eqref{eqn:lem3ap} is uniformly convergent in $j\in\Zset$.
Hence, if $h(\phi)\not\equiv 0$, then $\hat{h}_j\neq 0$ for some $j\in\Zset$ at least
 and consequently $\hat{h}(\phi)\not\equiv 0$.

On the other hand, we easily see that
 the integral in \eqref{eqn:hath} is uniformly convergent in $\phi$
 since the integrand exponentially tends to $0$.
Hence,
\begin{align*}
\int_0^{2\pi}\hat{h}(\phi)\d \phi
=&\int_0^{2\pi}\left[\int_{-\infty}^\infty
 \zeta_{2\pm}^{\h}(t)h(\hat{\omega}t+\phi)\d t \right]\d \phi\\
=&\int_{-\infty}^\infty\zeta_{2\pm}^{\h}(t)
 \left[ \int_{0}^{2\pi} h(\hat{\omega}t+\phi)\d \phi \right]\d t=0,
\end{align*}
which yields \eqref{eqn:lem3a}.
\end{proof}

Assume that $h(\phi)\not\equiv 0$.
If
\[
-\Delta\hat{h}_{\mathrm{max}}
<\frac{4}{3\sqrt{2}}\hat{\nu}+\frac{2\sqrt{2}}{15}s_2
<-\Delta\hat{h}_{\mathrm{min}}
\quad\mbox{and}\quad
\Delta\hat{h}_{\mathrm{min}}
<\frac{4}{3\sqrt{2}}\hat{\nu}+\frac{2\sqrt{2}}{15}s_2
<\Delta\hat{h}_{\mathrm{max}}
\]
i.e.,
\begin{equation}
\left(\frac{3\sqrt{2}}{4}\Delta\hat{h}_{\mathrm{max}}+\frac{1}{5}s_2\right)\nu_1
<\nu_2<\left(\frac{3\sqrt{2}}{4}\Delta\hat{h}_{\mathrm{min}}+\frac{1}{5}s_2\right)\nu_1,\quad
\nu_1<0,
\label{eqn:hcon1a}
\end{equation}
and
\begin{equation}
\left(-\frac{3\sqrt{2}}{4}\Delta\hat{h}_{\mathrm{min}}+\frac{1}{5}s_2\right)\nu_1
<\nu_2<\left(-\frac{3\sqrt{2}}{4}\Delta\hat{h}_{\mathrm{max}}+\frac{1}{5}s_2\right)\nu_1,\quad
\nu_1<0,
\label{eqn:hcon1b}
\end{equation}
then the Melnikov functions $M_\pm(\phi;\hat{\nu})$ cross the zero topologically.
Using arguments given in \cite{GH83,W03},
 we obtain the following result.

\begin{thm}
\label{thm:3a}
Suppose that $h(\phi)\not\equiv 0$.
If condition~\eqref{eqn:hcon1a} $($resp. \eqref{eqn:hcon1b}$)$ holds
 for $|\nu_1|,|\nu_2|,\epsilon>0$ sufficiently small,
 then the left branch of the stable manifold $($resp. of the unstable manifold$)$
 of a periodic orbit near $\zeta=(1,0)$
 intersect the right branch of the unstable manifold $($resp. of the stable manifold$)$
 of a periodic orbit near $\zeta=(-1,0)$ topologically transversely
 in \eqref{eqn:psys1} and hence in \eqref{eqn:sys}.
Moreover, 
 heteroclinic bifurcations occur near the four curves
\begin{align}
&
\nu_2=\left(\frac{3\sqrt{2}}{4}\Delta\hat{h}_{\mathrm{max}}+\frac{1}{5}s_2\right)\nu_1,\quad
\nu_2=\left(\frac{3\sqrt{2}}{4}\Delta\hat{h}_{\mathrm{min}}+\frac{1}{5}s_2\right)\nu_1,\quad
\nu_1<0,
\label{eqn:thm3a1}\\
&
\nu_2=\left(-\frac{3\sqrt{2}}{4}\Delta\hat{h}_{\mathrm{min}}+\frac{1}{5}s_2\right)\nu_1,\quad
\nu_2=\left(-\frac{3\sqrt{2}}{4}\Delta\hat{h}_{\mathrm{max}}+\frac{1}{5}s_2\right)\nu_1,\quad
\nu_1<0.
\label{eqn:thm3a2}
\end{align}
\end{thm}

\begin{rmk}
\label{rmk:3a}
In the overlap parameter region between \eqref{eqn:hcon1a} and \eqref{eqn:hcon1b},
 there exists a heteroclinic cycle and consequently topologically transverse homoclinic orbits
 to the periodic orbits near $\zeta=(\pm 1,0)$.
See Section~$26.1$ of {\rm\cite{W03}} for more details on the fact.
Such homoclinic orbits also yield chaotic motions in \eqref{eqn:psys1} and \eqref{eqn:sys}
 like geometrically transverse ones.
See {\rm\cite{BW95}} for more details.
\end{rmk}

\begin{figure}[t]
\includegraphics[width=45mm]{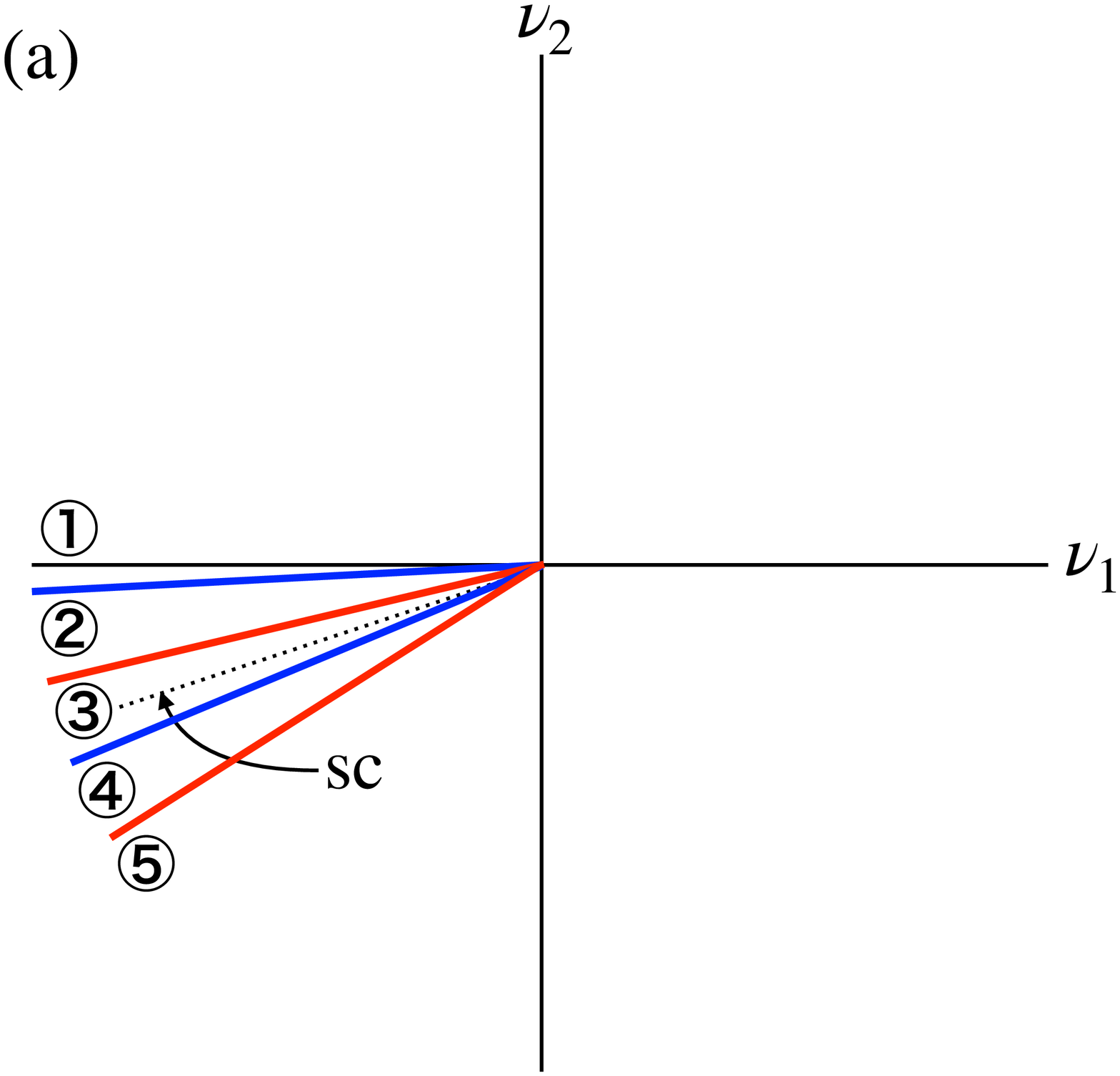}\quad
\includegraphics[width=45mm]{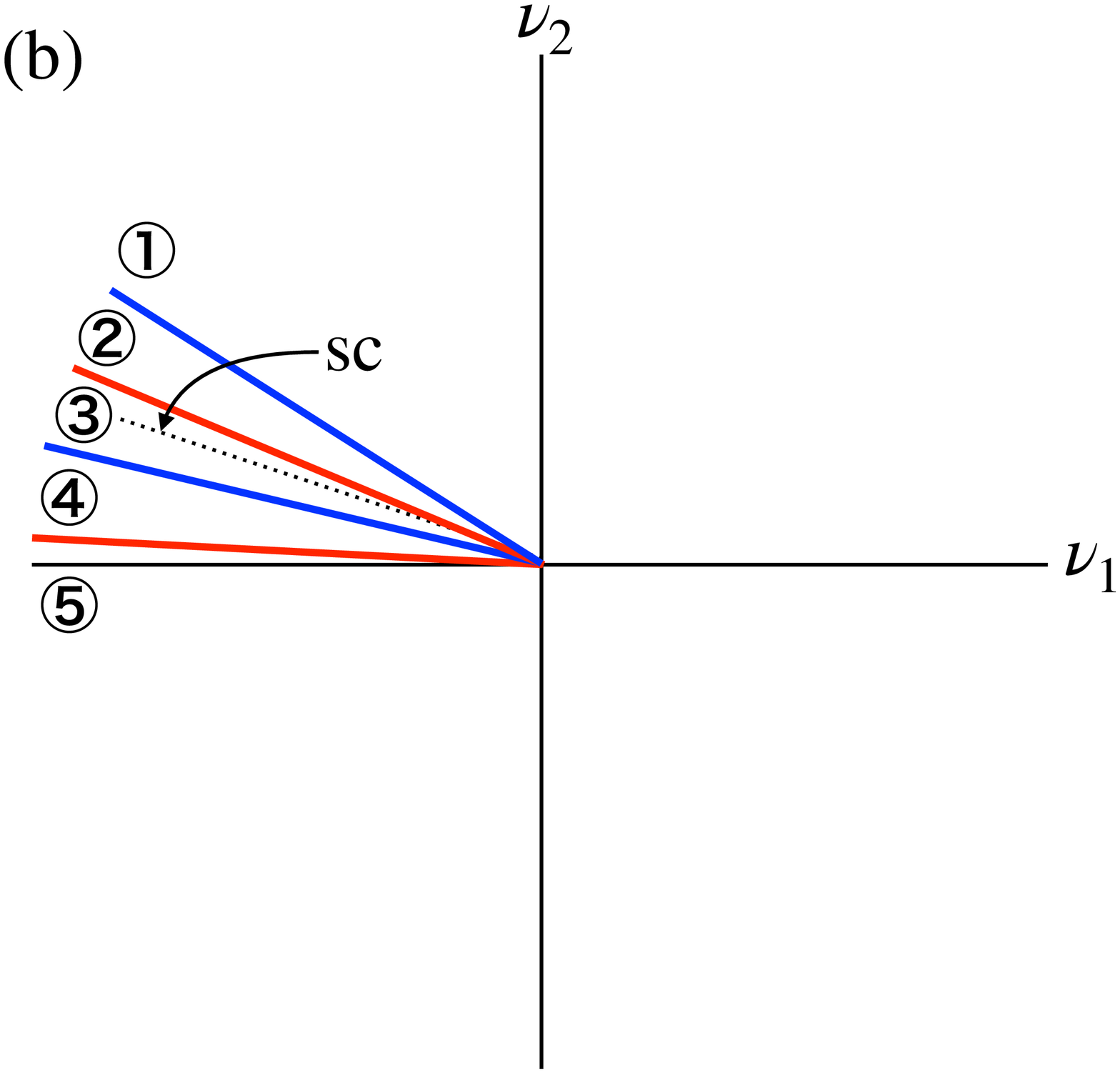}\\[2ex]
\includegraphics[width=100mm]{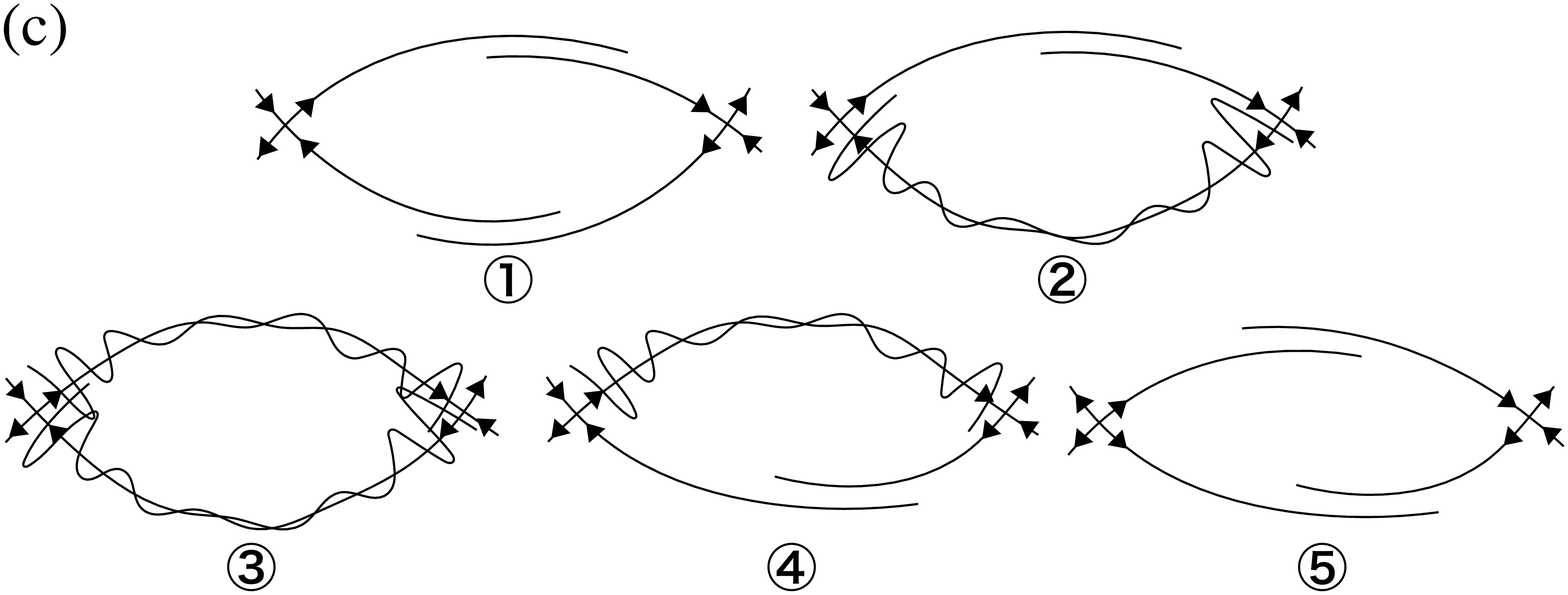}
\caption{Heteroclinic bifurcations in \eqref{eqn:psys1}:
(a) Approximate bifurcation sets for $s_2=1$; (b) for $s_2=-1$;
(c) stable and unstable manifolds for a Poincar\'e map.}
\label{fig:thm3a}
\end{figure}

Figures~\ref{fig:thm3a}(a) and (b)
 show the approximate heteroclinic bifurcation sets detected by Theorem~\ref{thm:3a}
 for $s_2=1$ and $-1$, respectively.
Here we have assumed that $|\hat{h}_{\mathrm{min}}|<|\hat{h}_{\mathrm{max}}|$.
The red and blue lines correspond
 to the bifurcation sets \eqref{eqn:thm3a1} and \eqref{eqn:thm3a2}, respectively.
Figure~\ref{fig:thm3a}(c) shows the behavior
 of the  stable and unstable manifolds of the periodic orbits
 near $\zeta=(\pm 1,0)$ for a Poincar\'e map
 whose orbit is given by sampling $(\zeta_1(t),\zeta_2(t))$ at every period,
 in regions \textcircled{\scriptsize 1}-\textcircled{\scriptsize 5} of both figures.
In particular, in region \textcircled{\scriptsize 3}, a heteroclinic cycle is constructed
 and consequently chaotic motions occur, as stated in Remark~\ref{rmk:3a}.
When $|\hat{h}_{\mathrm{min}}|>|\hat{h}_{\mathrm{max}}|$,
 the loci of the red and blue lines are exchanged in Figs.~\ref{fig:thm3a}(a) and (b)
 and when $|\hat{h}_{\mathrm{min}}|=|\hat{h}_{\mathrm{max}}|$,
 regions \textcircled{\scriptsize 2} and \textcircled{\scriptsize 4}
 disappear or their widths becomes $o(\sqrt{-\nu_1})$.

\subsection{Subharmonic orbits}
We next apply the subharmonic Melnikov method \cite{Y96}.
Since Eq.~\eqref{eqn:psys10} is a single-degree-of-freedom Hamiltonian and integrable,
 there exists a symplectic transformation from $\zeta$
 to the \emph{action-angle coordinates} $(I,\phi)$,
 e.g., as described in Chapter~10 of \cite{A89}.
Let $\Omega^k=2\pi/T^k$ and $I^k$, respectively,
 represent the angular frequency and action of  $\zeta^k(t)$, where
\begin{equation}
I^k=\frac{1}{2\pi}\int_0^{T^k}\zeta_2^k(t)^2\d t.
\label{eqn:Ik}
\end{equation}
We first compute the derivative $\d\Omega^k/\d I^k$,
 which is required in application of the theory.
Since the symplectic transformation is explicitly obtained in general,
 the reader may think that it is difficult,
 but it is not, as we see below.  

The Hamiltonian energy for $\zeta^k(t)$ is given by
\begin{align*}
H^k=\frac{k^2}{(k^2+1)^2},
\end{align*}
so that
\begin{align*}
\frac{\d H^k}{\d k}=\frac{2k(1-k^2)}{(k^2+1)^3}>0.
\end{align*}
Since
\[
\frac{\d H^k}{\d I^k}=\Omega^k=\frac{\pi}{2K(k)\sqrt{k^2+1}}>0
\]
and
\[
\frac{\d}{\d k}[(1-k^2)K(k)-(1+k^2)E(k)]=-3kE(k)<0,
\]
we have
\[
\frac{\d I^k}{\d k}=\frac{\d H^k/\d k}{\d H^k/\d I^k}>0
\]
and
\[
\frac{\d\Omega^k}{\d k}
 =\frac{\pi}{2k(1-k^2)(1+k^2)^{3/2}K(k)^2}[(1-k^2)K(k)-(1+k^2)E(k)]<0,
\]
where $E(k)$ is the complete elliptic integral of the second kind,
\[
E(k)=\int_0^{\pi/2}\sqrt{1-k^2\sin^2\psi}\,\d\psi.
\]
Here we have also used $K(0)=E(0)=\pi/2$ and
\[
\frac{\d K}{\d k}(k)=\frac{1}{k(1-k^2)}(E(k)-(1-k^2)K(k)),\quad
\frac{\d E}{\d k}(k)=\frac{1}{k}(E(k)-K(k)).
\]
Thus, we obtain
\begin{align*}
\frac{\d\Omega^k}{\d I^k}
 =\frac{\d\Omega^k/\d k}{\d I^k/\d k}<0.
\end{align*}

Let $nT^k=m\hat{T}$ for $m,n>0$ relatively prime integers,
 where $\hat{T}=|c/d|\hat{\epsilon}\,T$.
We compute the subharmonic Melnikov functions as
\begin{align*}
M^{m/n}(\phi;\hat{\nu})
=&\int_0^{m\hat{T}}\zeta_2^k(t)
 \left[\hat{\nu}\zeta_2^k(t)+s_2\zeta_1^k(t)^2\zeta_2^k(t)
 +\Delta h(\hat{\omega}t+\phi)\right]\d t\notag\\
=& \hat{\nu}J_1(k,n)+s_2J_2(k,n)+\Delta\hat{h}^{m/n}(\phi)
\end{align*}
and
\begin{align*}
L^{m/n}(\phi;\hat{\nu})
=& \int_0^{m\hat{T}}(\hat{\nu}+s_2\zeta_1^k(t)^2)\d t
=m\hat{\nu}\hat{T}+s_2 J_3(k,n),
\end{align*}
where
\begin{align*}
J_1(k,n)=&\int_0^{m\hat{T}}\zeta_2^k(t)^2\d t
=\frac{2k^2}{(k^2+1)^2}\int_0^{m\hat{T}}
\cn^2\left(\frac{t}{\sqrt{k^2+1}}\right)
\dn^2\left(\frac{t}{\sqrt{k^2+1}}\right)\d t\\
=&\frac{8n}{3(k^2+1)^{3/2}}\left[(k^2+1)E(k)-k'^2K(k)\right]>0,\\
J_2(k,n)=&\int_0^{m\hat{T}}\zeta_1^k(t)^2\zeta_2^k(t)^2\d t\\
=&\frac{4k^4}{(k^2+1)^3}\int_0^{m\hat{T}}
\sn^2\left(\frac{t}{\sqrt{k^2+1}}\right)
\cn^2\left(\frac{t}{\sqrt{k^2+1}}\right)
\dn^2\left(\frac{t}{\sqrt{k^2+1}}\right)\d t\\
=&\frac{16n}{15(k^2+1)^{5/2}}\left[2(k^4-k^2+1)E(k)-k'(2-k^2)K(k)\right]>0,\\
J_3(k,n)=&\int_0^{m\hat{T}}\zeta_1^k(t)^2\d t
=\frac{2k^2}{k^2+1}\int_0^{m\hat{T}}
\sn^2\left(\frac{t}{\sqrt{k^2+1}}\right)\d t\\
=& \frac{8n}{\sqrt{k^2+1}}\left[K(k)-E(k)\right]>0
\end{align*}
and
\begin{equation}
\hat{h}^{m/n}(\phi)=\int_0^{m\hat{T}}\zeta_2^k(t)h(\hat{\omega}t+\phi)\d t,
\label{eqn:hathmn}
\end{equation}
where $k'=\sqrt{1-k^2}$ is complementary elliptic modulus.

\begin{lem}
\label{lem:3b}
If $\hat{h}^{m/n}(\phi)\not\equiv 0$, then
 \begin{equation}
\hat{h}_{\mathrm{max}}^{m/n}=\max_{\phi\in\Sset^1}\hat{h}^{m/n}(\phi)>0
\quad\mbox{and}\quad
\hat{h}_{\mathrm{min}}^{m/n}=\min_{\phi\in\Sset^1}\hat{h}^{m/n}(\phi)<0.
\label{eqn:lem3b}
\end{equation}
\end{lem}

\begin{proof}
We easily see that
\begin{align*}
\int_0^{2\pi}\hat{h}^{m/n}(\phi)\d \phi
=&\int_0^{2\pi}\left[\int_0^{m\hat{T}}
 \zeta_2^k(t)h(\hat{\omega}t+\phi)\d t \right]\d \phi\\
=&\int_0^{m\hat{T}}\zeta_2^k(t)
 \left[ \int_{0}^{2\pi} h(\hat{\omega}t+\phi)\d \phi \right]\d t=0.
\end{align*}
Thus, if $\hat{h}^{m/n}(\phi)\not\equiv 0$, then Eq.~\eqref{eqn:lem3b} holds.
\end{proof}

\begin{thm}
\label{thm:3b}
Suppose that $\hat{h}^{m/n}(\phi)\not\equiv 0$
 and take $\nu_1$ or $\nu_2$ as a control parameter.
Then saddle-node bifurcations of $m\hat{T}$-periodic $($resp. $mT$-periodic$)$ orbits occur near
\begin{equation}
\nu_2=\frac{\Delta\hat{h}_{\mathrm{max}}^{m/n}+s_2J_2(k,n)}{J_1(k,n)}\nu_1
\quad\mbox{and}\quad
\nu_2=\frac{\Delta\hat{h}_{\mathrm{min}}^{m/n}+s_2J_2(k,n)}{J_1(k,n)}\nu_1,\quad
\nu_1<0
\label{eqn:thm3b}
\end{equation}
in \eqref{eqn:psys1} $($resp. in \eqref{eqn:sys}$)$.
\end{thm}

\begin{proof}
Suppose that
 the first (resp. second) condition in \eqref{eqn:thm3b} holds, i.e.,
\[
\hat{\nu}J_1(k,n)+s_2J_2(k,n)
=-\Delta\hat{h}_{\mathrm{max}}^{m/n}
\quad(\mbox{resp. }-\Delta\hat{h}_{\mathrm{min}}^{m/n}).
\]
We assume that $\hat{h}^{m/n}(\phi)$ attains the maximum (resp. minimum) at $\phi=\phi_0$,
 so that
\[
\frac{\d\hat{h}^{m/n}}{\d\phi}(\phi_0)=0.
\]
If in addition
\begin{equation}
\frac{\d^2\hat{h}^{m/n}}{\d\phi^2}(\phi_0)<0\quad\mbox{(resp. $>0$)},
\label{eqn:thm3bp}
\end{equation}
then the subharmonic Melnikov function $M^{m/n}(\phi;\hat{\nu})$ satisfies
\begin{itemize}
\setlength{\leftskip}{-1.2em}
\item[(i)]
$M^{m/n}(\phi_0;\hat{\nu})=0$;
\item[(ii)]
$\displaystyle\frac{\partial M^{m/n}}{\partial\phi}(\phi_0;\hat{\nu})=0$;
\item[(iii)]
$\displaystyle\frac{\partial^2M^{m/n}}{\partial\phi^2}(\phi_0;\hat{\nu})\neq 0$;
\item[(iv)]
$\displaystyle\frac{\partial M^{m/n}}{\partial\hat{\nu}}(\phi_0;\hat{\nu})\neq 0$,
\end{itemize}
which yields the conclusion by Theorem~4.1 of \cite{Y96}.
If condition~\eqref{eqn:thm3bp} does not hold,
 then we slightly modify the theorem to obtain the desired result.
\end{proof}

\begin{figure}[t]
\begin{center}
\includegraphics[width=40mm]{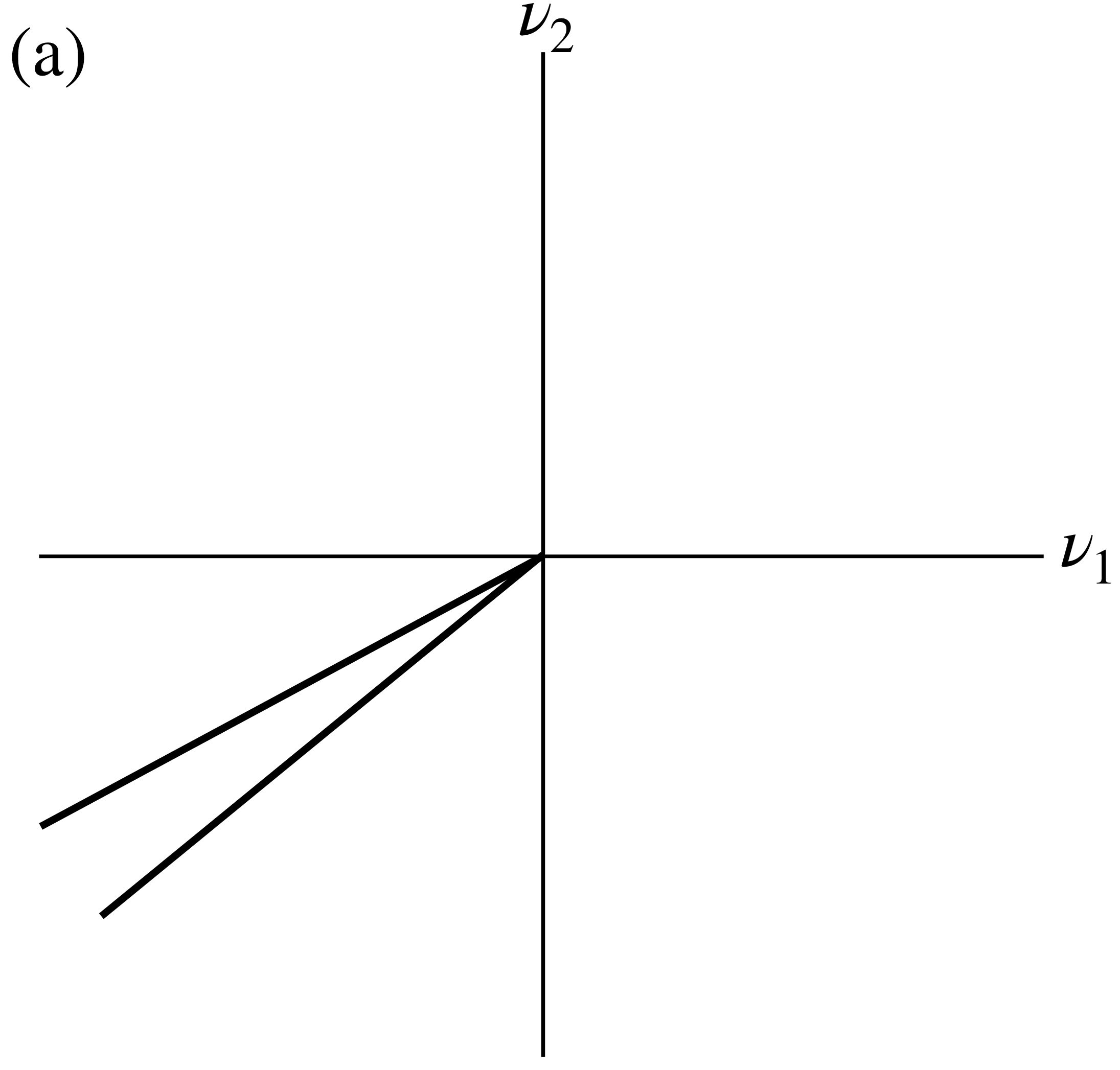}\quad
\includegraphics[width=40mm]{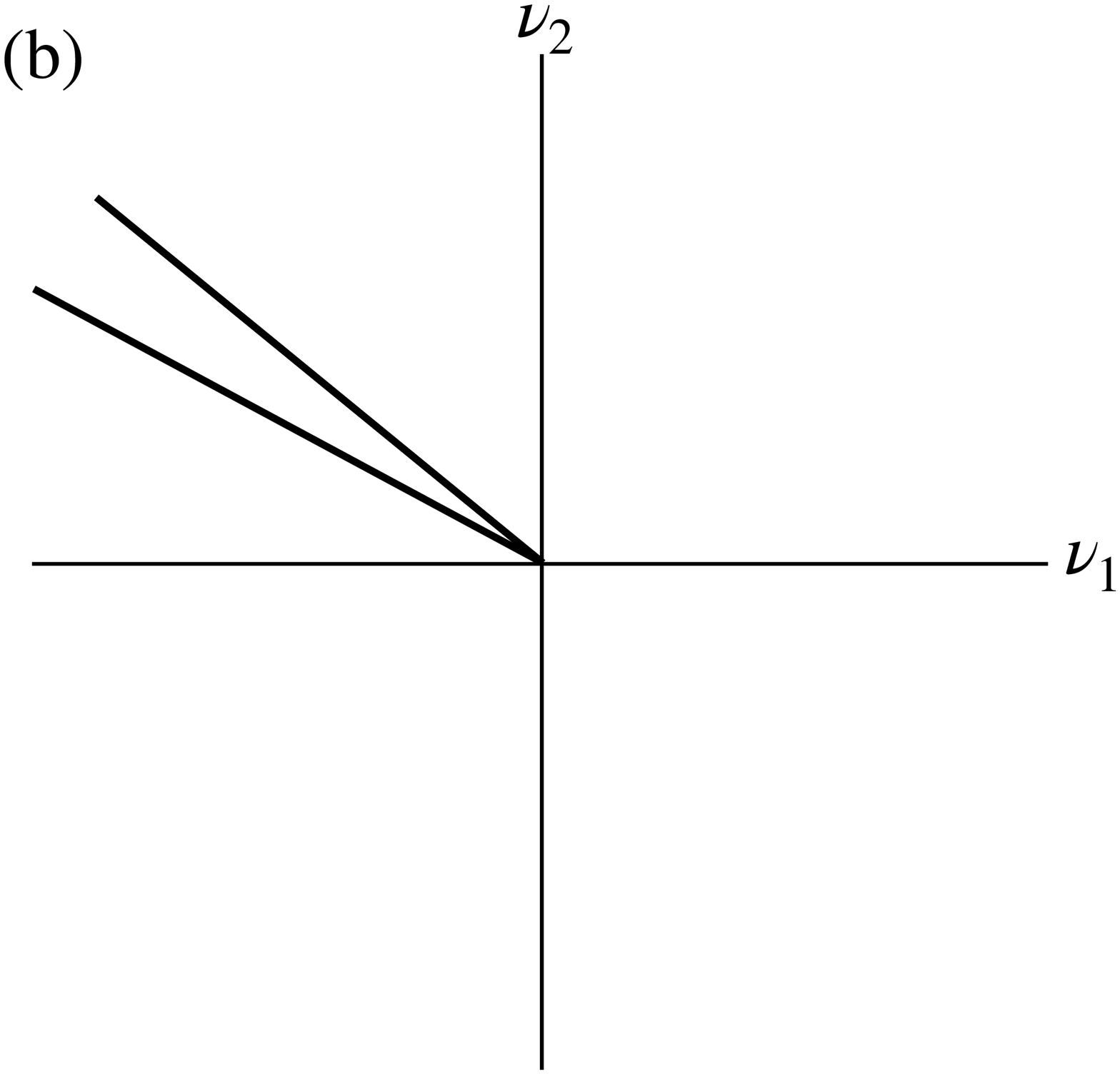}\\[2ex]
\includegraphics[width=50mm]{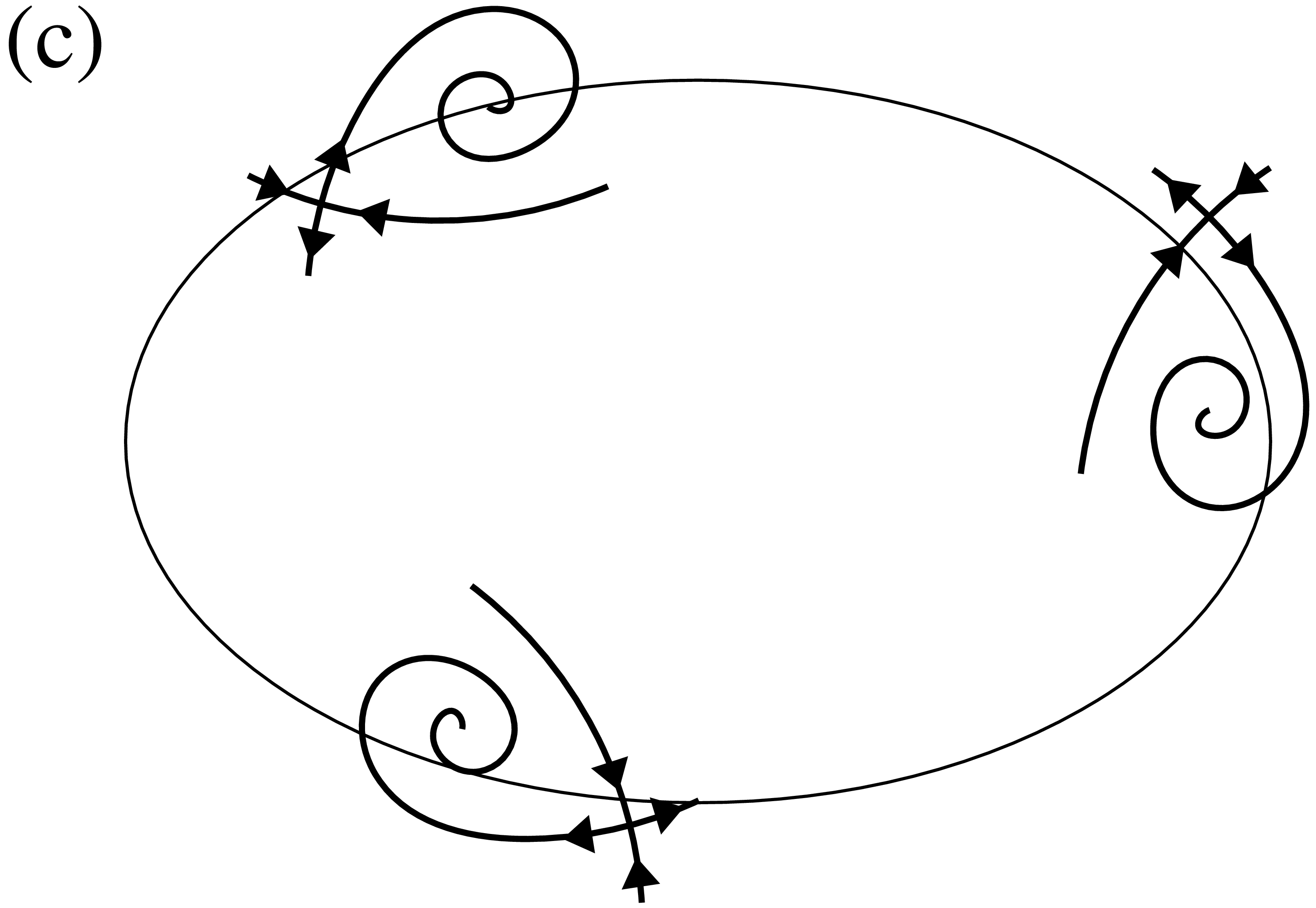}
\end{center}
\caption{
Saddle-node bifurcations in \eqref{eqn:psys1}:
(a) Approximate bifurcation sets for $s_2=1$;
(b) for $s_2=-1$;
(c) behavior of a Poincar\'e map near the unperturbed periodic orbit
 when periodic orbits of sink and saddle types appear,
 where $(m,n)=(3,1)$.}
\label{fig:thm3b}
\end{figure}

\begin{rmk}
\label{rmk:3b}
It follows from Theorem~$3.2$ of {\rm\cite{Y96}}
 that one of periodic orbits born at the saddle-node bifurcation
 detected in Theorem~$\ref{thm:3b}$
 is of a sink $($resp. a source$)$ type and  the other is of a saddle type if
\[
L^{m/n}(\phi;\hat{\nu})=m\hat{\nu}\hat{T}+s_2J_3(k,n)<0\quad(\mbox{resp. $>0$})
\]
i.e.,
\begin{equation}
\nu_2<\frac{s_2J_3(k,n)}{m\hat{T}}\nu_1\quad
\left(\mbox{resp. $\displaystyle\nu_2>\frac{s_2J_3(k,n)}{m\hat{T}}\nu_1$}\right).
\label{eqn:rmk3b1}
\end{equation}
Hence, the former changes its stability from a source type to a sink one or vice versa near
\begin{equation}
L^{m/n}(\phi;\hat{\nu})=m\hat{\nu}\hat{T}+s_2J_3(k,n)=0,\quad\mbox{i.e.,}\quad
\nu_2=\frac{s_2J_3(k,n)}{m\hat{T}}\nu_1,\quad
\nu_1<0.
\label{eqn:rmk3b2}
\end{equation}
This suggests that a Hopf bifurcation may occur there.
\end{rmk}

Figures~\ref{fig:thm3b}(a) and (b),
 show the approximate saddle-node bifurcation sets detected by Theorem~\ref{thm:3b}
 for $s_2=1$ and $-1$, respectively.
Figure~\ref{fig:thm3b}(c) shows the behavior of a Poincar\'e map
 near the resonant periodic orbit $\zeta^k(t)$ with $nT^k=m\hat{T}$ for $(m,n)=(3,1)$
 when periodic orbits of sink and saddle types appear,
 in the region between the two bifurcation curves in both figures.

\section{Melnikov Analyses for $s_1=-1$}

We turn to the case of $s_1=-1$ in \eqref{eqn:pnf}.
Assume that $|\nu_1|,|\nu_2|\ll 1$, $\nu_2=O(\nu_1)$ and $\epsilon=O(\nu_1^2)$.
We discuss the two cases of $\nu_1<0$ and of $\nu_1>0$ separately.

\subsection{Case of $\nu_1<0$}\label{section4.1}

Let $\hat{\epsilon}$ be a small positive parameter such that $\nu_1=-\hat{\epsilon}^2$,
 and introduce the new state variables $\zeta=(\zeta_1,\zeta_2)$
 and the new parameter $\hat{\nu}=O(1)$ as in \eqref{eqn:sc1}.
Rescaling the time variable $t$ as $t\to\sqrt{-\nu_1}\,t$,
 we rewrite \eqref{eqn:pnf} as
\begin{align}
\dot{\zeta}_1=\zeta_2,\quad
\dot{\zeta}_2=-\zeta_1-\zeta_1^3
 +\hat{\epsilon}(\hat{\nu}\zeta_2+s_2\zeta_1^2\zeta_2+\Delta h(\hat{\omega}t)),
\label{eqn:psys2a}
\end{align}
where $\nu_1=-\hat{\epsilon}^2$,
 and $\hat{\omega}$ and $\Delta$ are given by \eqref{eqn:hat}.

When $\hat{\epsilon}=0$, the system~\eqref{eqn:psys2a} becomes
\begin{equation}
\dot{\zeta}_1=\zeta_2,\quad
\dot{\zeta}_2=-\zeta_1-\zeta_1^3,
\label{eqn:psys2a0}
\end{equation}
which is a planar Hamiltonian system with the Hamiltonian
\[
H(\zeta)=\frac{1}{2}\zeta_2^2+\frac{1}{2}\zeta_1^2+\frac{1}{4}\zeta_1^4.
\]
The unperturbed system \eqref{eqn:psys2a0} has a one-parameter family of periodic orbits,
\begin{align}
\zeta^k(t)=&\left(\frac{\sqrt{2}k}{\sqrt{1-2k^2}}\cn\left(\frac{t}{\sqrt{1-2k^2}}
\right)\right.,\notag\\
&\quad\left. -\frac{\sqrt{2}k}{1-2k^2}\sn\left(\frac{t}{\sqrt{1-2k^2}}\right)
\dn\left(\frac{t}{\sqrt{1-2k^2}}\right)\right),\quad k\in(0,1/\sqrt{2}),
\label{eqn:po2a}
\end{align}
where $k$ represents the elliptic modulus of the Jacobi's elliptic functions.
The period of $\zeta^k(t)$ is given by
\begin{align*}
T^k=4K(k)\sqrt{1-2k^2}.
\end{align*}
See Fig.~\ref{fig:4a} for the phase portrait of \eqref{eqn:psys2a0}.

\begin{figure}[t]
\begin{center}
\includegraphics[scale=0.6]{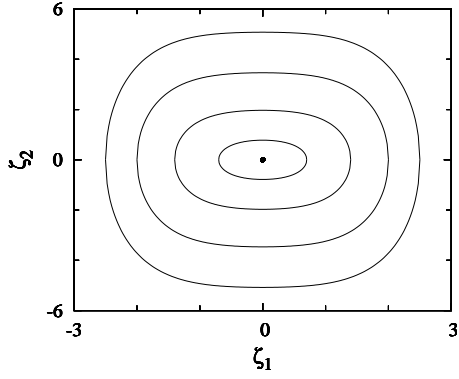}
\caption{Phase portraits of \eqref{eqn:psys2a0}.}
\label{fig:4a}
\end{center}
\end{figure}

We apply the subharmonic Melnikov method of \cite{Y96}.
Let $\Omega^k=2\pi/T^k$ and $I^k$, respectively,
 represent the angular frequency and action of  $\zeta^k(t)$ (see \eqref{eqn:Ik}).
The Hamiltonian energy for $\zeta^k(t)$ is given by
\begin{align*}
H^k=\frac{k^2(1-k^2)}{(1-2k^2)^2},
\end{align*}
so that
\begin{align*}
\frac{\d H^k}{\d k}=\frac{2k}{(1-2k^2)^3}>0.
\end{align*}
Since $\d H^k/\d I^k=\Omega^k>0$, we have
\begin{align*}
\frac{\d I^k}{\d k}=\frac{\d H^k/\d k}{\d H^k/\d I^k}>0.
\end{align*}
We also compute
\begin{align*}
\frac{\d\Omega^k}{\d k}=-\frac{\pi\left[(1-2k^2)(E(k)-K(k))-k^2K(k)\right]}
{2k(1-k^2)(1-2k^2)^{3/2}K(k)^2}>0
\end{align*}
to obtain
\begin{align*}
\frac{\d\Omega^k}{\d I^k}
 =\frac{\d\Omega^k/\d k}{\d I^k/\d k}>0.
\end{align*}

Let $nT^k=m\hat{T}$ for $m,n>0$ relatively prime integers.
Recall that $\hat{T}=|c/d|\hat{\epsilon}\,T$.
We compute the subharmonic Melnikov functions as
\begin{align*}
M^{m/n}(\phi;\hat{\nu})
=&\int_0^{m\hat{T}}\zeta_2^k(t)
 \left[\hat{\nu}\zeta_2^k(t)+s_2\zeta_1^k(t)^2\zeta_2^k(t)
 +\Delta h(\hat{\omega}t+\phi)\right]\d t\notag\\
=& \hat{\nu}J_1(k,n)+s_2J_2(k,n)+\Delta\hat{h}^{m/n}(\phi)
\end{align*}
and
\begin{align*}
L^{m/n}(\phi;\hat{\nu})
=& \int_0^{m\hat{T}}(\hat{\nu}+s_2\zeta_1^k(t)^2)\d t
=m\hat{\nu}\hat{T}+s_2 J_3(k,n),
\end{align*}
where
\begin{align*}
J_1(k,n)=&\int_0^{m\hat{T}}\zeta_2^k(t)^2\d t
=\frac{2k^2}{(1-2k^2)^2}\int_0^{m\hat{T}}\sn^2\left(\frac{t}{\sqrt{1-2k^2}}\right)
\dn^2\left(\frac{t}{\sqrt{1-2k^2}}\right)\d t\\
=&\frac{8n}{3(1-2k^2)^{3/2}}\left[(2k^2-1)E(k)+k'^2K(k)\right]>0,\\
J_2(k,n)=&\int_0^{m\hat{T}}\zeta_1^k(t)^2\zeta_2^k(t)^2\d t\\
=&\frac{4k^4}{(1-2k^2)^3}\int_0^{m\hat{T}}
\dn^2\left(\frac{t}{\sqrt{1-2k^2}}\right)
\sn^2\left(\frac{t}{\sqrt{1-2k^2}}\right)
\cn^2\left(\frac{t}{\sqrt{1-2k^2}}\right)\d t\\
=&\frac{16n}{15(1-2k^2)^{5/2}}\left[2(k^4-k^2+1)E(k)-k'^2(2-k^2)K(k)\right]>0,\\
J_3(k,n)=&\int_0^{m\hat{T}}\zeta_1^k(t)^2\d t
=\frac{2k^2}{1-2k^2}\int_0^{m\hat{T}}\cn^2\left(\frac{t}{\sqrt{1-2k^2}}\right)\d t\\
=&\frac{8n}{\sqrt{1-2k^2}}\left[E(k)-k'^2K(k)\right]>0
\end{align*}
and $\hat{h}^{m/n}(\phi)$ is given by \eqref{eqn:hathmn}
 with $\zeta_2^k(t)$  of \eqref{eqn:po2a},
 for which the statement of Lemma~\ref{lem:3b} also holds.
Recall that $k'=\sqrt{1-k^2}$ is complementary elliptic modulus.
We prove the following theorem like Theorem~\ref{thm:3b}.

\begin{thm}
\label{thm:4a}
Suppose that $\hat{h}^{m/n}(\phi)\not\equiv 0$
 and take $\nu_1$ or $\nu_2$ as a control parameter.
Then saddle-node bifurcations of $m\hat{T}$-periodic $($resp. $mT$-periodic$)$ orbits occur
 near \eqref{eqn:thm3b} in \eqref{eqn:psys2a} {\rm(}resp. in \eqref{eqn:sys}$)$.
\end{thm}

\begin{rmk}
\label{rmk:4a}
As in Remark~$\ref{rmk:3b}$,
 one of periodic orbits born at the saddle-node bifurcation
 detected in Theorem~$\ref{thm:4a}$ is of a sink $($resp. a source$)$ type
 and the other is of a saddle type if condition~\eqref{eqn:rmk3b1} holds,
 and the former changes its stability from a source type to a sink one or vice versa
 near \eqref{eqn:rmk3b2}.
This suggests that a Hopf bifurcation may occur there.
\end{rmk}

Saddle-node bifurcation sets detected by Theorem~\ref{thm:4a}
 and the behavior of a Poincar\'e map near $\zeta^k(t)$ with $nT^k=m\hat{T}$
 for $(m,n)=(3,1)$
 are similar to those displayed in Fig.~\ref{fig:thm3b}.

\subsection{Case of $\nu_1>0$}\label{section4.2}
Let $\hat{\epsilon}$ be a small positive parameter such that $\nu_1=\hat{\epsilon}^2$,
 and introduce the new state variables $\zeta=(\zeta_1,\zeta_2)$
 and the new parameter $\hat{\nu}=O(1)$ as
\begin{align}\label{trans4.2}
y_1=\sqrt{\nu_1}\zeta_1,\quad
y_2=\nu_1\zeta_2,\quad
\hat{\nu}=\nu_2/\nu_1.
\end{align}
Rescaling the time variable $t$ as $t\to\sqrt{\nu_1}\,t$,
 we rewrite \eqref{eqn:pnf} as
\begin{align}
\dot{\zeta}_1=\zeta_2,\quad
\dot{\zeta}_2=\zeta_1-\zeta_1^3
 +\hat{\epsilon}(\hat{\nu}\zeta_2+s_2\zeta_1^2\zeta_2+\Delta h(\hat{\omega}t)),
\label{eqn:psys2b}
\end{align}
where $\hat{\omega}$ and $\Delta$ are given by \eqref{eqn:hat}.

\begin{figure}[t]
\begin{center}
\includegraphics[scale=0.6]{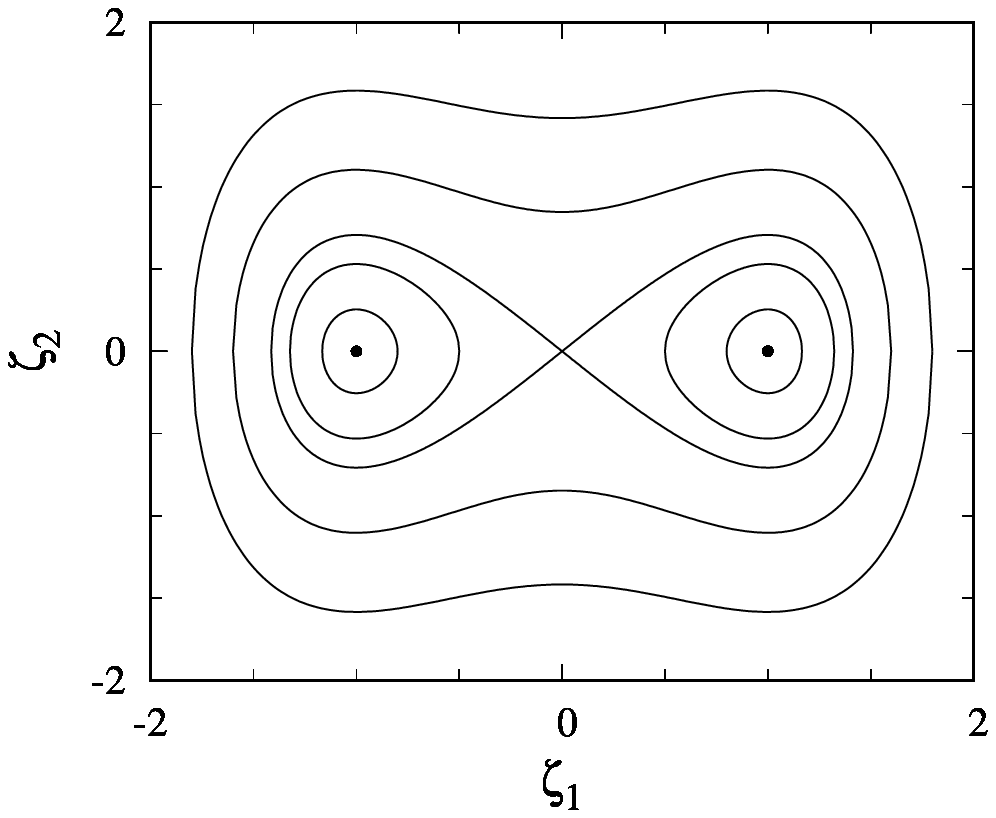}
\caption{Phase portraits of \eqref{eqn:psys2b0}.}
\label{fig:4b}
\end{center}
\end{figure}

When $\hat{\epsilon}=0$, the system~\eqref{eqn:psys2b} becomes
\begin{equation}
\dot{\zeta}_1=\zeta_2,\quad
\dot{\zeta}_2=\zeta_1-\zeta_1^3
\label{eqn:psys2b0}
\end{equation}
which is a planar Hamiltonian system with the Hamiltonian
\[
H(\zeta)=\frac{1}{2}\zeta_2^2-\frac{1}{2}\zeta_1^2+\frac{1}{4}\zeta_1^4.
\]
The unperturbed system \eqref{eqn:psys2b0} has a hyperbolic saddle at the origin
 to which there exists a pair of homoclinic orbits
\begin{align}
\zeta_{\pm}^{\h} (t)=
\left(\pm\sqrt{2}\sech t,\mp\sqrt{2}\sech t\tanh t\right)
\label{eqn:ho2}
\end{align}
Moreover, inside each homoclinic orbits,
 there exists two one-parameter families of periodic orbits,
\begin{align}
\zeta_{\pm}^k(t)=&\left(\pm\sqrt{\frac{2}{2-k^2}}\dn\left(\frac{t}{\sqrt{2-k^2}}
\right)\right.,\notag\\
&\quad\left. \mp\frac{\sqrt{2}k^2}{2-k^2}\sn\left(\frac{t}{\sqrt{2-k^2}}\right)
\cn\left(\frac{t}{\sqrt{2-k^2}}\right)\right),\quad k\in(0,1),
\label{eqn:po2b1}
\end{align}
and outside of the homoclinic orbits,
 there exists a one-parameter family of periodic orbits,
\begin{align}
\zeta^k(t)=&\left(\frac{\sqrt{2}k}{\sqrt{2k^2-1}}
\cn\left(\frac{t}{\sqrt{2k^2-1}}\right),\right.\notag\\
&\quad\left.-\frac{\sqrt{2}k}{{2k^2-1}}
\sn\left(\frac{t}{\sqrt{2k^2-1}}\right)
\dn\left(\frac{t}{\sqrt{2k^2-1}}\right)
\right),\quad k\in(1/\sqrt{2},1),
\label{eqn:po2b2}
\end{align}
where $k$ represents the elliptic modulus of the Jacobi's elliptic functions.
The periods of $\zeta_{\pm}^k(t)$ and $\zeta^k(t)$ are, respectively, given by
\begin{align}
T_\pm^k=2K(k)\sqrt{2-k^2}
\quad\mbox{and}\quad
T^k=4K(k)\sqrt{2k^2-1}.
\end{align}
See Fig.~\ref{fig:4b} for the phase portrait of \eqref{eqn:psys2b0}.

\subsubsection{Homoclinic orbits}
We apply the homoclinic Melnikov method \cite{GH83,W03}.
For $\hat{\epsilon}>0$ sufficiently small
 we easily see that the system~\eqref{eqn:psys2b} has a hyperbolic periodic orbit near the origin.
We compute the Melnikov functions for $\zeta_\pm^{\h}(t)$ as
\begin{align*}
M_\pm(\phi;\hat{\nu})
=&\int_{-\infty}^\infty\zeta_{2\pm}^{\h}(t)
 \left[\hat{\nu}(\zeta_{2\pm}^{\h}(t))
 +s_2\zeta_{1\pm}^{\h}(t)^2\zeta_{2\pm}^\h(t)
 +\Delta h(\hat{\omega}t+\phi)\right]\d t\notag\\
=&\frac{4}{3}\hat{\nu}+\frac{16}{15}s_2\pm\Delta\hat{h}(\phi).
\end{align*}
where $\hat{h}(\phi)$ is given by \eqref{eqn:hath}
 with $\zeta_{2+}^\h(t)$ of \eqref{eqn:ho2},
 for which the statement of Lemma~\ref{lem:3a} also holds.
In particular, corresponding to \eqref{eqn:lem3ap}, we have
\[
\tilde{\zeta}_{2+}(\chi)
 =\int_{-\infty}^\infty\zeta_{2+}^{\h}(t)e^{i\chi t}\d t
 =i\pi\chi\sech\left(\frac{\pi\chi}{2}\right)\neq 0
\]
for any $\chi\in\Rset$.

Assume that $h(\phi)\not\equiv 0$.
Then the conclusion of Lemma~\ref{lem:3a} holds.
If
\[
-\hat{h}_{\mathrm{max}}
<\frac{4}{3}\hat{\nu}+\frac{16}{15}s_2
<-\hat{h}_{\mathrm{min}}
\]
and
\[
\hat{h}_{\mathrm{min}}
<\frac{4}{3}\hat{\nu}+\frac{16}{15}s_2
<\hat{h}_{\mathrm{max}},
\]
i.e.,
\begin{equation}
-\left(\frac{3}{4}\Delta\hat{h}_{\mathrm{max}}+\frac{4}{5}s_2\right)\nu_1
<\nu_2<-\left(\frac{3}{4}\Delta\hat{h}_{\mathrm{min}}+\frac{4}{5}s_2\right)\nu_1
\label{eqn:hcon2a}
\end{equation}
and
\begin{equation}
\left(\frac{3}{4}\Delta\hat{h}_{\mathrm{min}}-\frac{4}{5}s_2\right)\nu_1
<\nu_2<\left(\frac{3}{4}\Delta\hat{h}_{\mathrm{max}}-\frac{4}{5}s_2\right)\nu_1,
\label{eqn:hcon2b}
\end{equation}
respectively, then the Melnikov functions $M_\pm(\phi)$ cross the zero topologically.
Using arguments given in \cite{GH83,W03}.
 we obtain the following result.

\begin{thm}
\label{thm:4b}
Suppose that $h(\phi)\not\equiv 0$.
If condition~\eqref{eqn:hcon2a} $($resp. \eqref{eqn:hcon2b}$)$ holds,
 then for $|\nu_1|,|\nu_2|,\epsilon>0$ sufficiently small
 the right $($resp. left$)$ branches of the stable and unstable manifolds
 of a periodic orbit near the origin intersect topologically transversely
 in \eqref{eqn:psys2b} and hence in \eqref{eqn:sys}.
Moreover, homoclinic bifurcations occur near the four curves
\begin{align}
&
\nu_2=-\left(\frac{3}{4}\Delta\hat{h}_{\mathrm{max}}+\frac{4}{5}s_2\right)\nu_1,\quad
\nu_2=-\left(\frac{3}{4}\Delta\hat{h}_{\mathrm{min}}+\frac{4}{5}s_2\right)\nu_1,\quad
\nu_1>0.
\label{eqn:thm4b1}\\
&
\nu_2=\left(\frac{3}{4}\Delta\hat{h}_{\mathrm{min}}-\frac{4}{5}s_2\right)\nu_1,\quad
\nu_2=\left(\frac{3}{4}\Delta\hat{h}_{\mathrm{max}}-\frac{4}{5}s_2\right)\nu_1,\quad
\nu_1>0.
\label{eqn:thm4b2}
\end{align}
\end{thm}

Such transverse intersections also yield chaotic dynamics
 in \eqref{eqn:psys2b} and \eqref{eqn:sys} like geometrically transverse intersections,
 as stated in Remark~\ref{rmk:3a}.
 
\begin{figure}[t]
\begin{center}
\includegraphics[width=45mm]{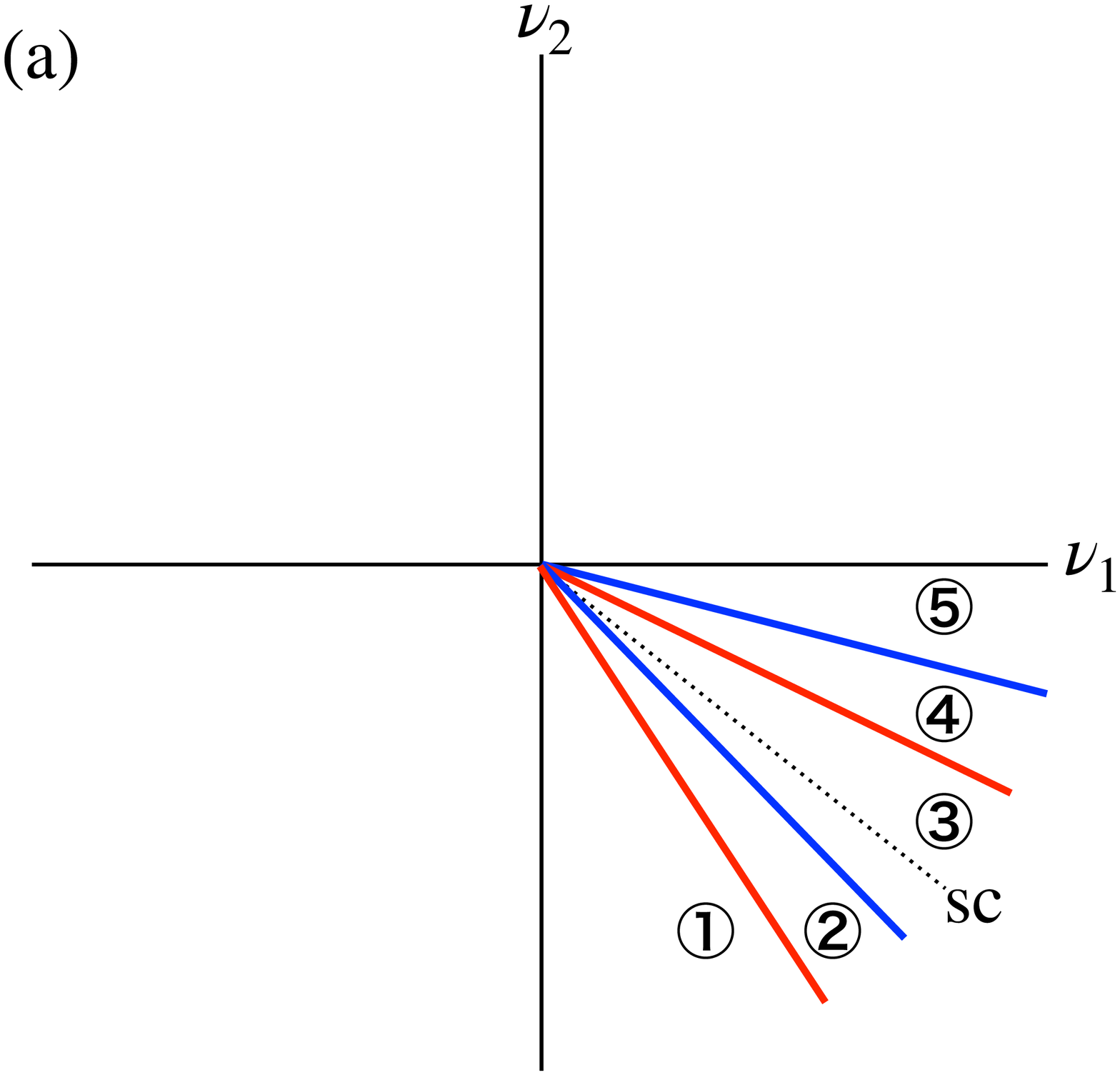}\quad
\includegraphics[width=45mm]{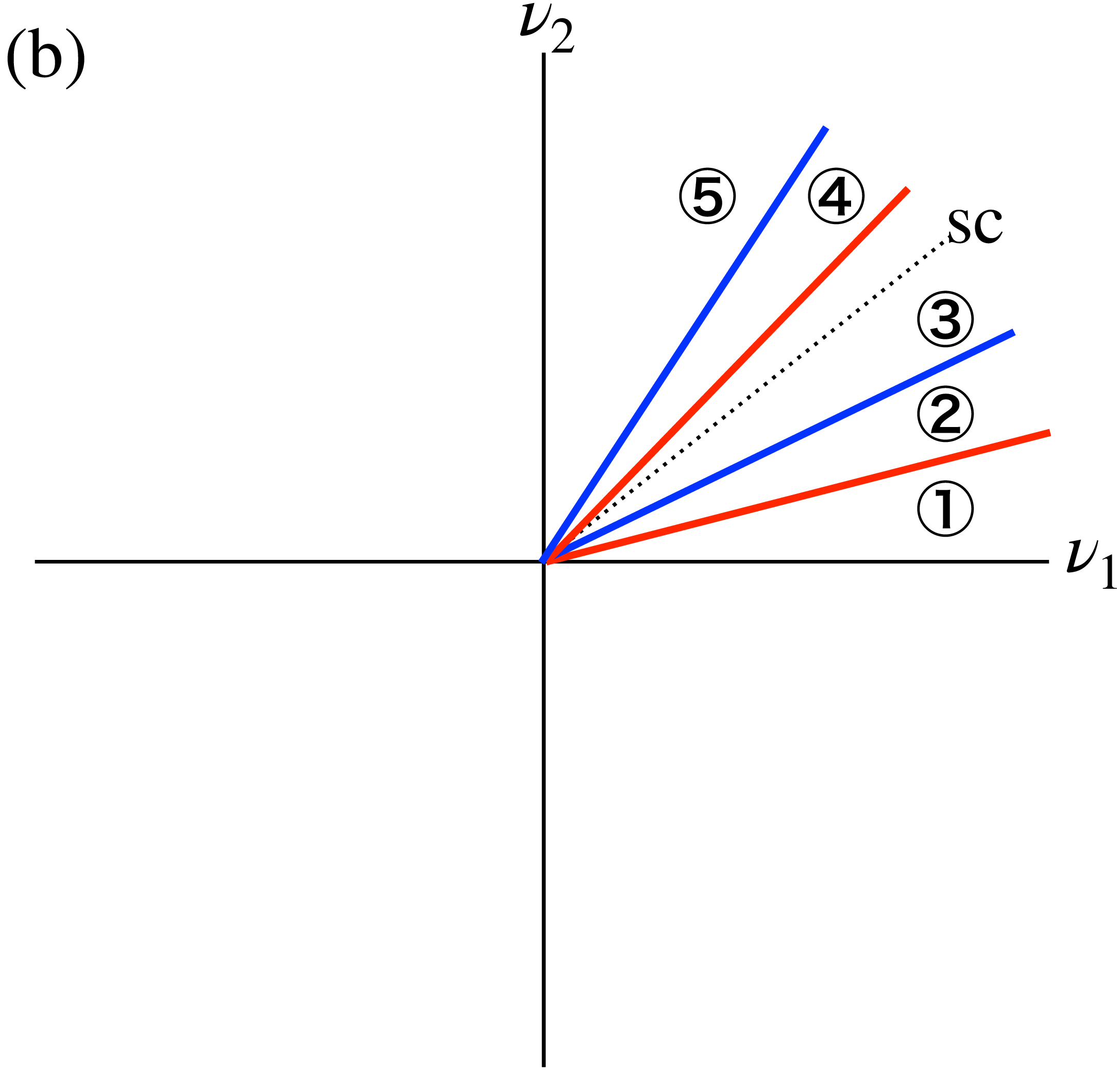}\\[2ex]
\includegraphics[width=100mm]{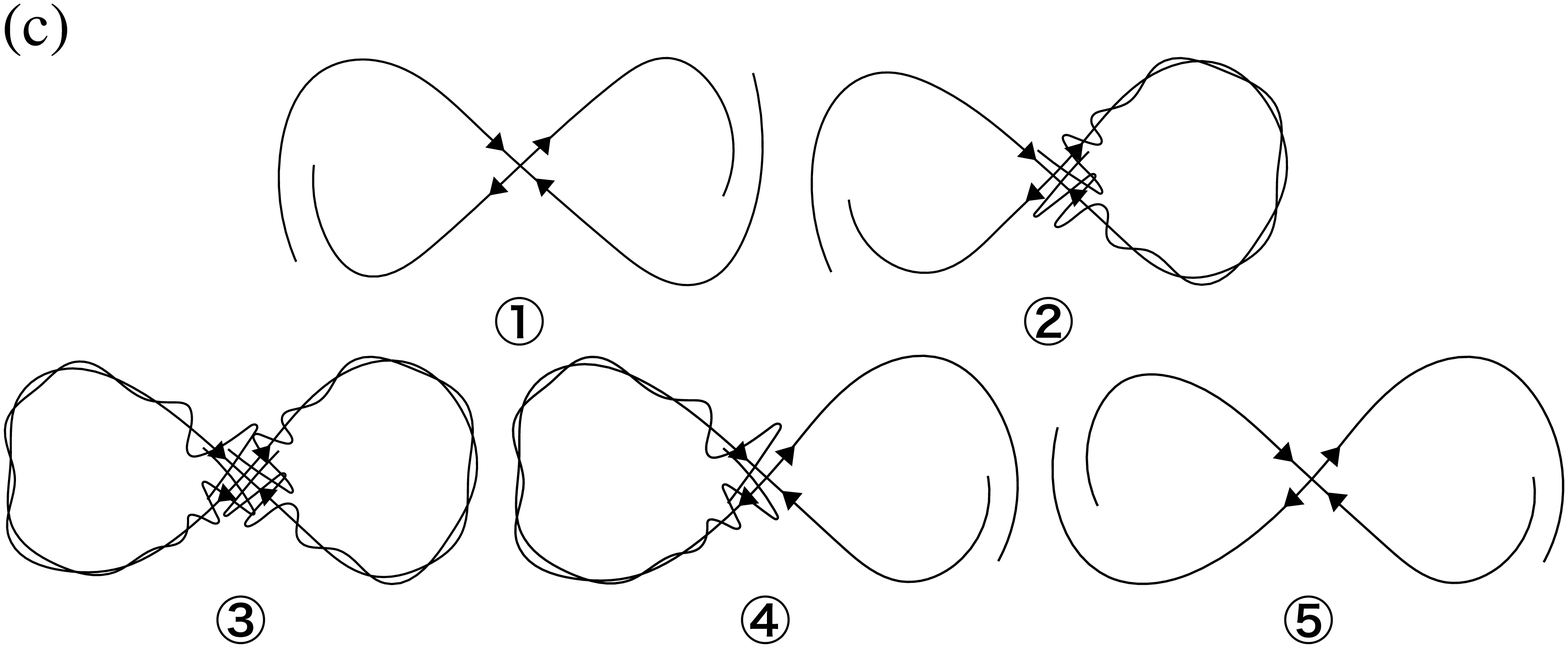}
\end{center}
\caption{
Homoclinic bifurcations in \eqref{eqn:psys2b}:
(a) Approximate bifurcation sets for $s_2=1$;
(b) for $s_2=-1$;
(c) stable and unstable manifolds for a Poincar\'e map.}
\label{fig:thm4b}
\end{figure}

Figures~\ref{fig:thm4b}(a) and (b)
 show the approximate homoclinic bifurcation sets detected by Theorem~\ref{thm:4b}
 for $s_2=1$ and $-1$, respectively.
The red and blue lines correspond
 to the bifurcation sets \eqref{eqn:thm4b1} and \eqref{eqn:thm4b2}, respectively.
Here we have assumed that $|\hat{h}_{\mathrm{min}}|<|\hat{h}_{\mathrm{max}}|$.
Figure~\ref{fig:thm4b}(c) shows the behavior
 of the  stable and unstable manifolds of periodic orbits near the origin for a Poincar\'e map
 in regions \textcircled{\scriptsize 1}-\textcircled{\scriptsize 5} of both figures.
In particular, in regions \textcircled{\scriptsize 2}-\textcircled{\scriptsize 4},
 there are transverse homoclinic orbits which yield chaotic motions, as stated above.
When $|\hat{h}_{\mathrm{min}}|>|\hat{h}_{\mathrm{max}}|$,
 the loci of the red and blue lines are exchanged in Figs.~\ref{fig:thm4b}(a) and (b);
 and when $|\hat{h}_{\mathrm{min}}|=|\hat{h}_{\mathrm{max}}|$,
 regions \textcircled{\scriptsize 2} and \textcircled{\scriptsize 4} disappear
 or their widths become $o(\sqrt{\nu_1})$.

\subsubsection{Subharmonic orbits}

We apply the subharmonic Melnikov method of \cite{Y96}.
We begin with the families of periodic orbits $\zeta_\pm^k(t)$ given by \eqref{eqn:po2b1}.
Let $\Omega_\pm^k$ and $I_\pm^k$, respectively,
 represent the angular frequency and action of  $\zeta_\pm^k(t)$,
 where $I_\pm^k$ are given by \eqref{eqn:Ik} with $\zeta^k(t)=\zeta_\pm^k(t)$.
The Hamiltonian energy for $\zeta_\pm^k(t)$ is given by
\begin{align*}
H^k=\frac{k^2-1}{(2-k^2)^2},
\end{align*}
so that
\begin{align*}
\frac{\d H^k}{\d k}=
\frac{2k^3}{(2-k^2)^3}>0.
\end{align*}
Since $\d H^k/\d I_\pm^k=\Omega_\pm^k>0$ and
\[
\frac{\d}{\d k}\left[(2-k^2)E(k)-2(1-k^2)K(k)\right]
=-3k\left[E(k)-K(k)\right]>0,
\]
we have
\[
\frac{\d I_\pm^k}{\d k}=\frac{\d H^k/\d k}{\d H^k/\d I_\pm^k}>0
\]
and
\begin{align*}
\frac{\d\Omega_\pm^k}{\d k}
=-\frac{\pi[(2-k^2)E(k)-2(1-k^2)K(k)]}{k(2-k^2)^{3/2}(1-k^2)K(k)}<0.
\end{align*}
Hence,
\begin{align*}
\frac{\d\Omega_\pm^k}{\d I_\pm^k}
 =\frac{\d\Omega_\pm^k/\d k}{\d I_\pm^k/\d k}<0.
\end{align*}

Let $nT^k=m\hat{T}$ for $m,n>0$ relatively prime integers.
Recall that $\hat{T}=|c/d|\hat{\epsilon}\,T$.
We compute the subharmonic Melnikov functions as
\begin{align*}
M_\pm^{m/n}(\phi;\hat{\nu})
=&\int_0^{m\hat{T}}\zeta_{2\pm}^k(t)
 \left[\hat{\nu}\zeta_{2\pm}^k(t)+s_2\zeta_{1\pm}^k(t)^2\zeta_{2\pm}^k(t)
 +\Delta h(\hat{\omega}t+\phi)\right]\d t\notag\\
=& \hat{\nu}J_1(k,n)+s_2J_2(k,n)\pm\Delta\hat{h}^{m/n}(\phi)
\end{align*}
and
\begin{align*}
L_\pm^{m/n}(\phi;\hat{\nu})
=& \int_0^{m\hat{T}}(\hat{\nu}+s_2\zeta_{1\pm}^k(t)^2)\d t
=m\hat{\nu}\hat{T}+s_2 J_3(k,n),
\end{align*}
where
\begin{align*}
J_1(k,n)=&\int_0^{m\hat{T}}({\zeta}_{2\pm}^k(t))^2\d t
=\frac{2k^4}{(2-k^2)^2}\int_0^{m\hat{T}}\sn^2\left(\frac{t}{\sqrt{2-k^2}}\right)
\cn^2\left(\frac{t}{\sqrt{2-k^2}}\right)\d t\notag\\
=&\frac{4n}{3(2-k^2)^{3/2}}\left[(2-k^2)E(k)-k'^2K(k)\right],\\
J_2(k,n)=&\int_0^{m\hat{T}}\zeta_{1\pm}^k(t)^2\zeta_{2\pm}^k(t)^2\d t\\
=&\frac{4k^4}{(2-k^2)^3}\int_0^{m\hat{T}}
\dn^2\left(\frac{t}{\sqrt{2-k^2}}\right)
\sn^2\left(\frac{t}{\sqrt{2-k^2}}\right)
\cn^2\left(\frac{t}{\sqrt{2-k^2}}\right)\d t\notag\\
=&\frac{8n}{15(2-k^2)^{5/2}}\left[2(k^4-k^2+1)E(k)-k'^2(2-k^2)K(k)\right],\\
J_3(k,n)=&\int_0^{m\hat{T}}\zeta_{1\pm}^k(t)^2\d t
=\frac{2}{2-k^2}\int_0^{m\hat{T}}\dn^2\left(\frac{t}{\sqrt{2-k^2}}\right)\d t
=\frac{4nE(k)}{\sqrt{2-k^2}}
\end{align*}
and $\hat{h}^{m/n}(\phi)$ is given by \eqref{eqn:hathmn}
 with $\zeta_2^k(t)=\zeta_{2+}^k(t)$,
 for which the statement of Lemma~\ref{lem:3b} also holds.
Recall that $k'=\sqrt{1-k^2}$ is complementary elliptic modulus.
We prove the following theorem like Theorems~\ref{thm:3b} and \ref{thm:4a}.

\begin{thm}
\label{thm:4c}
Suppose that $\hat{h}^{m/n}(\phi)\not\equiv 0$
 and take $\nu_1$ or $\nu_2$ as a control parameter.
Then saddle-node bifurcations of $m\hat{T}$-periodic $($resp. $mT$-periodic$)$ orbits
 near the unperturbed periodic orbits with $\zeta_1>0$ and with $\zeta_1<0$ occur near the curves
\begin{equation}
\nu_2=-\frac{\Delta\hat{h}_{\mathrm{max}}^{m/n}+s_2J_2(k,n)}{J_1(k,n)}\nu_1,\quad
\nu_2=-\frac{\Delta\hat{h}_{\mathrm{min}}^{m/n}+s_2J_2(k,n)}{J_1(k,n)}\nu_1,\quad
\nu_1>0,
\label{eqn:thm4c1}
\end{equation}
and
\begin{equation}
\nu_2=\frac{\Delta\hat{h}_{\mathrm{min}}^{m/n}-s_2J_2(k,n)}{J_1(k,n)}\nu_1,\quad
\nu_2=\frac{\Delta\hat{h}_{\mathrm{max}}^{m/n}-s_2J_2(k,n)}{J_1(k,n)}\nu_1,\quad
\nu_1>0,
\label{eqn:thm4c2}
\end{equation}
in \eqref{eqn:psys2b} {\rm(resp. in \eqref{eqn:sys})}, respectively.
\end{thm}

\begin{figure}[t]
\begin{center}
\includegraphics[width=45mm]{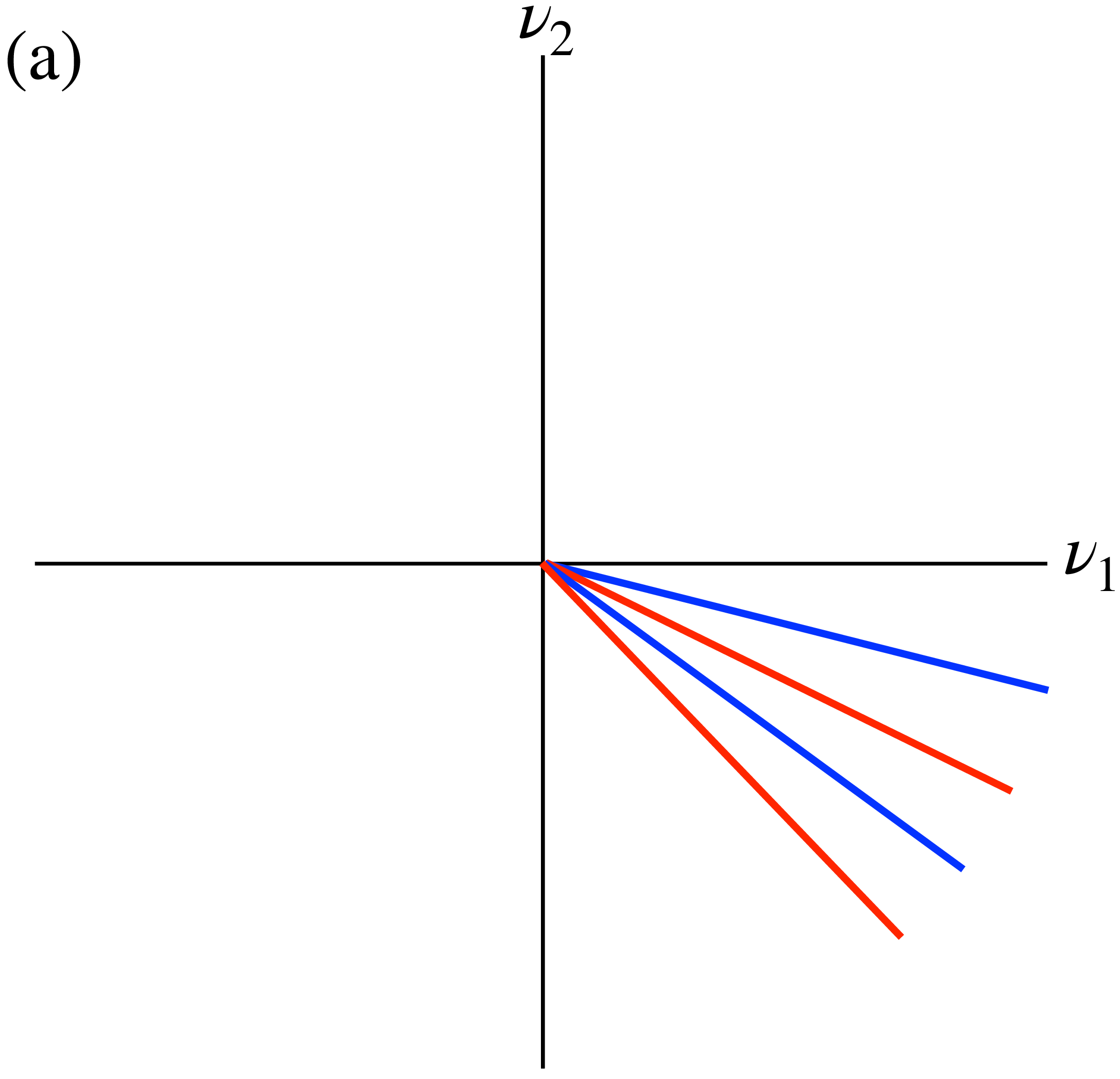}\quad
\includegraphics[width=45mm]{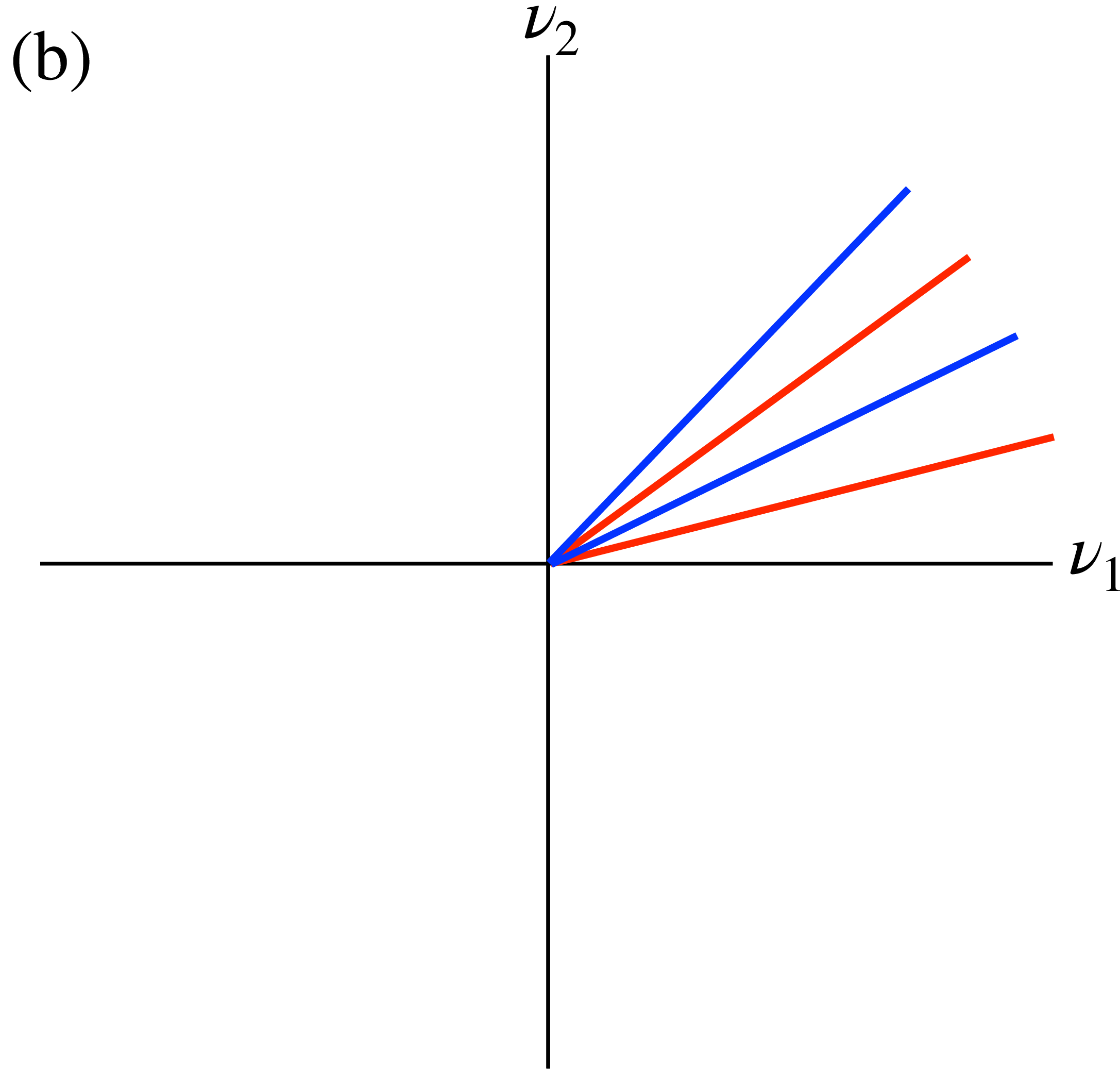}
\end{center}
\caption{Sddle-node bifurcations in \eqref{eqn:psys2b}:
(a) Approximate bifurcation sets for $s_2=1$;
(b) for $s_2=-1$.}
\label{fig:thm4c}
\end{figure}

Figures~\ref{fig:thm4c}(a) and (b)
 show the approximate saddle-node bifurcation sets detected by Theorem~\ref{thm:4c}
 for $s_2=1$ and $-1$, respectively.
Here we have assumed that $|\hat{h}_{\mathrm{min}}|<|\hat{h}_{\mathrm{max}}|$.
The red and blue lines correspond
 to the bifurcation sets \eqref{eqn:thm4c1} and \eqref{eqn:thm4c2}, respectively.
The behavior of a Poincar\'e map near the resonant periodic orbit $\zeta_+^k(t)$
 (resp. near $\zeta_-^k(t)$) with $nT_+^k=m\hat{T}$ (resp. with $nT_-^k=m\hat{T}$)
 for $(m,n)=(3,1)$ in the regions
 between the two red lines (resp. the two blue lines) in Figs.~\ref{fig:thm4c}(a) and (b)
 is similar to that displayed in Figure~\ref{fig:thm3b}(c).
When $|\hat{h}_{\mathrm{min}}|>|\hat{h}_{\mathrm{max}}|$,
 the loci of the red and blue lines are exchanged in Figs.~\ref{fig:thm4c}(a) and (b).

We turn to the family of periodic orbits $\zeta^k(t)$ given by \eqref{eqn:po2b2}.
Let $\Omega^k=2\pi/T^k$ and $I_k$, respectively,
 represent the angular frequency and action of $\zeta^k(t)$,
 where $I^k$ is given by \eqref{eqn:Ik}.
The Hamiltonian energy for $\zeta^k(t)$ is given by
\begin{align*}
H^k=\frac{k^2(1-k^2)}{(2k^2-1)^2},
\end{align*}
so that
\begin{align*}
\frac{\d H^k}{\d k}=-\frac{2k}{(2k^2-1)^3}<0.
\end{align*}
Since $\d H^k/\d I^k=\Omega_\pm^k=2\pi/T^k>0$,
we have
\[
\frac{\d I^k}{\d k}=\frac{\d H^k/\d k}{\d H^k/\d I^k}<0
\]
and
\begin{align*}
\frac{\d\Omega^k}{\d k}
=-\frac{\pi[(1-k^2)(K(k)-E(k))+k^2E(k)]}{2k(1-k^2)(2k^2-1)^{3/2}K(k)^2}<0.
\end{align*}
Hence,
\begin{align*}
\frac{\d\Omega^k}{\d I^k}
 =\frac{\d\Omega^k/\d k}{\d I^k/\d k}>0.
\end{align*}

Let $nT^k=m\hat{T}$ for $m,n>0$ relatively prime integers.
Recall that $\hat{T}=|c/d|\hat{\epsilon}\,T$.
We compute the subharmonic Melnikov functions as
\begin{align*}
M^{m/n}(\phi;\hat{\nu})
=&\int_0^{m\hat{T}}\zeta_2^k(t)
 \left[\hat{\nu}\zeta_2^k(t)+s_2\zeta_1^k(t)^2\zeta_2^k(t)
 +\Delta h(\hat{\omega}t+\phi)\right]\d t\notag\\
=& \hat{\nu}J_1(k,n)+s_2J_2(k,n)+\Delta\hat{h}^{m/n}(\phi)
\end{align*}
and
\begin{align*}
L^{m/n}(\phi;\hat{\nu})
=& \int_0^{m\hat{T}}(\hat{\nu}+s_2\zeta_1^k(t)^2)\d t
=m\hat{\nu}\hat{T}+s_2 J_3(k,n),
\end{align*}
where
\begin{align*}
J_1(k,n)=&\int_0^{m\hat{T}}{\zeta}_2^k(t)^2\d t
=\frac{2k^2}{(2k^2-1)^2}\int_0^{m\hat{T}}\sn^2\left(\frac{t}{\sqrt{2k^2-1}}\right)
\dn^2\left(\frac{t}{\sqrt{2k^2-1}}\right)\d t\notag\\
=&\frac{8n}{3(2k^2-1)^{3/2}}\left[(2k^2-1)E(k)+k'^2K(k)\right]>0,\\
J_2(k,n)=&\int_0^{m\hat{T}}\zeta_1^k(t)^2\zeta_2^k(t)^2\d t\\
=&\frac{4k^4}{(2k^2-1)^3}\int_0^{m\hat{T}}
\dn^2\left(\frac{t}{\sqrt{2k^2-1}}\right)
\sn^2\left(\frac{t}{\sqrt{2k^2-1}}\right)
\cn^2\left(\frac{t}{\sqrt{2k^2-1}}\right)\d t\\
=&\frac{16n}{15(2k^2-1)^{5/2}}\left[2(k^4-k^2+1)E(k)-k'^2(2-k^2)K(k)\right]>0,\\
J_3(k,n)=&\int_0^{m\hat{T}}\zeta_1^k(t)^2\d t
=\frac{2k^2}{2k^2-1}\int_0^{m\hat{T}}\cn^2\left(\frac{t}{\sqrt{2k^2-1}}\right)\d t\notag\\
=&\frac{8n}{\sqrt{2k^2-1}}\left[E(k)-k'^2K(k)\right]>0
\end{align*}
and $\hat{h}^{m/n}(\phi)$ is given by \eqref{eqn:hathmn}
 with $\zeta_2^k(t)$ of \eqref{eqn:po2b2},
 for which the statement of Lemma~\ref{lem:3b} also holds.
We prove the following theorem like Theorems~\ref{thm:3b}, \ref{thm:4a} and \ref{thm:4c}.

\begin{thm}
\label{thm:4d}
Suppose that $\hat{h}^{m/n}(\phi)\not\equiv 0$
 and take $\nu_1$ or $\nu_2$ as a control parameter.
Then saddle-node bifurcations of $m\hat{T}$-periodic $($resp. $mT$-periodic$)$ orbits occur near
\begin{equation}
\nu_2=-\frac{\Delta\hat{h}_{\mathrm{max}}+s_2J_2(k,n)}{J_1(k,n)}\nu_1
\quad\mbox{and}\quad
\nu_2=-\frac{\Delta\hat{h}_{\mathrm{min}}+s_2J_2(k,n)}{J_1(k,n)}\nu_1,\quad
\nu_1>0
\label{eqn:thm4d}
\end{equation}
in \eqref{eqn:psys2b} {\rm(resp. in \eqref{eqn:sys})}.
\end{thm}

\begin{rmk}
\label{rmk:4b}
As in Remarks~$\ref{rmk:3b}$ and $\ref{rmk:4a}$,
 one of periodic orbits born at the saddle-node bifurcations
 detected in Theorem~$\ref{thm:4c}$ or $\ref{thm:4d}$
 is of a sink $($resp. a source$)$ type and the other is of a saddle type if
\[
m\hat{\nu}\hat{T}+s_2J_3(k,n)<0\quad(\mbox{resp. $>0$}),
\]
i.e.,
\[
\nu_2<-\frac{s_2J_3(k,n)}{m\hat{T}}\nu_1\quad
\left(\mbox{resp. $\displaystyle\nu_2>-\frac{s_2J_3(k,n)}{m\hat{T}}\nu_1$}\right),
\]
and the former changes its stability
 from a source type to a sink one or vice versa near
\[
\nu_2=-\frac{s_2J_3(k,n)}{m\hat{T}}\nu_1,\quad
\nu_1>0.
\]
This suggests that a Hopf bifurcation may occur there.
\end{rmk}

\begin{figure}[t]
\begin{center}
\includegraphics[width=45mm]{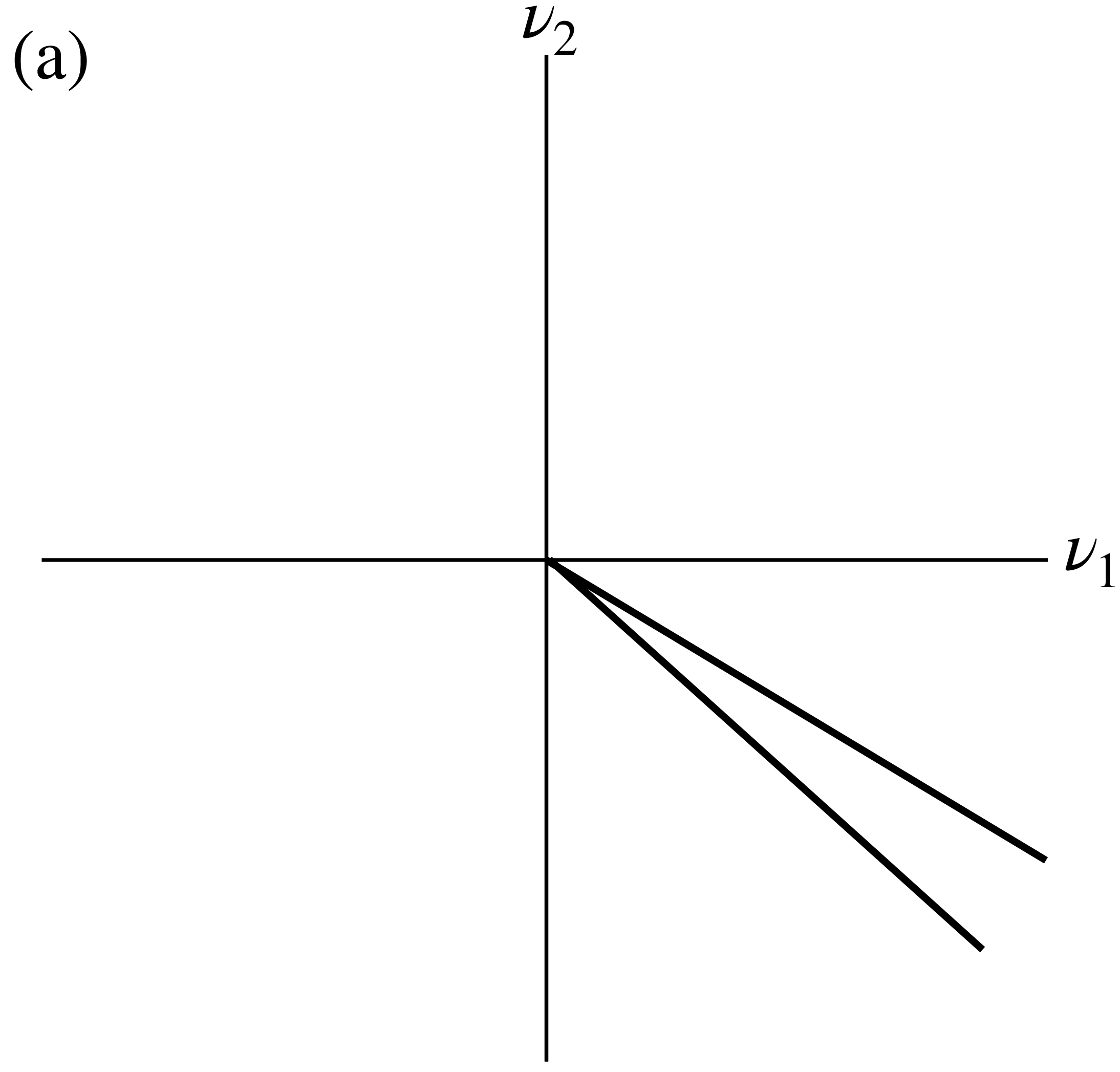}\quad
\includegraphics[width=45mm]{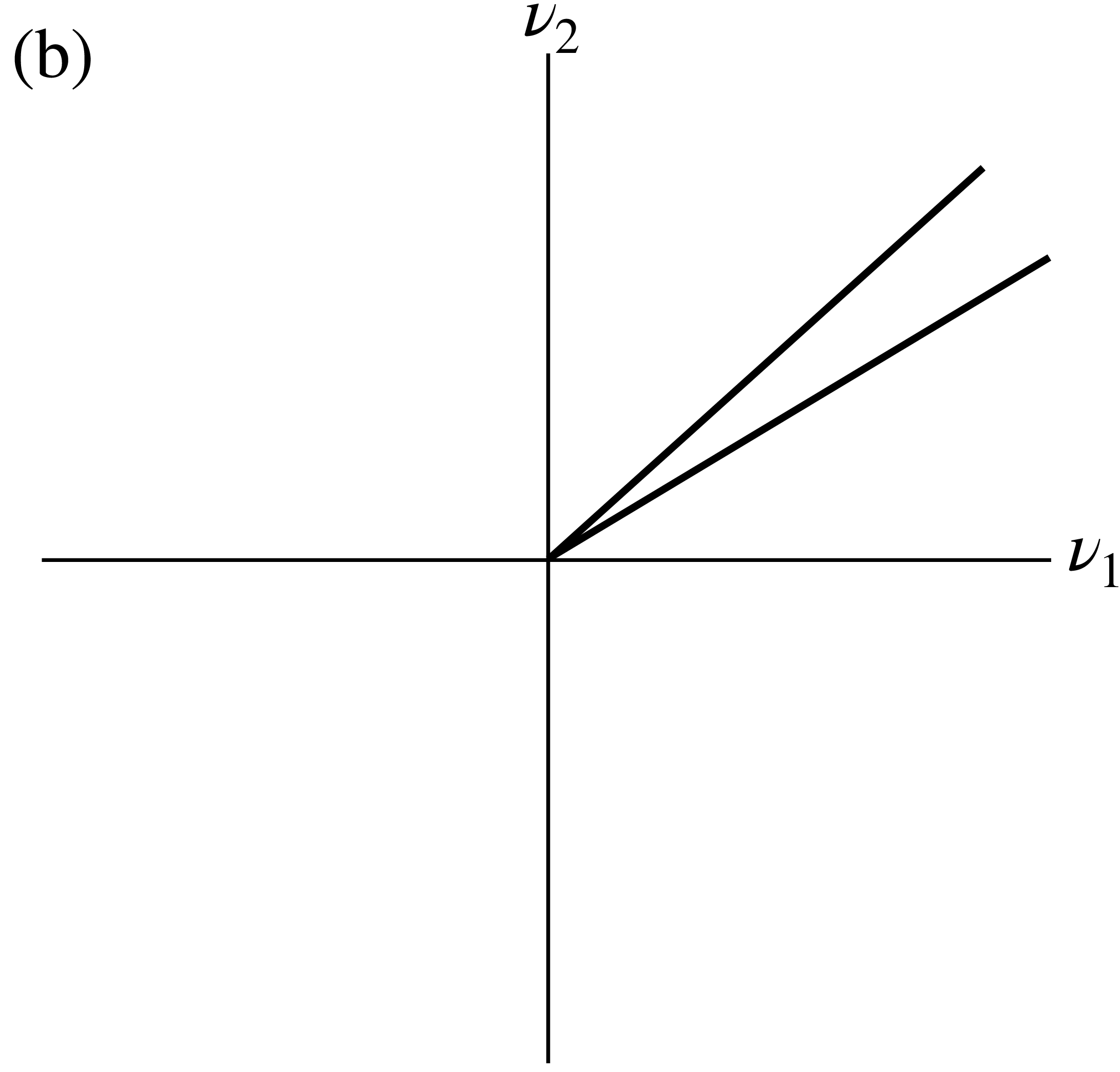}
\end{center}
\caption{Saddle-node bifurcations of \eqref{eqn:psys2b}:
(a) Approximate bifurcation sets for $s_2=1$;
(b) for $s_2=-1$.}
\label{fig:thm4d}
\end{figure}

Figures~\ref{fig:thm4d}(a) and (b)
 show the approximate saddle-node bifurcation sets detected by Theorem~\ref{thm:4d}
 for $s_2=1$ and $-1$, respectively.
The behavior of a Poincar\'e map near the resonant periodic orbit $\zeta^k(t)$
with $nT_+^k=m\hat{T}$ for $(m,n)=(3,1)$ is similar to that displayed in Fig.~\ref{fig:thm3b}(c).

\section{Higher-dimensional systems}

Let $\epsilon>0$ be a small parameter
 and let $\omega>0$ be a constant, as in the previous sections.
Consider $(d+2)$-dimensional systems of the form
\begin{equation}
\begin{split}
\dot{u}=&Ju+f_u(u,v;\mu)+\epsilon g_u(u,v,\omega t),\\
\dot{v}=&Av+f_v(u,v;\mu)+\epsilon g_v(u,v,\omega t),
\end{split}\qquad
(u,v)\in\Rset^2\times\Rset^d,
\label{eqn:sysuv}
\end{equation}
with a control parameter vector $\mu\in\Rset^2$,
 where $A$ is an $n\times n$ matrix of which all eigenvalues have nonzero real parts,
 $f_u:\Rset^2\times\Rset^d\times\Rset^2\to\Rset^2$
 and $f_v:\Rset^2\times\Rset^d\times\Rset^2\to\Rset^d$
 are $C^4$ and symmetric around $(u,v)=(0,0)$, i.e.,
\begin{equation}
f_u(-u,-v;\mu)=-f(u,v;\mu),\quad
f_v(-u,-v;\mu)=-f(u,v;\mu),
\label{eqn:con5}
\end{equation}
with $f_u(u,v;\mu),f_v(u,v;\mu)=O(|u|^3+|v|^3+|\mu|(|u|+|v|))$,
 $g_u:\Rset^2\times\Rset^d\times\Rset\to\Rset^2$
 and $g_v:\Rset^2\times\Rset^d\times\Rset\to\Rset^d$ are $C^2$,
 and $g_u(u,v,\phi)$ and $g_v(u,v,\phi)$ have mean zero in $\phi$
 for any $(u,v)\in\Rset^2\times\Rset^d$.
In particular, $(u,v)=(0,0)$ is an equilibrium in \eqref{eqn:sysuv} for $\mu\neq 0$
 since by \eqref{eqn:con5}
\[
f_u(0,0;\mu)=f_v(0,0;\mu)=0.
\]
Recall that $J$ is the Jordan normal form given by \eqref{eqn:J}.
We extend \eqref{eqn:sysuv} to the $(d+5)$-dimensional system
\begin{equation}\begin{split}
\dot{u}=&Ju+f_u(u,v;\mu)+\epsilon g_u(u,v,\phi),\\
\dot{v}=&Av+f_v(u,v;\mu)+\epsilon g_v(u,v,\phi),\\
\dot{\phi}=&\omega,\quad
\dot{\mu}=0
\end{split}
\label{eqn:esys}
\end{equation}
in $\Rset^2\times\Rset^d\times\Sset^1\times\Rset^2$.

When $\epsilon=0$,
 the system \eqref{eqn:esys} has a one-dimensional invariant manifold
\[
\M=\{(u,v,\phi,\mu)\mid (u,v,\mu)=(0,0,0),\phi\in\Sset^1\}
\]
for which there is a five-dimensional center manifold
\[
W^\c(\M)=\{(u,v,\phi,\mu)\mid v=\bar{v}(u,\mu),\phi\in\Sset^1,\mu\in U_0\}
\]
near $\M$, where $\bar{v}(u,\mu)=O(|u|^3+|\mu||u|)$
 and $U_0\subset\Rset^2$ is a neighborhood of $\mu=0$.
The dynamics of \eqref{eqn:esys} with $\epsilon=0$ on $W^\c(\M)$ are governed by
\[
\dot{u}=Ju+f_u(u,0;\mu)+O(|u|^5+|\mu|\,|u|^3),\quad
\dot{\phi}=\omega,\quad
\dot{\mu}=0
\]
since
\[
f_u(u,\bar{v}(u;\mu);\mu)=f_u(u,0;\mu)+O(|u|^5+|\mu|\,|u|^3).
\]
See \cite{C81,GH83,HI11,K04,W03} for the general approach of center manifold reduction.
$\N_0=W^\c(\M)$ is also a normally hyperbolic, locally invariant manifold \cite{E13,F71,F74,W94}.

Let $\epsilon>0$.
From the invariant manifold theory of Fenichel \cite{F71,F74} (see also \cite{E13,W94})
 we show that there exists a five-dimensional normally hyperbolic, locally invariant manifold 
\[
\N_\epsilon=\{(u,v,\phi,\mu)\mid
 v=\bar{v}(u,\mu)+\epsilon\tilde{v}(u,\phi,\mu),\phi\in\Sset^1,\mu\in U\}
\]
near $\N_0$, where $U\subset U_0$ is a neighborhood of $\mu=0$.
The dynamics of \eqref{eqn:esys} on $\N_\epsilon$ are governed by
\begin{equation}
\begin{split}
\dot{u}=&Ju+f_u(u,\bar{v}(u,\mu)+\epsilon\tilde{v}(u,\phi,\mu);\mu)
 +\epsilon g_u(u,\bar{v}(u,\mu)+\epsilon\tilde{v}(u,\phi,\mu),\phi),\\
\dot{\phi}=&\omega,\quad
\dot{\mu}=0.
\end{split}
\label{eqn:esyse}
\end{equation}
Here we have
\begin{align*}
&
f_u(u,\bar{v}(u,\mu)+\epsilon\tilde{v}(u,\phi,\mu);\mu)
=f_u(u,0;\mu)+O(|u|^5+|\mu|\,|u|^3+\epsilon|u|^2),\\
&
g_u(u,\bar{v}(u,\mu)+\epsilon\tilde{v}(u,\phi,\mu),\phi)
=g_u(u,0,\phi)+O(|u|^3+\epsilon)
\end{align*}
by assumption.
Hence, letting $\phi=\omega t$,
 we rewrite the first equation of \eqref{eqn:esyse} as
\begin{equation}
\dot{u}=Ju+f_u(u,0;\mu)+\epsilon g_u(u,0,\omega  t)
 +O(|u|^5+|\mu|\,|u|^3+\epsilon|u|^2+\epsilon^2).
\label{eqn:esyse1}
\end{equation}
We eliminate the higher-order terms in \eqref{eqn:esyse1},
 and apply the arguments of Section~2 to obtain \eqref{eqn:pnf}.

\section{Example}

We apply our theory
 to the three-dimensional system \eqref{eqn:ex} with \eqref{eqn:td}.
Henceforth we assume that $\theta_0=0$ or $\pi$ and that $\delta_1\neq \pm 1$. 
Recall that $\alpha,\delta_0>0$.
Following the approach of \cite{Y99} basically,
 we first obtain a two-dimensional system of the form \eqref{eqn:pnf}
 governing its dynamics approximately on the locally invariant manifold $\N_\epsilon$.

\subsection{Center manifold reduction}

Let $\epsilon=0$ in \eqref{eqn:td}, i.e., $\theta_\d=\theta_0\,(=0$ or $\pi$).
The system \eqref{eqn:ex} has an equilibrium at $z=(\theta_\d,0,0)$
 and its right hand side has  the Jacobian matrix
\begin{equation}
\begin{pmatrix}
0 & 1 & 0\\
-\sigma & -\delta_0 & 1\\
-\gamma & -\delta_1 & -\alpha
\end{pmatrix}
\label{eqn:Jex}
\end{equation}
at the equilibrium,
 where $\sigma=1$ if $\theta_0=0$ and $\sigma=-1$ if $\theta_\d=\pi$.
So we easily see that the matrix \eqref{eqn:Jex} has a double zero when
\begin{equation}
\alpha\delta_0+\delta_1+\sigma=0,\quad
\gamma=-\sigma\alpha
\label{eqn:d0ex}
\end{equation}
(see Section~3 of \cite{Y99}).
Let
\begin{equation}
\alpha_0=-\frac{\sigma+\delta_1}{\delta_0},\quad
\gamma_0=\frac{\sigma(\sigma+\delta_1)}{\delta_0},
\label{eqn:exag0}
\end{equation}
so that $\alpha=\alpha_0$ and $\gamma=\gamma_0$ satisfy \eqref{eqn:d0ex}.
The system~\eqref{eqn:ex} with $\theta_\d=0$ or $\pi$
 possesses a symmetry about $z=(0,0,0)$ or $(\pi,0,0)$,
 i.e., it is invariant under the transformation
\[
(z_1,z_2,z_3)\mapsto(-z_1,-z_2,-z_3)\quad\mbox{or}\quad
(z_1-\pi,z_2,z_3)\mapsto(\pi-z_1,-z_2,-z_3).
\]

We now set $\beta>0$.
Letting
\[
\tilde{z}_1=z_1-\theta_0,\quad
\tilde{z}_j=z_j,\quad
j=2,3,
\]
we rewrite \eqref{eqn:ex} as
\begin{align}
\begin{pmatrix}
\dot{\tilde{z}}_1 \\ 
\dot{\tilde{z}}_2\\ 
\dot{\tilde{z}}_3
\end{pmatrix}=&\begin{pmatrix}
0 & 1 & 0 \\ 
-\sigma & -\delta_0 & 1 \\ 
-(\gamma_0+\gamma_1) & -\delta_1 & -(\alpha_0+\alpha_1)
\end{pmatrix} 
\begin{pmatrix}
\tilde{z}_1 \\ 
\tilde{z}_2\\ 
\tilde{z}_3
\end{pmatrix}
+\begin{pmatrix}
0 \\ 
\tfrac{1}{6}\sigma\tilde{z}_1^3\\ 
\epsilon\beta\gamma_0\cos\omega t
\end{pmatrix}+\mbox{h.o.t},
\label{eqn:ex0}
\end{align}
where $\alpha_1=\alpha-\alpha_0$ and $\gamma_1=\gamma-\gamma_0$.
Using the change of the coordinates 
\[
\begin{pmatrix}
\tilde{z}_1\\
\tilde{z}_2\\
\tilde{z}_3
\end{pmatrix}=\begin{pmatrix}
1 & 0 & \delta_0\\
0 & 1 & -\delta_0(\alpha_0+\delta_0)\\
\sigma & \delta_0 & -\delta_1(\alpha_0+\delta_0)+\gamma_0
\end{pmatrix}\begin{pmatrix}
u_1\\
u_2\\
v
\end{pmatrix}
\]
in \eqref{eqn:ex0}, we have
\begin{align}
\begin{pmatrix}
\dot{u}_1 \\ 
\dot{u}_2\\ 
\dot{v}
\end{pmatrix} =&\begin{pmatrix}
\tilde{a}_{11}\alpha_1+\tilde{b}_{11}\gamma_1
 & 1+\tilde{a}_{12}\alpha_1 & \tilde{a}_{13}\alpha_1+\tilde{b}_{13}\gamma_1 \\ 
\tilde{a}_{21}\alpha_1+\tilde{b}_{21}\gamma_1
 & \tilde{a}_{22}\alpha_1 & \tilde{a}_{23}\alpha_1+\tilde{b}_{23}\gamma_1 \\ 
\tilde{a}_{31}\alpha_1+\tilde{b}_{31}\gamma_1
 & \tilde{a}_{32}\alpha_1 & -(\alpha_0+\delta_0)+\tilde{a}_{33}\alpha_1+\tilde{b}_{33}\gamma_1
\end{pmatrix} 
\begin{pmatrix}
u_1 \\ 
u_2\\ 
v
\end{pmatrix}\notag\\
&+
\begin{pmatrix}
\dfrac{\sigma\delta_0 (u_1+\delta_0 v)^3}{6 (\alpha_0+\delta_0)^2} \\ 
\dfrac{\sigma\alpha_0 (u_1+\delta_0 v)^3}{6 (\alpha_0+\delta_0)}\\ 
-\dfrac{\sigma(u_1+\delta_0 v)^3}{6 (\alpha_0+\delta_0)^2}
\end{pmatrix}-
\epsilon\beta\gamma_0\cos\omega t
\begin{pmatrix}
\tilde{b}_{11} \\ 
\tilde{b}_{21}\\ 
\tilde{b}_{31}
\end{pmatrix},
\label{eqn:ex1}
\end{align}
where
\begin{align*}
&
\tilde{a}_{11}=\frac{\sigma}{(\alpha_0+\delta_0)^2},\quad
\tilde{a}_{12}=\frac{\delta_0}{(\alpha_0+\delta_0)^2},\quad
\tilde{a}_{13}=\frac{\delta_0(\alpha_0^2-\delta_1)}{(\alpha_0+\delta_0)^2},\\
&
\tilde{a}_{21}=-\frac{\sigma}{\alpha_0+\delta_0},\quad
\tilde{a}_{22}=-\frac{\delta_0}{\alpha_0+\delta_0},\quad
\tilde{a}_{23}=-\frac{\delta_0(\alpha_0^2-\delta_1)}{\alpha_0+\delta_0},\\
&
\tilde{a}_{31}=-\frac{\sigma}{\delta_0(\alpha_0+\delta_0)^2},\quad
\tilde{a}_{32}=-\frac{1}{(\alpha_0+\delta_0)^2},\quad
\tilde{a}_{33}=-\frac{\alpha_0^2-\delta_1}{\delta_0(\alpha_0+\delta_0)^2},\\
&
\tilde{b}_{11}=\frac{1}{(\alpha_0+\delta_0)^2},\quad
\tilde{b}_{13}=\frac{\delta_0}{(\alpha_0+\delta_0)^2},\quad
\tilde{b}_{21}=-\frac{1}{\alpha_0+\delta_0},\\
&
\tilde{b}_{23}=-\frac{\delta_0}{\alpha_0+\delta_0},\quad
\tilde{b}_{31}=-\frac{1}{\delta_0(\alpha_0+\delta_0)^2},\quad
\tilde{b}_{33}=-\frac{1}{(\alpha_0+\delta_0)^2}.
\end{align*}
The system \eqref{eqn:ex1} has the form \eqref{eqn:sysuv} with $A=-(\alpha_0+\delta_0)$.
So we obtain \eqref{eqn:esyse1} without the higher-order terms as
\begin{align}
\begin{pmatrix}
\dot{u}_1\\
\dot{u}_2
\end{pmatrix}
=&J\begin{pmatrix}
u_1\\
u_2
\end{pmatrix}+\begin{pmatrix}
\dfrac{\sigma\delta_0 {u_1}^3}{6 (\alpha_0+\delta_0)^2}\\
\dfrac{\sigma\alpha_0 {u_1}^3}{6 (\alpha_0+\delta_0)}
\end{pmatrix}\notag\\
&
+\begin{pmatrix}
 (\tilde{a}_{11}\alpha_1+\tilde{b}_{11}\gamma_1) u_1 +\tilde{a}_{12}\alpha_1 u_2\\
 (\tilde{a}_{21}\alpha_1+\tilde{b}_{21}\gamma_1) u_1+\tilde{a}_{22}\alpha_1 u_2
 \end{pmatrix}
-\epsilon\beta\gamma_0\cos\omega t
\begin{pmatrix}
\tilde{b}_{11}\\
\tilde{b}_{21}
\end{pmatrix}.
\label{eqn:ex2}
\end{align}

We now apply the arguments of Section~2.
The system~\eqref{eqn:ex1} has the form \eqref{eqn:sys}
 with $x=(u_1,u_2)$ and $\mu=(\alpha_1,\gamma_1)$,
 where $f(x;\mu)$ is given by \eqref{eqn:f} with
\begin{align*}
&
a_{130}=\frac{\sigma\delta_0}{(\alpha_0+\delta_0)^2},\quad
a_{230}=\frac{\sigma\alpha_0}{\alpha_0+\delta_0},\quad
a_{jkl}=0,\quad
j=1,2,\quad\mbox{for $(k,l)\neq(3,0)$},\\
&
b_{j1k}=\tilde{a}_{jk},\quad
b_{j12}=\tilde{b}_{j1},\quad
b_{j22}=0,\quad
j,k=1,2,
\end{align*}
and
\[
g(x,\phi)=-\beta\gamma_0\cos\phi
\begin{pmatrix}
\tilde{b}_{11}\\
\tilde{b}_{21}
\end{pmatrix}.
\]
We easily see that conditions~\eqref{eqn:mean} and \eqref{eqn:cd} hold, where
\[
c=\dfrac{\sigma\alpha_0}{6 (\alpha_0+\delta_0)},\quad
d=\dfrac{\sigma\delta_0}{2(\alpha_0+\delta_0)^2}.
\]
Equations~\eqref{eqn:nu} and \eqref{eqn:h} become
\begin{equation}
\begin{split}
&
\nu_1=\frac{9\delta_0^2(\tilde{a}_{21}\alpha_1+\tilde{b}_{21}\gamma_1)}
 {\alpha_0^2(\alpha_0+\delta_0)^2},\quad
\nu_2=\frac{3\delta_0((\tilde{a}_{11}+\tilde{a}_{22})\alpha_1+\tilde{b}_{11}\gamma_1)}
 {\alpha_0 (\alpha_0+\delta_0)} ,\\
&
s_1,s_2=\sign(\sigma),\quad
\bar{\omega}=\frac{3\delta_0\omega}{\alpha_0 (\alpha_0+\delta_0)}
\end{split}
\label{eqn:exnu}
\end{equation}
and
\[
h(\phi)=\frac{9\sqrt{6}\beta\gamma_0\delta_0^3}
 {2\alpha_0^{5/2}(\alpha_0+\delta_0)^{9/2}}\cos\phi,
\]
respectively.
We have $s_1,s_2=1$ or $s_1,s_2=-1$ depending on whether $\sigma=1$ or $-1$, i.e.,
 $\theta_0=0$ or $\pi$, since $\alpha_0,\delta_0>0$.


\subsection{Case of $\theta_0=0$}
We begin with the case of $\theta_0=0$, in which $\sigma=1$ and $s_1,s_2=1$.
Let $\nu_1=-\hat{\epsilon}^2$
 and define $\hat{\omega},\Delta$ as in \eqref{eqn:hat}.
We compute \eqref{eqn:hath} for \eqref{eqn:ho1} as
\begin{align*}
\hat{h}(\phi)
 =&\pm\frac{9\sqrt{3}\beta\gamma_0\delta_0^3}
 {2\alpha_0^{5/2}(\alpha_0+\delta_0)^{9/2}}
 \int_{-\infty}^\infty\sech^2\left(\frac{t}{\sqrt{2}}\right)\,\cos(\hat{\omega} t+\phi)\d t\\
 =& \pm\frac{9\sqrt{3}\pi\beta\gamma_0\delta_0^3\hat{\omega}}
 {\alpha_0^{5/2}(\alpha_0+\delta_0)^{9/2}}
 \csch\left(\frac{\pi\hat{\omega}}{\sqrt{2}}\right)\cos\phi
\end{align*}
and \eqref{eqn:hathmn} for \eqref{eqn:po1} as
\begin{align*}
\hat{h}^{m/n}(\phi)
=& \frac{9\sqrt{3}k\beta\gamma_0\delta_0^3}
 {\alpha_0^{5/2}(\alpha_0+\delta_0)^{9/2}(k^2+1)}\\
&\times
\int_0^{m\hat{T}}
 \cn\left(\frac{t}{\sqrt{k^2+1}}\right)\dn\left(\frac{t}{\sqrt{k^2+1}}\right)
 \cos(\hat{\omega} t+\phi)\d t\\[1ex]
=&
\begin{cases}
\displaystyle
\frac{18\sqrt{3}\pi\beta\gamma_0\delta_0^3\hat{\omega}}
 {\alpha_0^{5/2}(\alpha_0+\delta_0)^{9/2}}
 \csch\left(\frac{m\pi K(k')}{2K(k)}\right)\cos\phi & \mbox{if $m$ is odd and $n=1$};\\[3ex]
0 & \mbox{otherwise,}
\end{cases}
\end{align*}
where the resonance relation $n T^k=m\hat{T}=2m\pi/\hat{\omega}$ has been used.
Note that by \eqref{eqn:exag0} $\gamma_0<0$ since $\alpha_0,\delta_0>0$.
Using Theorems~\ref{thm:3a} and \ref{thm:3b} and Remark~\ref{rmk:3b},
 we obtain the following.

\begin{prop}
\label{prop:6a}
Let $\theta_0=0$.
There exist two periodic orbits near $z=(\pm\sqrt{-\nu_1},0,0)$ in \eqref{eqn:ex} with \eqref{eqn:td}
 such that the left branch of the stable $($resp. unstable$)$ manifold of the right periodic orbit
 intersect the right branch of the unstable $($resp. stable$)$ manifold of the left periodic orbit
 transversely
 if for $|\nu_1|,|\nu_2|,\epsilon>0$ sufficiently small
 condition~\eqref{eqn:hcon1a} holds with
\[
\hat{h}_\mathrm{max}=-\frac{9\sqrt{3}\pi\beta\gamma_0\delta_0^3\hat{\omega}}
 {\alpha_0^{5/2}(\alpha_0+\delta_0)^{9/2}}
 \csch\left(\frac{\pi\hat{\omega}}{\sqrt{2}}\right),\quad
\hat{h}_\mathrm{min}=-\hat{h}_\mathrm{max}.
\]
Moreover, heteroclinic bifurcations occur
 near the two curves given by \eqref{eqn:thm3a1} or equivalently by \eqref{eqn:thm3a2}.
\end{prop}

\begin{prop}
\label{prop:6b}
Let $\theta_0=0$ and let $m>0$ be an odd integer.
Then saddle-node bifurcations of $2m\pi/\omega$-periodic orbits around the origin occur
 in \eqref{eqn:ex} with \eqref{eqn:td} near the two curves given by \eqref{eqn:thm3b} with
\[
\hat{h}_\mathrm{max}^{m/1}=-\frac{18\sqrt{3}\pi\beta\gamma_0\delta_0^3\hat{\omega}}
 {\alpha_0^{5/2}(\alpha_0+\delta_0)^{9/2}}
 \csch\left(\frac{m\pi K(k')}{2K(k)}\right),\quad
\hat{h}_\mathrm{min}^{m/1}=-\hat{h}_\mathrm{max}^{m/1}.
\]
Moreover, one of the periodic orbits born at the bifurcations
 is of a sink or saddle type, depending on whether
\begin{equation}
\hat{h}_\mathrm{max}^{m/1},\hat{h}_\mathrm{min}^{m/1}
>-\frac{1}{\Delta}\biggl(\frac{J_1(k,1)J_3(k,1)}{m\hat{T}}-J_2(k,1)\biggr)
\label{eqn:prop6b}
\end{equation}
or the opposite of inequality holds, while the other is always a saddle.
\end{prop}

\begin{figure}[t]
\begin{center}
\includegraphics[scale=0.5]{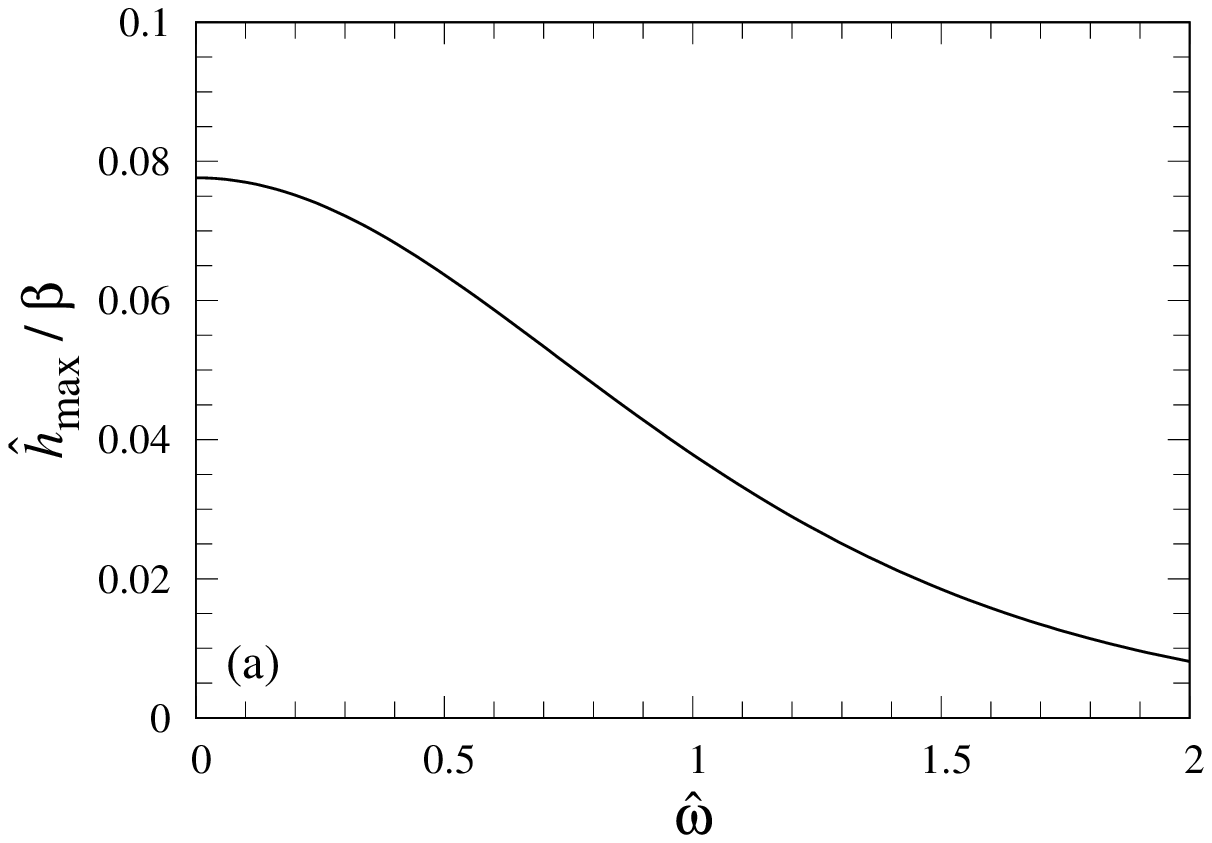}\quad
\includegraphics[scale=0.5]{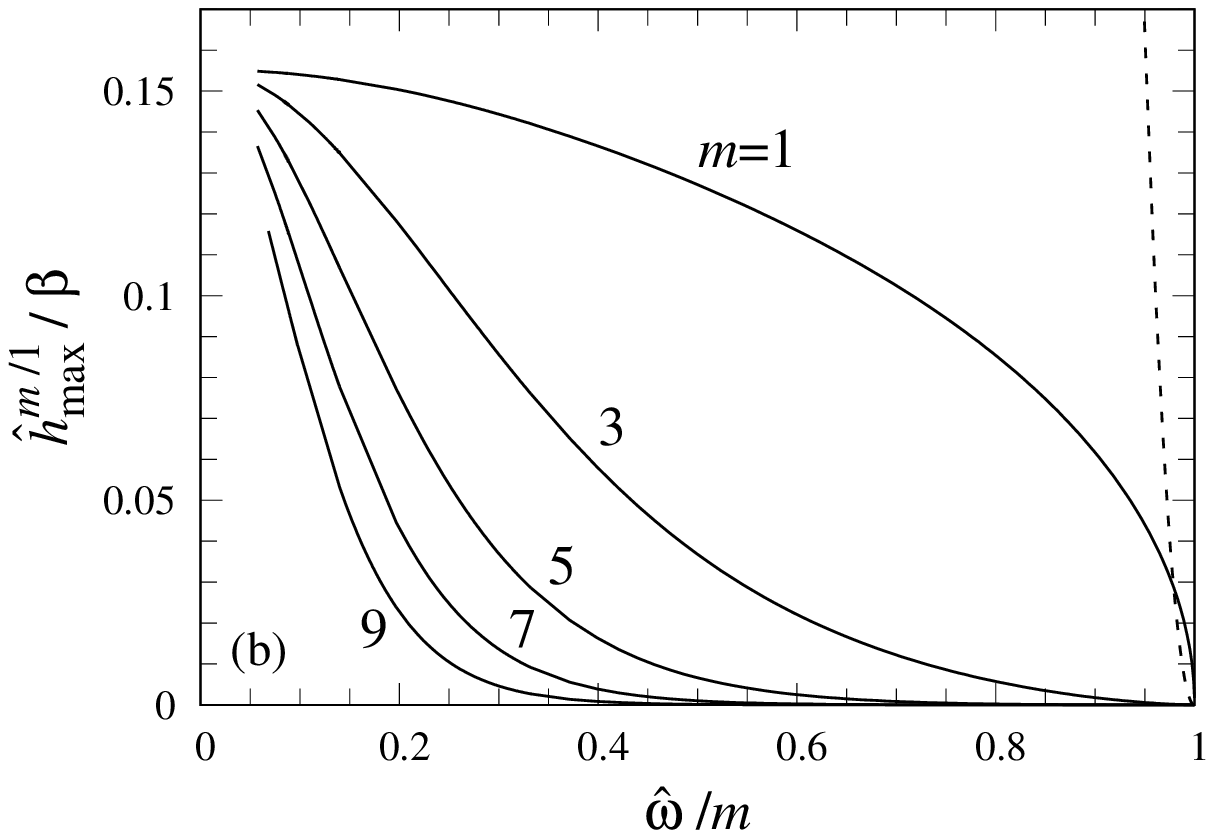}
\end{center}
\caption{Dependence of $\hat{h}_{\mathrm max}$ and $\hat{h}_{\mathrm max}^{m/1}$
 on $\hat{\omega}$ for $\alpha_0=1$, $\gamma_0=-1$ and $\delta_0=0.2$
 in Propositions~\ref{prop:6a} and \ref{prop:6b}:
(a) $\hat{h}_{\mathrm max}$;
(b) $\hat{h}_{\mathrm max}^{m/1}$ for $m=1,3,5,7,9$.
Condition~\eqref{eqn:prop6b} holds for $\hat{h}_{\mathrm{min}}$
 below the broken line in Fig.~(b).}
\label{fig:6a}
\end{figure}

The dependence of $\hat{h}_{\mathrm max}$ and $\hat{h}_{\mathrm max}^{m/1}$
 on $\hat{\omega}$ when $\alpha_0=1$, $\gamma_0=-1$, $\delta_0=0.2$ and $\Delta=1$
 is shown for $m=1,3,5,7,9$ in Fig.~\ref{fig:6a}.
Condition~\eqref{eqn:prop6b} holds for $\hat{h}_{\mathrm{min}}$,
 i.e., one of the periodic orbits born at the saddle-node bifurcation
 given by the second equation of \eqref{eqn:thm3b} is of a sink type,
 below the broken line in Fig.~\ref{fig:6a}(b),
 while it always holds for $\hat{h}_{\mathrm{max}}$.

Henceforth we fix $\delta_0=0.2$, $\delta_1=-1.2$, $\beta=0.5$, $\Delta=1$,
 and $\hat{\omega}=0.8$ or $1.4$, so that $\alpha_0=1$ and $\gamma_0=-1$,
 and give numerical results to demonstrate the above theoretical ones.
The values of $\epsilon$, $\hat{\epsilon}$ and $\omega$ are computed
 from \eqref{eqn:hat}, \eqref{eqn:exnu} and the relation $\nu_1=-\hat{\epsilon}^2$
 when $\alpha$ and $\gamma$ are changed.
Some numerical computations obtained by the computer software \texttt{AUTO} \cite{DO12},
 in which the approaches are described in Appendix~A, are also provided.
 
\begin{figure}[t]
\begin{center}
\includegraphics[scale=0.55]{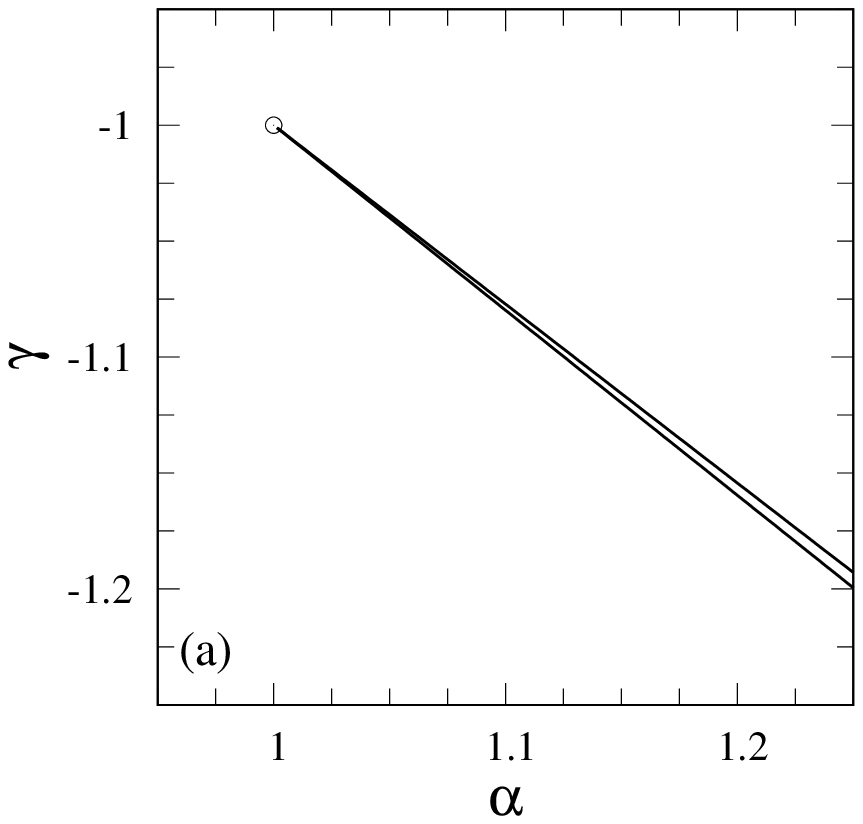}\quad
\includegraphics[scale=0.55]{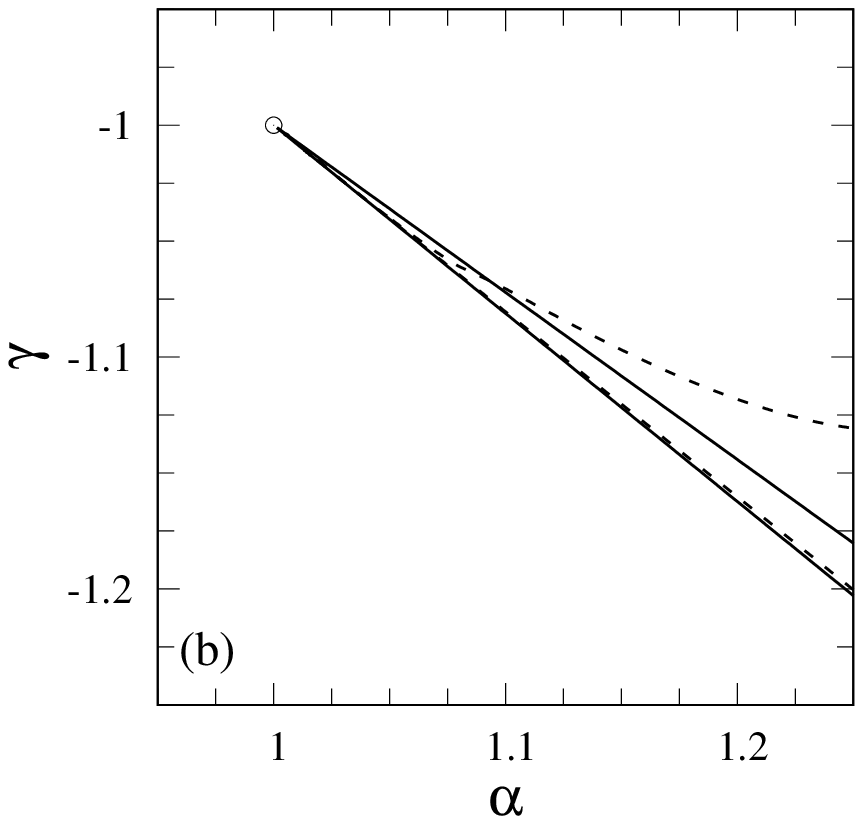}
\end{center}
\caption{Bifurcation sets in \eqref{eqn:ex} with \eqref{eqn:td}
 for $\theta_0=0$, $\delta_0=0.2$, $\delta_1=-1.2$,
 $\beta=5$ and $\hat{\omega}=0.8$:
(a) Heteroclinic bifurcations for $\hat{\omega}=1.4$;
(b) saddle-node bifurcations for $\hat{\omega}=0.8$ and $m=1$.
The circle represents the codimension-two bifurcation point.
The solid and broken lines, respectively, represent the theoretical predictions
 and numerical results in Fig.~(b).}
\label{fig:6b}
\end{figure}
 
Figure~\ref{fig:6b} shows saddle-node bifurcation sets of harmonic orbits ($m,n=1$)
 as well as heteroclinic bifurcation sets,
 detected by Propositions~\ref{prop:6a} and \ref{prop:6b}.
In the regions between the two curves in Figs.~\ref{fig:6b}(a) and (b), respectively,
 there exist transversely heteroclinic and harmonic orbits
 near the unperturbed heteroclinic and periodic orbits with $\epsilon=0$
 (see Figs.~12(c) and 14(a) of \cite{Y99}).
In Fig.~\ref{fig:6b}(b),
 numerical computations by \texttt{AUTO} are also plotted as broken lines
 and found to agree well with the theoretical prediction,
 especially near the codimension-two bifurcation point $(\alpha,\gamma)=(1,-1)$.

\begin{figure}[t]
\begin{center}
\includegraphics[scale=0.51]{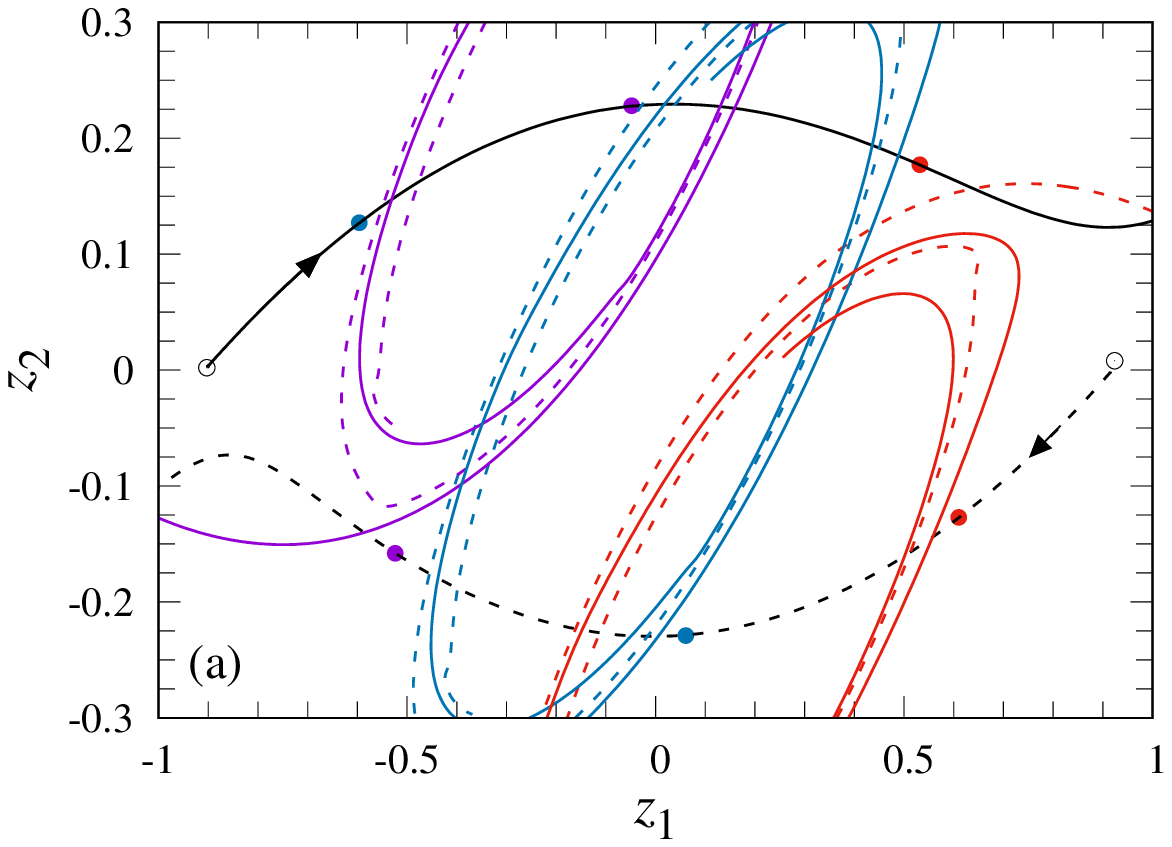}\quad
\includegraphics[scale=0.51]{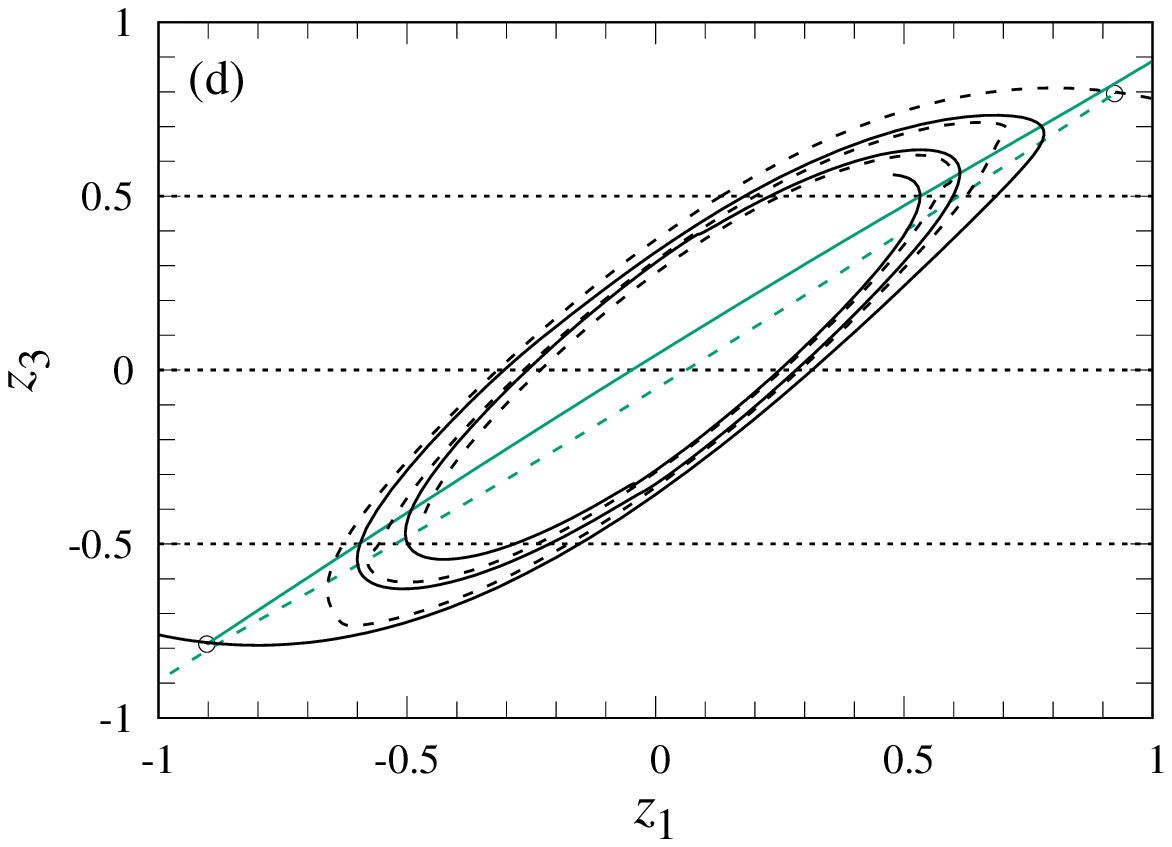}\\[2ex]
\includegraphics[scale=0.51]{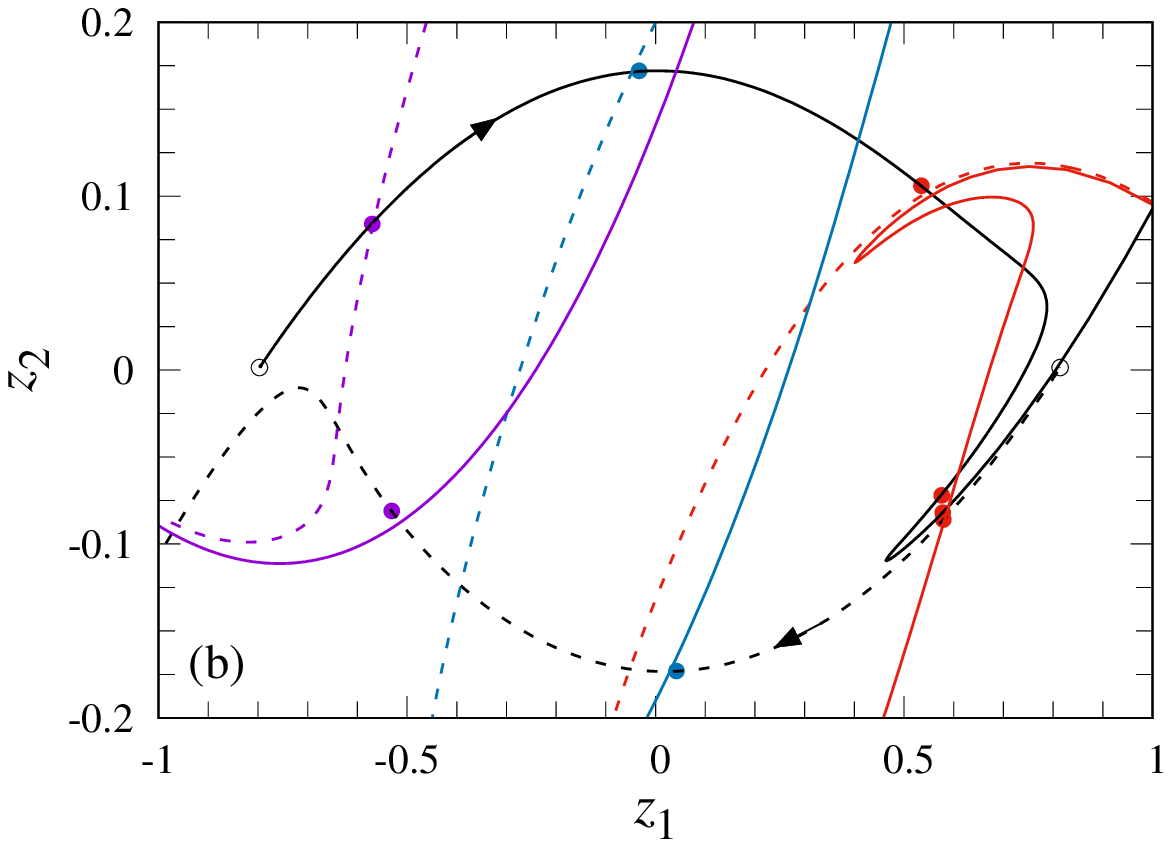}\quad
\includegraphics[scale=0.51]{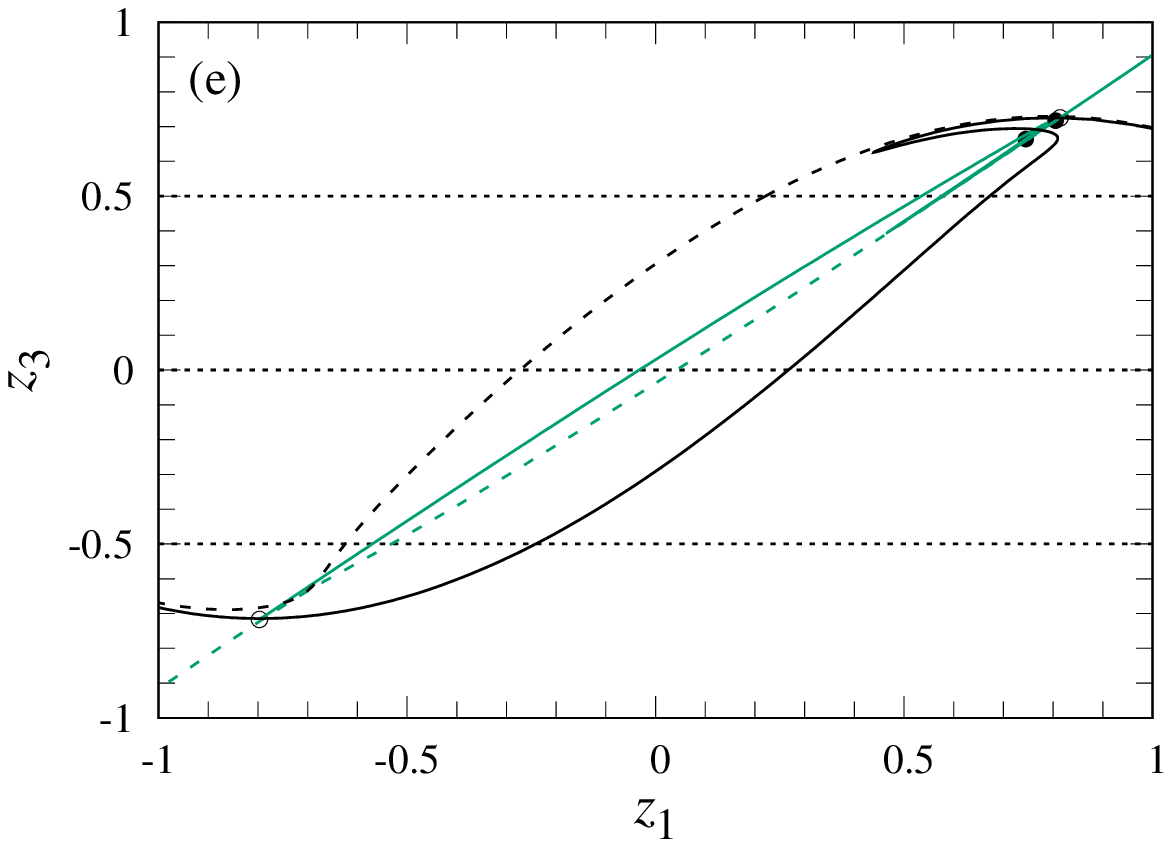}
\end{center}
\caption{Stable and unstable manifolds of periodic orbits
 in \eqref{eqn:ex} with \eqref{eqn:td} on the Poincar\'e section $\{t=0\mod 2\pi/\omega\}$
 for $\alpha=1.5$, $\theta_0=0$, $\delta_0=0.2$, $\delta_1=-1.2$,
 $\beta=5$ and $\hat{\omega}=1.4$:
(a) and (d) $\gamma=-1.1$;
(b) and (e) $\gamma=-1.342915$;
(c) and (f) $\gamma=-1.4$.
The circle represents the periodic orbit.
In Figs.~(a)-(c),
 the black line represents the projection of the unstable manifold onto the $(z_1,z_2)$-plane,
 the small purple, blue and red disks, respectively,  represent
 the intersections of the unstable manifold with the planes $z_3=-0.5$, $0$ and $0.5$,
 and the purple, blue and red lines, respectively, represent
 the intersections of the stable manifolds
 with the planes $z_3=-0.5$, $0$ and $0.5$ onto  the $(z_1,z_2)$-plane.
In Figs.~(d)-(f),
 the black lines represent the intersections of the stable manifold with the plane $z_2=0$,
 the green line represents the projection of the unstable manifold onto the $(z_1,z_3)$-plane,
 and the small black disks represent
 the intersections of the unstable manifold with the planes $z_2=0$. 
The solid and dashed lines represent these intersections or projections
 for the left and right periodic orbits, respectively.
}
\label{fig:6c}
\end{figure}

Figure~\ref{fig:6c} shows the numerically computed stable and unstable manifolds
 of periodic orbits near the unperturbed saddles with $\epsilon=0$,
 the behavior of which is described in Proposition~\ref{prop:6a},
 for $\alpha=1.5$, $\hat{\omega}=1.4$ and $\gamma=-1.1$, $-1.342915$ or $-1.4$.
It was numerically observed in Fig.~14(a) of \cite{Y99}
 that there exist a pair of heteroclinic orbits in \eqref{eqn:ex} with $\theta_\d=0$
 (i.e., $\epsilon=0$) for $\gamma\approx-1.342915$.
We see that these manifolds intersect transversely in Fig.~\ref{fig:6c}(b)
 while they do not in Figs.~\ref{fig:6c}(a) and (c).
Their behavior in Figs.~\ref{fig:6c}(a), (b) and (c) is, respectively, 
 similar to plates~\textcircled{\scriptsize 1},  \textcircled{\scriptsize 3}
 and \textcircled{\scriptsize 5} of Fig.~\ref{fig:thm3a}(c).

 \setcounter{figure}{14}
\begin{figure}[t]
\begin{center}
\includegraphics[scale=0.51]{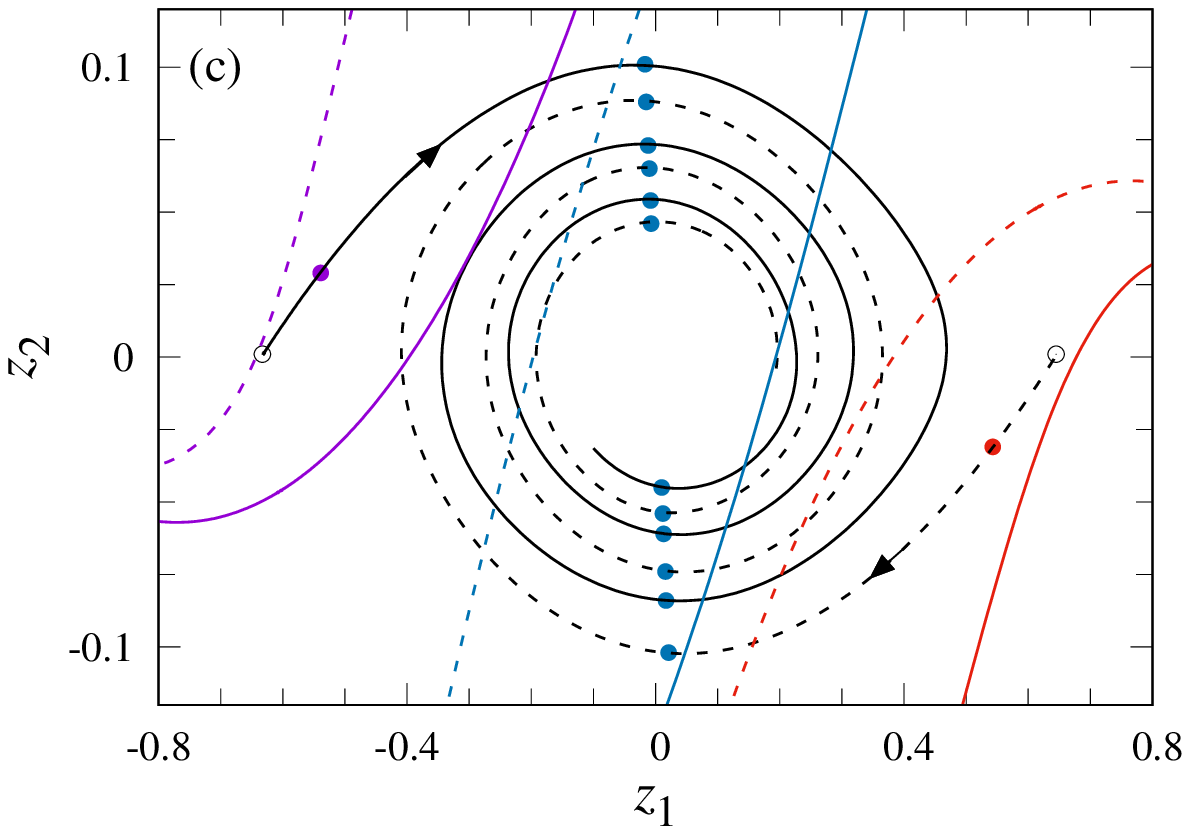}\quad
\includegraphics[scale=0.51]{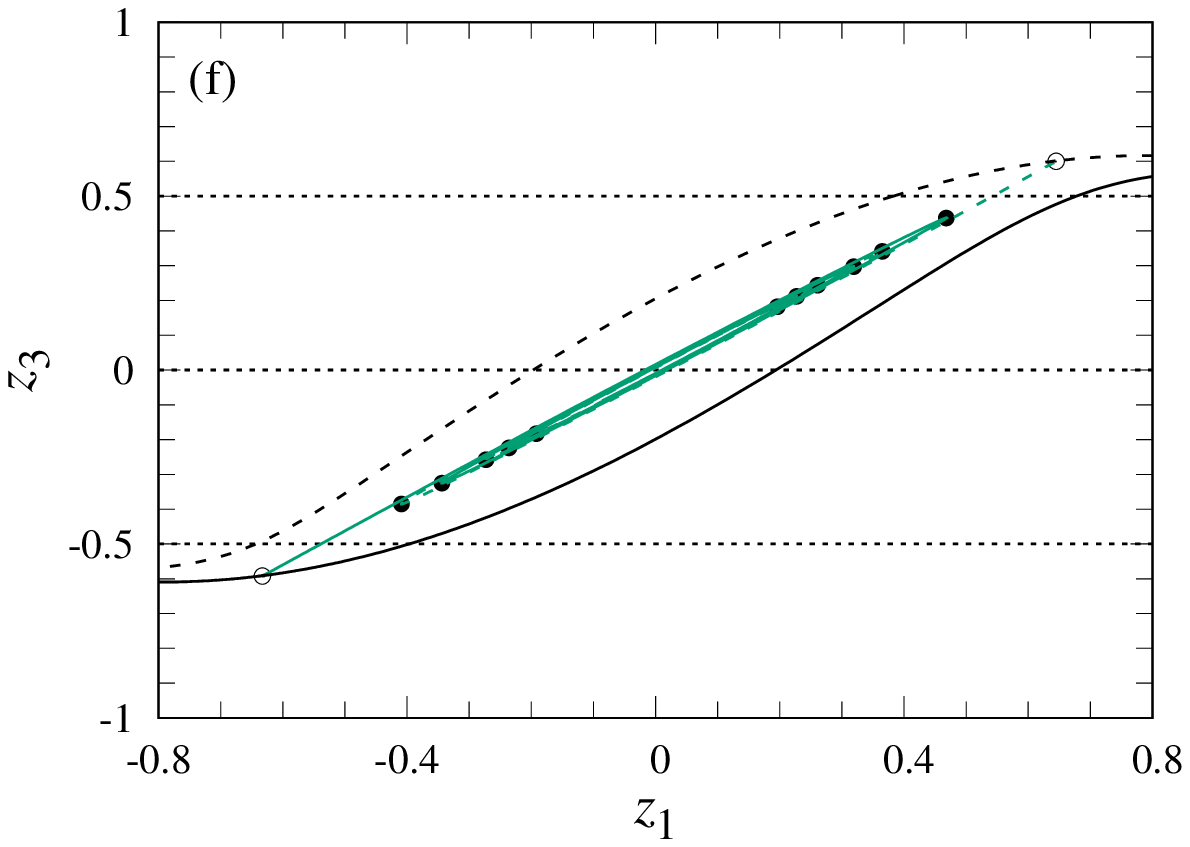}
\end{center}
\caption{Continued.}
\end{figure}

\begin{figure}[t]
\begin{center}
\includegraphics[scale=0.55]{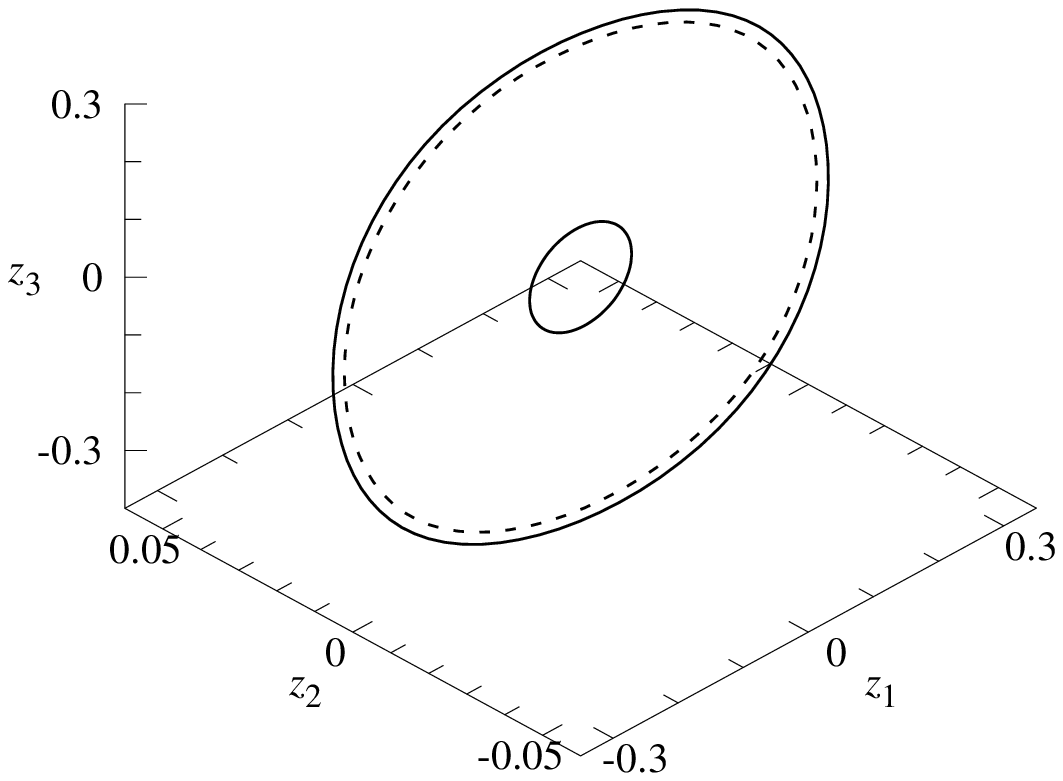}
\end{center}
\caption{Harmonic orbits ($m,n=1$) in \eqref{eqn:ex} with \eqref{eqn:td}
 for $\alpha=1.25$, $\gamma=-1.2$, $\theta_0=0$, $\delta_0=0.2$, $\delta_1=-1.2$,
 $\beta=5$ and $\hat{\omega}=0.8$.
The solid and broken lines, respectively, represent stable and unstable ones.}
\label{fig:6d}
\end{figure}

Figure~\ref{fig:6d} shows numerically computed periodic orbits of period $2\pi/\omega$
 for $\alpha=1.25$ and $\gamma=-1.2$.
The two large ones are born at the saddle-node bifurcation
 detected by Proposition~\ref{prop:6b}:
One of them is stable and the other is unstable, as predicted there.
The small one, which is continued from the equilibrium at the origin
 with $\epsilon$ from $\epsilon=0$,
 always exists near it when $\epsilon>0$ is sufficiently small.

The values of $\alpha,\gamma$ used in Figs.~\ref{fig:6b}-\ref{fig:6d} are not necessarily
 close to the codimension-two bifurcation point $(\alpha,\gamma)=(1,-1)$
 and they are rather far from it in some cases.
Thus, our theoretical results are still valid even in such a situation
 although it does not satisfy our assumption in the theory.

\subsection{Case of $\theta_0=\pi$}
We turn to the case of $\theta_0=\pi$, in which $s_1,s_2=-1$.
We first assume that $\nu_1<0$.
Let $\nu_1=-\hat{\epsilon}^2$ and define $\hat{\omega},\Delta$ as in \eqref{eqn:hat}.
We compute \eqref{eqn:hathmn} for \eqref{eqn:po2a} as
\begin{align*}
\hat{h}^{m/n}(\phi)
=& -\frac{9\sqrt{3}k\beta\gamma_0\delta_0^3}
 {\alpha_0^{5/2}(\alpha_0+\delta_0)^{9/2}(1-2k^2)}\\
& \times\int_0^{m\hat{T}}
 \sn\left(\frac{t}{\sqrt{1-2k^2}}\right)\dn\left(\frac{t}{\sqrt{1-2k^2}}\right)
 \cos(\hat{\omega} t+\phi)\d t\\[1ex]
=&
\begin{cases}
\displaystyle
\frac{18\sqrt{3}\pi\beta\gamma_0\delta_0^3\hat{\omega}}
 {\alpha_0^{5/2}(\alpha_0+\delta_0)^{9/2}}
 \sech\left(\frac{m\pi K(k')}{2K(k)}\right)\sin\phi  & \mbox{if $m$ is odd and $n=1$};\\[3ex]
0 & \mbox{otherwise,}
\end{cases}
\end{align*}
where the resonance relation $n T^k=m\hat{T}=2m\pi/\hat{\omega}$ has been used.
Using Theorem~\ref{thm:4a} and Remark~\ref{rmk:4a},
 we obtain the following.

\begin{prop}
\label{prop:6c}
Let $\theta_0=\pi$ and let $m>0$ be an odd integer.
Then saddle-node bifurcations of $2m\pi/\omega$-periodic orbits around $z=(\pi,0,0)$ occur
 in \eqref{eqn:ex} with \eqref{eqn:td} near the two curves given by \eqref{eqn:thm3b} with
\[
\hat{h}_\mathrm{max}^{m/1}=\frac{18\sqrt{3}\pi\beta\gamma_0\delta_0^3\hat{\omega}}
 {\alpha_0^{5/2}(\alpha_0+\delta_0)^{9/2}}
 \sech\left(\frac{m\pi K(k')}{2K(k)}\right),\quad
\hat{h}_\mathrm{min}^{m/1}=-\hat{h}_\mathrm{max}^{m/1}.
\]
Moreover, one of the periodic orbits born at the bifurcations is of a sink or saddle type,
 depending on whether
\begin{equation}
\hat{h}_\mathrm{max}^{m/1},\hat{h}_\mathrm{min}^{m/1}
<\frac{1}{\Delta}\biggl(\frac{J_1(k,1)J_3(k,1)}{m\hat{T}}+J_2(k,1)\biggr)
\label{eqn:prop6c}
\end{equation}
or the opposite of inequality holds while the other is always a saddle.
\end{prop}

\begin{figure}[t]
\begin{center}
\includegraphics[scale=0.5]{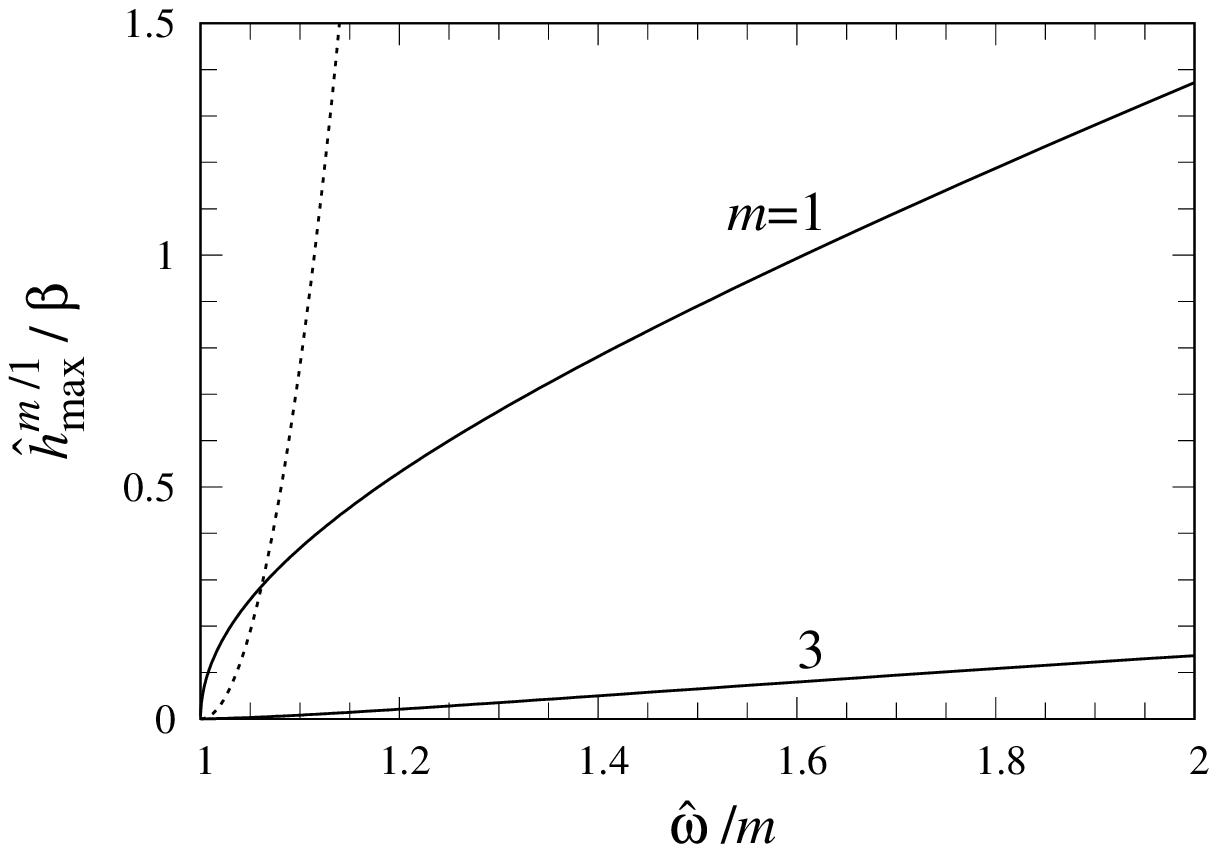}
\end{center}
\caption{Dependence of $\hat{h}_{\mathrm max}^{m/1}$
 on $\hat{\omega}$ for $\alpha_0=1$, $\gamma_0=1$, $\delta_0=0.5$ and $m=1,3$
 in Proposition~\ref{prop:6c}.}
\label{fig:6e}
\end{figure}

The dependence of $\hat{h}_{\mathrm max}^{m/1}$
 on $\hat{\omega}$ when $\alpha_0=1$, $\gamma_0=1$ and $\delta_0=0.5$
 is shown for $m=1,3$ in Fig.~\ref{fig:6e}.
The values of $\hat{h}_{\mathrm max}^{m/1}$ for $m\ge 5$ are very small
 compared with those for $m=1,3$.
Condition~\eqref{eqn:prop6c} 
 always holds for both $\hat{h}_{\mathrm{max}}$ and $\hat{h}_{\mathrm{min}}$,
 i.e., one of the periodic orbits born at the saddle-node bifurcations
 given by the first and second equations of \eqref{eqn:thm3b} is of a sink type.
 
We next assume that $\nu_1>0$.
Let $\nu_1=\hat{\epsilon}^2$
 and define $\hat{\omega},\Delta$ as in \eqref{eqn:hat}.
We compute \eqref{eqn:hath} for \eqref{eqn:ho2} as
\begin{align*}
\hat{h}(\phi)
 =&\mp\frac{9\sqrt{3}\beta\gamma_0\delta_0^3}
 {\alpha_0^{5/2}(\alpha_0+\delta_0)^{9/2}}
 \int_{-\infty}^\infty\sech t\,\tanh t\,\cos(\hat{\omega} t+\phi)\d t\\
 =& \pm\frac{9\sqrt{3}\pi\beta\gamma_0\delta_0^3\hat{\omega}}
 {\alpha_0^{5/2}(\alpha_0+\delta_0)^{9/2}}
 \sech\left(\frac{\pi\hat{\omega}}{2}\right)\sin\phi,
\end{align*}
and \eqref{eqn:hathmn} for \eqref{eqn:po2b1} and \eqref{eqn:po2b2} as
\begin{align*}
\hat{h}^{m/n}(\phi)
=& \mp\frac{9\sqrt{3}k^2\beta\gamma_0\delta_0^3}
 {(\alpha_0(\alpha_0+\delta_0))^{5/2}(2-k^2)}\\
&\times
\int_0^{m\hat{T}}
 \sn\left(\frac{t}{\sqrt{2-k^2}}\right)\cn\left(\frac{t}{\sqrt{2-k^2}}\right)
 \cos(\hat{\omega} t+\phi)\d t\\[1ex]
=&
\begin{cases}
\displaystyle
\pm\frac{9\sqrt{3}\pi\beta\gamma_0\delta_0^3\hat{\omega}}
 {\alpha_0^{5/2}(\alpha_0+\delta_0)^{9/2}}
 \sech\left(\frac{m\pi K(k')}{K(k)}\right)\sin\phi & \mbox{if $n=1$};\\[3ex]
0 & \mbox{otherwise}
\end{cases}
\end{align*}
and
\begin{align*}
\hat{h}^{m/n}(\phi)
=&- \frac{9\sqrt{3}k\beta\gamma_0\delta_0^3}
 {\alpha_0^{5/2}(\alpha_0+\delta_0)^{9/2}(2k^2-1)}\\
&\times
\int_0^{m\hat{T}}
 \sn\left(\frac{t}{\sqrt{2k^2-1}}\right)\dn\left(\frac{t}{\sqrt{2k^2-1}}\right)
 \cos(\hat{\omega} t+\phi)\d t\\[1ex]
=&
\begin{cases}
\displaystyle
\frac{18\sqrt{3}\pi\beta\gamma_0\delta_0^3\hat{\omega}}
 {\alpha_0^{5/2}(\alpha_0+\delta_0)^{9/2}}
 \sech\left(\frac{m\pi K(k')}{2K(k)}\right)\sin\phi & \mbox{if $m$ is odd and $n=1$};\\[3ex]
0 & \mbox{otherwise,}
\end{cases}
\end{align*}
respectively, where the resonance relation $n T^k=m\hat{T}=2m\pi/\hat{\omega}$ has been used.
Using Theorems~\ref{thm:4b}-\ref{thm:4d} and Remark~\ref{rmk:4b},
 we obtain the following.

\begin{prop}
\label{prop:6d}
Let $\theta_0=\pi$.
There exist a periodic orbit near $z=(\pi,0,0)$ in \eqref{eqn:ex} with \eqref{eqn:td}
 such that its stable and unstable manifolds intersect transversely in both sides
 if for $|\nu_1|,|\nu_2|,\epsilon>0$ sufficiently small
 condition~\eqref{eqn:hcon2a} or equivalently \eqref{eqn:hcon2b} holds with
\[
\hat{h}_\mathrm{max}=\frac{9\sqrt{3}\pi\beta\gamma_0\delta_0^3\hat{\omega}}
 {\alpha_0^{5/2}(\alpha_0+\delta_0)^{9/2}}
 \sech\left(\frac{\pi\hat{\omega}}{2}\right),\quad
\hat{h}_\mathrm{min}=-\hat{h}_\mathrm{max}.
\]
Moreover, homoclinic bifurcations occur
 near the two curves given by \eqref{eqn:thm4b1} or equivalently by \eqref{eqn:thm4b2}.
\end{prop}

\begin{prop}
\label{prop:6e}
Let $\theta_0=\pi$ and let $m>0$ be an integer.
Then saddle-node bifurcations of $2m\pi/\omega$-periodic orbits in both sides of $z=(\pi,0,0)$
 occur in \eqref{eqn:ex} with \eqref{eqn:td} near the two curves
 given by \eqref{eqn:thm4c1} or equivalently by \eqref{eqn:thm4c2} with
\[
\hat{h}_\mathrm{max}^{m/1}=\frac{9\sqrt{3}\pi m\beta\gamma_0\delta_0^3\hat{\omega}}
 {\alpha_0^{5/2}(\alpha_0+\delta_0)^{9/2}}
 \sech\left(\frac{m\pi K(k')}{K(k)}\right),\quad
\hat{h}_\mathrm{min}^{m/1}=-\hat{h}_\mathrm{max}^{m/1}.
\]
Moreover, one of the periodic orbits born at the bifurcations is of a sink or saddle type,
 depending on whether
\begin{align}
&
\hat{h}_\mathrm{max}^{m/1},\hat{h}_\mathrm{min}^{m/1}
>-\frac{1}{\Delta}\biggl(\frac{J_1(k,1)J_3(k,1)}{m\hat{T}}-J_2(k,1)\biggr)\notag\\
&
\left(\mbox{resp.}\quad \hat{h}_\mathrm{max}^{m/1},\hat{h}_\mathrm{min}^{m/1}
<\frac{1}{\Delta}\biggl(\frac{J_1(k,1)J_3(k,1)}{m\hat{T}}-J_2(k,1)\biggr) \right)
\label{eqn:prop6e}
\end{align}
or the opposite of inequality holds, while the other is always a saddle.
\end{prop}

\begin{prop}
\label{prop:6f}
Let $\theta_0=\pi$ and let $m>0$ be an odd integer.
Then saddle-node bifurcations of $2m\pi/\omega$-periodic orbits
 encircling the pair of homoclinic orbits to $z=(\pi,0,0)$ for $\epsilon=0$ occur
 in \eqref{eqn:ex} with \eqref{eqn:td} near the two curves given by \eqref{eqn:thm4d} with
\[
\hat{h}_\mathrm{max}^{m/1}=\frac{18\sqrt{3}\pi m\beta\gamma_0\delta_0^3\hat{\omega}}
 {\alpha_0^{5/2}(\alpha_0+\delta_0)^{9/2}}
 \sech\left(\frac{m\pi K(k')}{2K(k)}\right),\quad
\hat{h}_\mathrm{min}^{m/1}=-\hat{h}_\mathrm{max}^{m/1}.
\]
Moreover, one of the periodic orbits born at the bifurcations is of a sink or source type,
 depending on whether
\begin{equation}
\hat{h}_\mathrm{max}^{m/1},\hat{h}_\mathrm{min}^{m/1}
>-\frac{1}{\Delta}\biggl(\frac{J_1(k,1)J_3(k,1)}{m\hat{T}}-J_2(k,1)\biggr)
\label{eqn:prop6f}
\end{equation}
or the opposite of inequality holds, while the other is always a saddle.
\end{prop}

\begin{figure}[t]
\begin{center}
\includegraphics[scale=0.5]{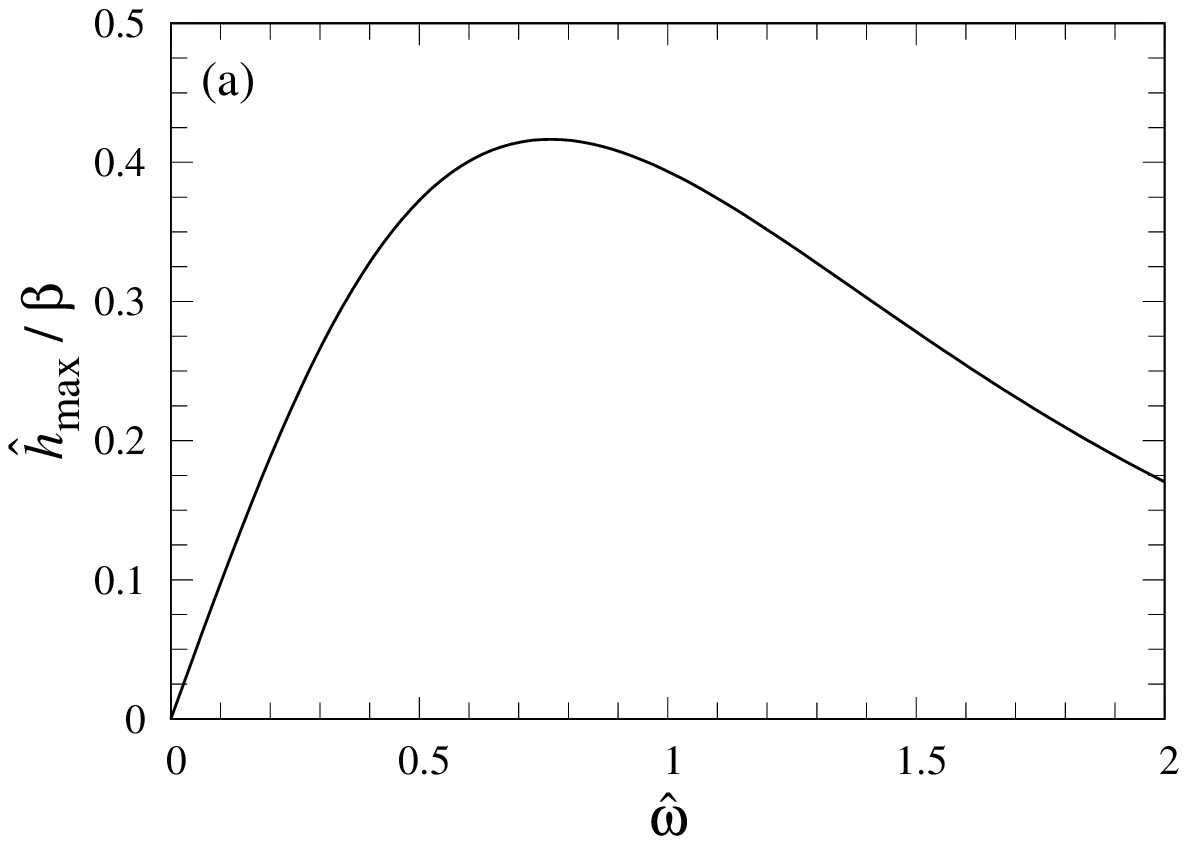}\\[2ex]
\includegraphics[scale=0.5]{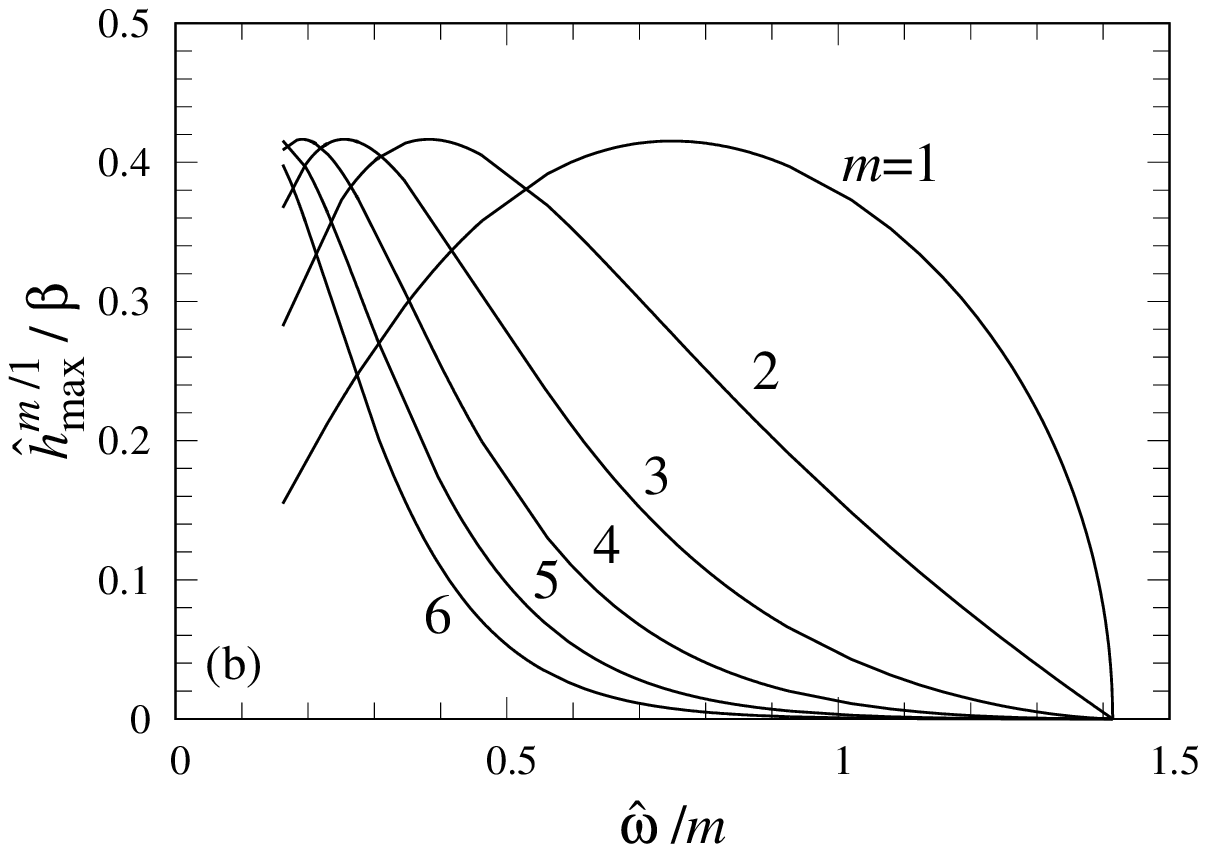}\quad
\includegraphics[scale=0.5]{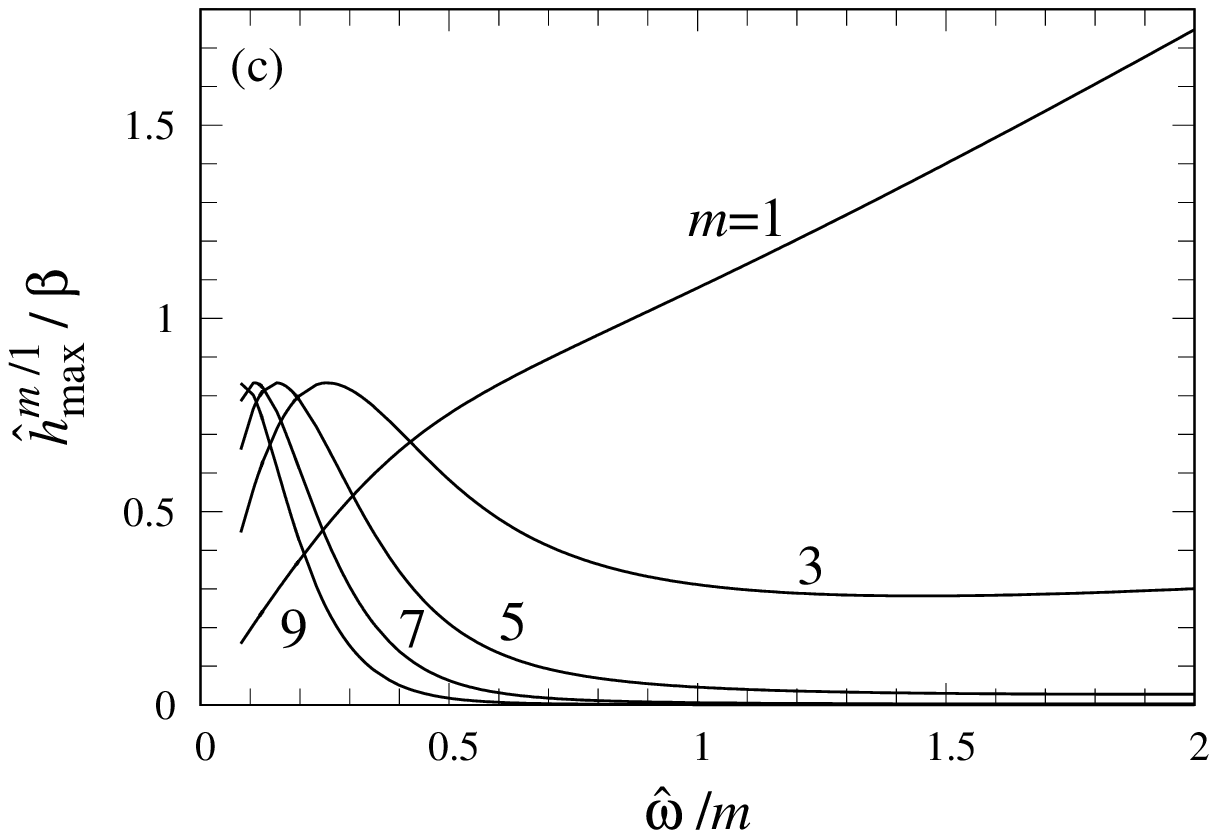}
\end{center}
\caption{Dependence of $\hat{h}_{\mathrm max}$ and $\hat{h}_{\mathrm max}^{m/1}$
 on $\hat{\omega}$ for $\alpha_0=1$, $\gamma_0=1$ and $\delta_0=0.5$
 in Propositions~\ref{prop:6d}-\ref{prop:6f}:
(a) $\hat{h}_{\mathrm max}$;
(b) $\hat{h}_{\mathrm max}^{m/1}$ for $m=1$-$6$ in Proposition~\ref{prop:6e};
(c) $\hat{h}_{\mathrm max}^{m/1}$ for $m=1,3,5,7,9$ in Proposition~\ref{prop:6f}.}
\label{fig:6f}
\end{figure}

Figure~\ref{fig:6f}(a) displays
 the dependence of $\hat{h}_{\mathrm max}$ on $\hat{\omega}$ in Proposition~\ref{prop:6d},
 and Figs.~\ref{fig:6f}(b) and (c) display the dependence of $\hat{h}_{\mathrm max}^{m/1}$
 on $\hat{\omega}$ for $m=1$-$6$ and $m=1,3,5,7,9$,
 in Propositions~\ref{prop:6e} and \ref{prop:6f}, respectively, 
 when $\alpha_0=1$, $\gamma_0=1$ and $\delta_0=0.5$.
Condition~\eqref{eqn:prop6e} was shown numerically to always hold
 for $\hat{h}_{\mathrm{max}}$ and $\hat{h}_{\mathrm{min}}$ in Fig.~\ref{fig:6f}(c),
 i.e., one of the periodic orbits born at the saddle-node bifurcations
 given by \eqref{eqn:thm4c1} or \eqref{eqn:thm4c2} is of a sink type,
Condition~\eqref{eqn:prop6f} was also shown numerically to always hold  in Fig.~\ref{fig:6f}(c).

Henceforth we fix $\delta_0=0.5$, $\delta_1=0.5$, 
 and $\Delta=1$, so that $\alpha_0=1$ and $\gamma_0=1$.
The values of $\epsilon$, $\hat{\epsilon}$ and $\omega$ are computed
 from \eqref{eqn:hat}, \eqref{eqn:exnu}
 and the relation $\nu_1=-\hat{\epsilon}^2$ or $\hat{\epsilon}^2$
 when $\alpha$ and $\gamma$ are changed.
Numerical computations obtained by the computer tool \texttt{AUTO} \cite{DO12}
 are also given (see Appendix~A for the approach used).
 
\begin{figure}[t]
\begin{center}
\includegraphics[scale=0.51]{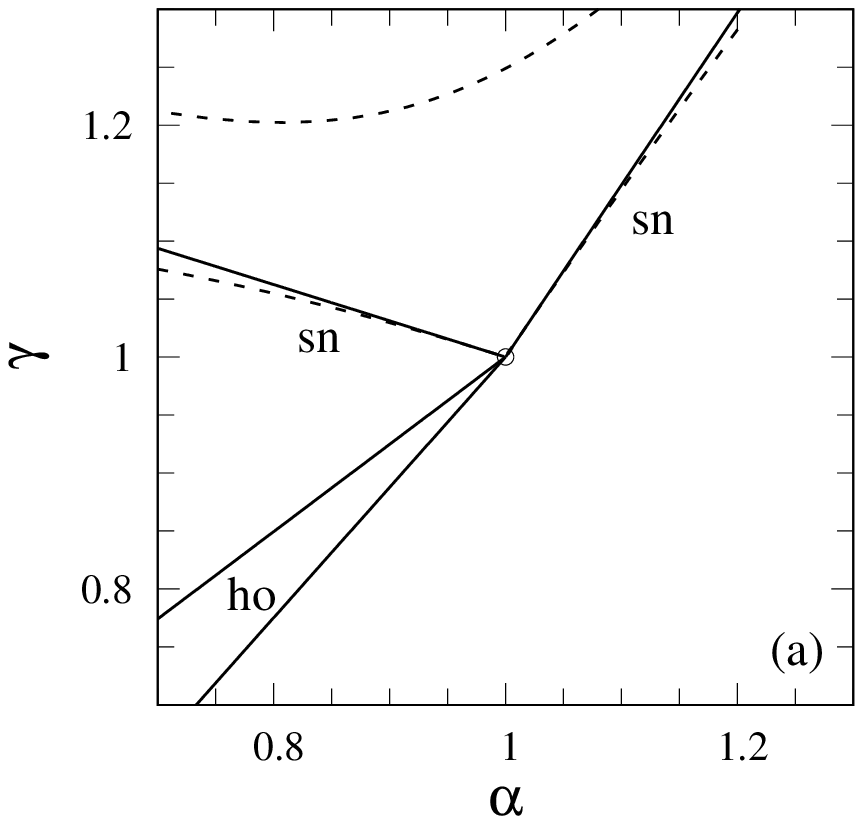}\\[2ex]
\includegraphics[scale=0.51]{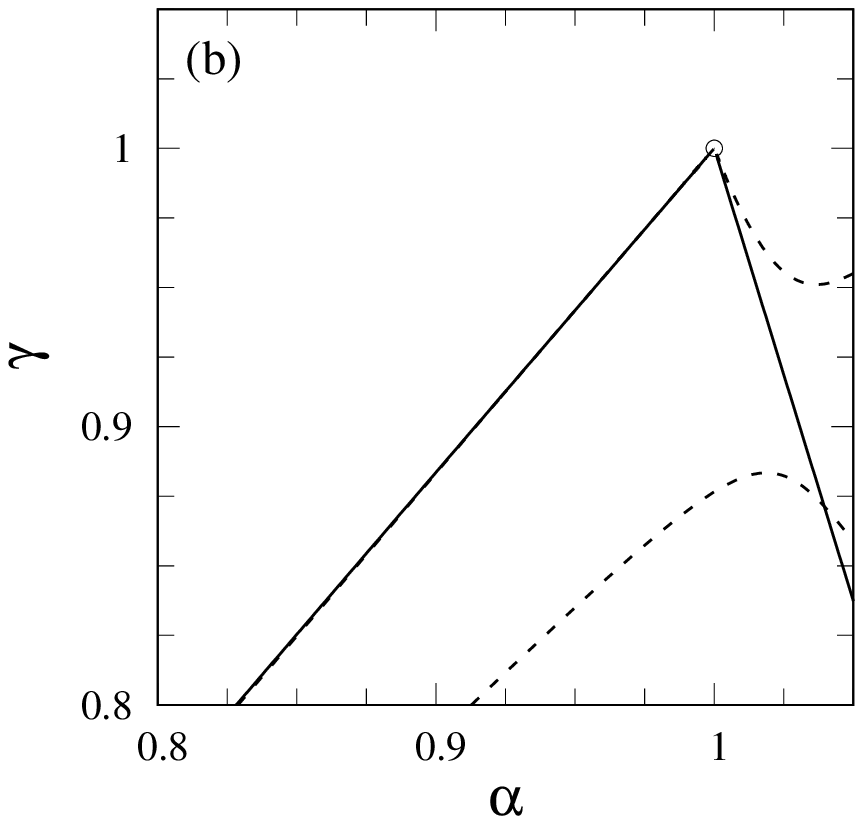}
\includegraphics[scale=0.51]{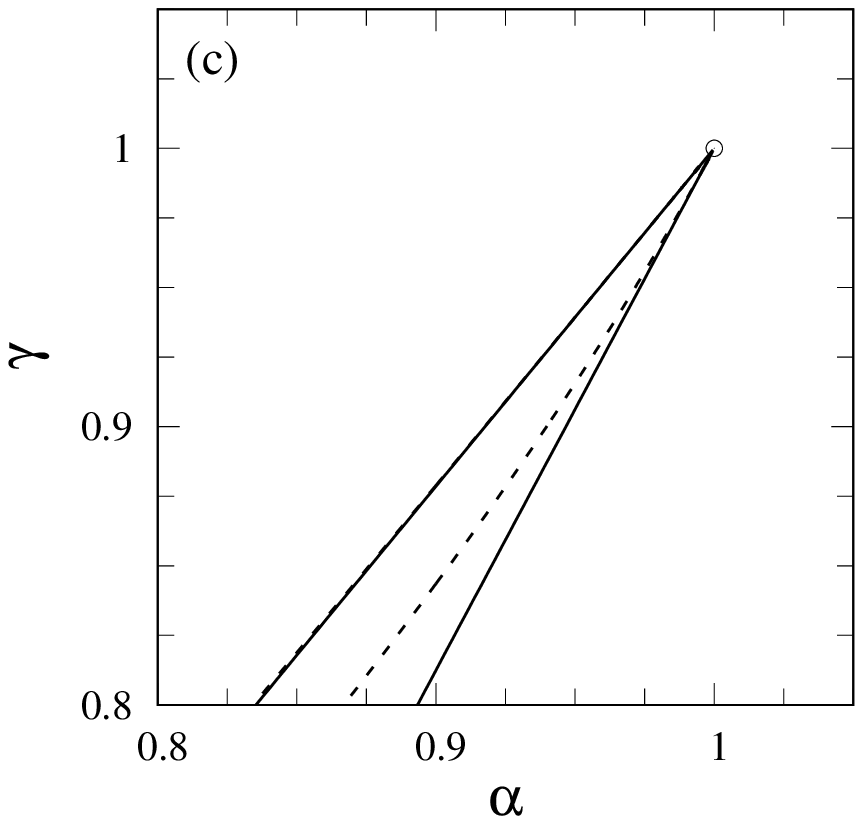}
\end{center}
\caption{Bifurcation sets in \eqref{eqn:ex} with \eqref{eqn:td}
 for $\theta_0=\pi$, $\delta_0=0.5$, $\delta_1=0.5$ and $\beta=5$:
(a) Homoclinic and saddle-node bifurcations for $\hat{\omega}=1.4$; 
(b) and (c) saddle-node bifurcations for $\hat{\omega}=0.8$.
The circle represents the codimension-two bifurcation point.
The solid and broken lines, respectively, represent the theoretical predictions and numerical results
 for the saddle-node bifurcations.
 The labels ``ho'' and ``sn'', respectively, represent heteroclinic and saddle-node bifurcations
 in Fig.~(a).
Bifurcations of periodic orbits inside and outside of the homocliinic orbits
 are plotted in Figs.~(b) and (c).}
\label{fig:6g}
\end{figure}

Figure~\ref{fig:6g} shows saddle-node bifurcation sets of harmonic orbits ($m,n=1$)
 as well as homoclinic bifurcation sets,
 detected by Propositions~\ref{prop:6c}-\ref{prop:6f}.
The saddle-node bifurcation sets in Figs.~\ref{fig:6g}(a), (b) and (c) were plotted
  by using Propositions~\ref{prop:6c}, \ref{prop:6e} and \ref{prop:6f}, respectively.
In the regions between the two curves
 there exist transversely homoclinic and harmonic orbits
 near the unperturbed homoclinic and periodic orbits with $\epsilon=0$
 (see Figs.~13(b) and (d) and 14(b) of \cite{Y99}).
Numerical computations by \texttt{AUTO} for saddle-node bifurcation sets
 are also plotted as broken lines
 and found to agree well with the theoretical predictions,
 especially near the codimension-two bifurcation point $(\alpha,\gamma)=(1,1)$.
In some cases,
 the agreement between theoretical and numerical ones is almost complete.
 
\begin{figure}[t]
\begin{center}
\includegraphics[scale=0.51]{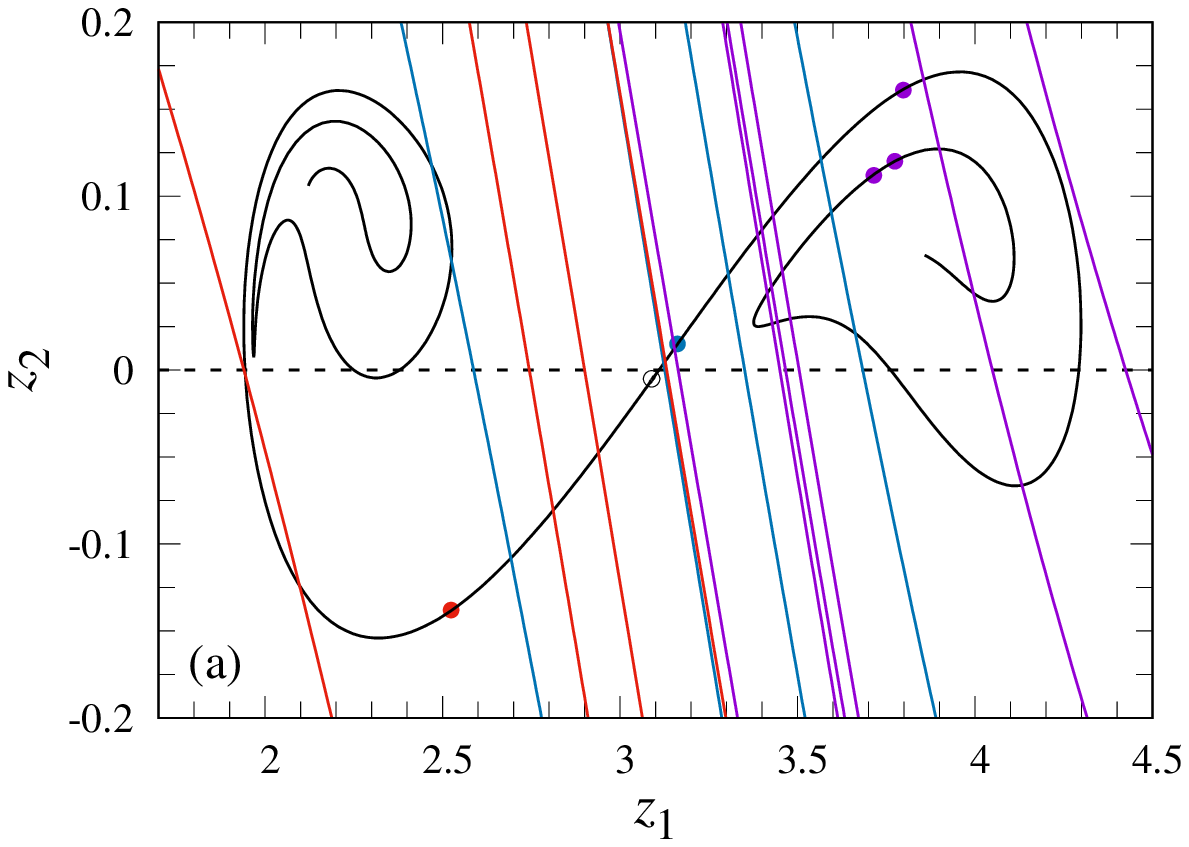}\
\includegraphics[scale=0.51]{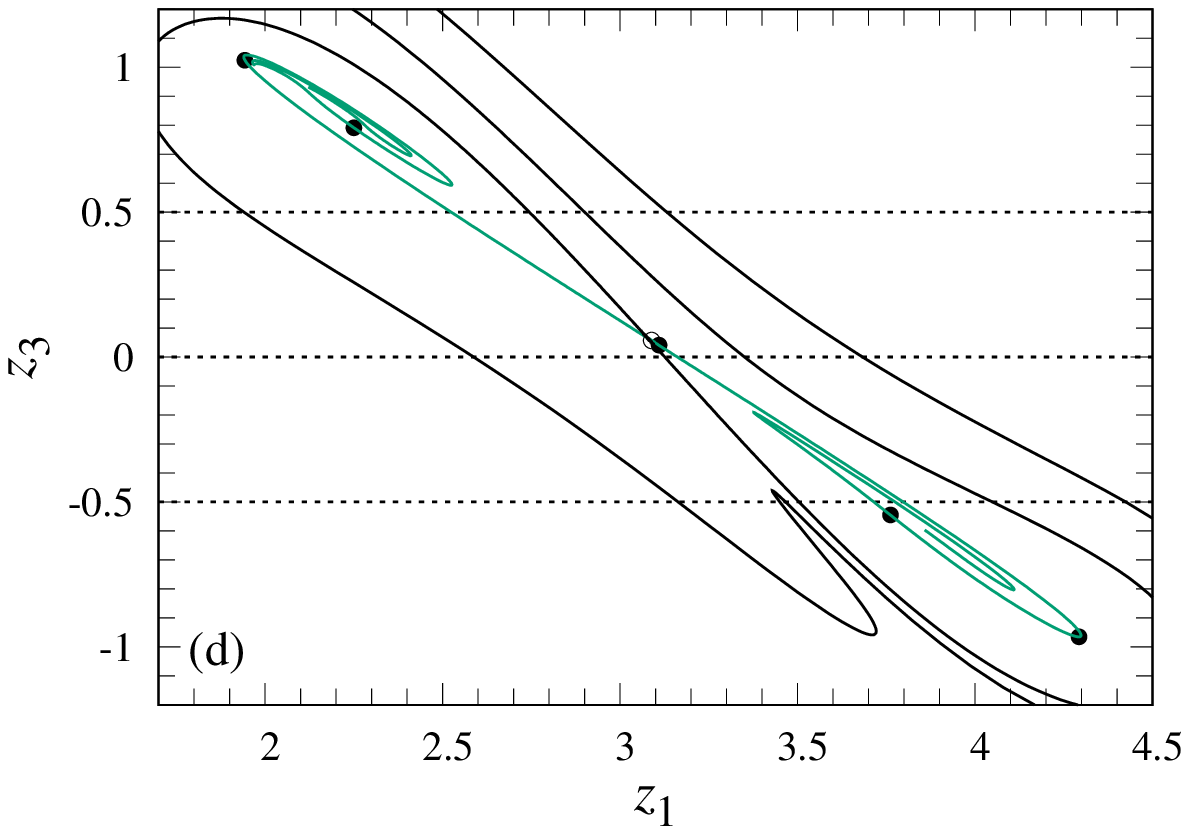}\\[2ex]
\includegraphics[scale=0.51]{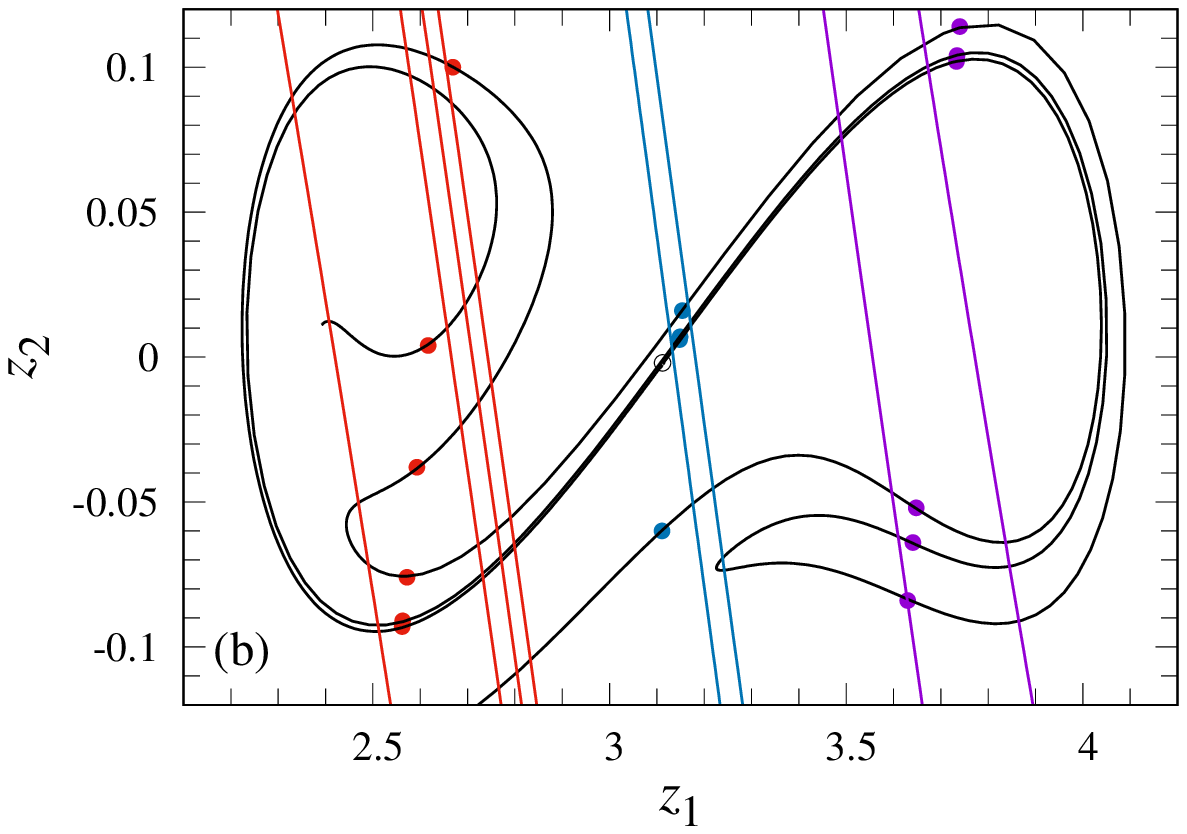}\ 
\includegraphics[scale=0.51]{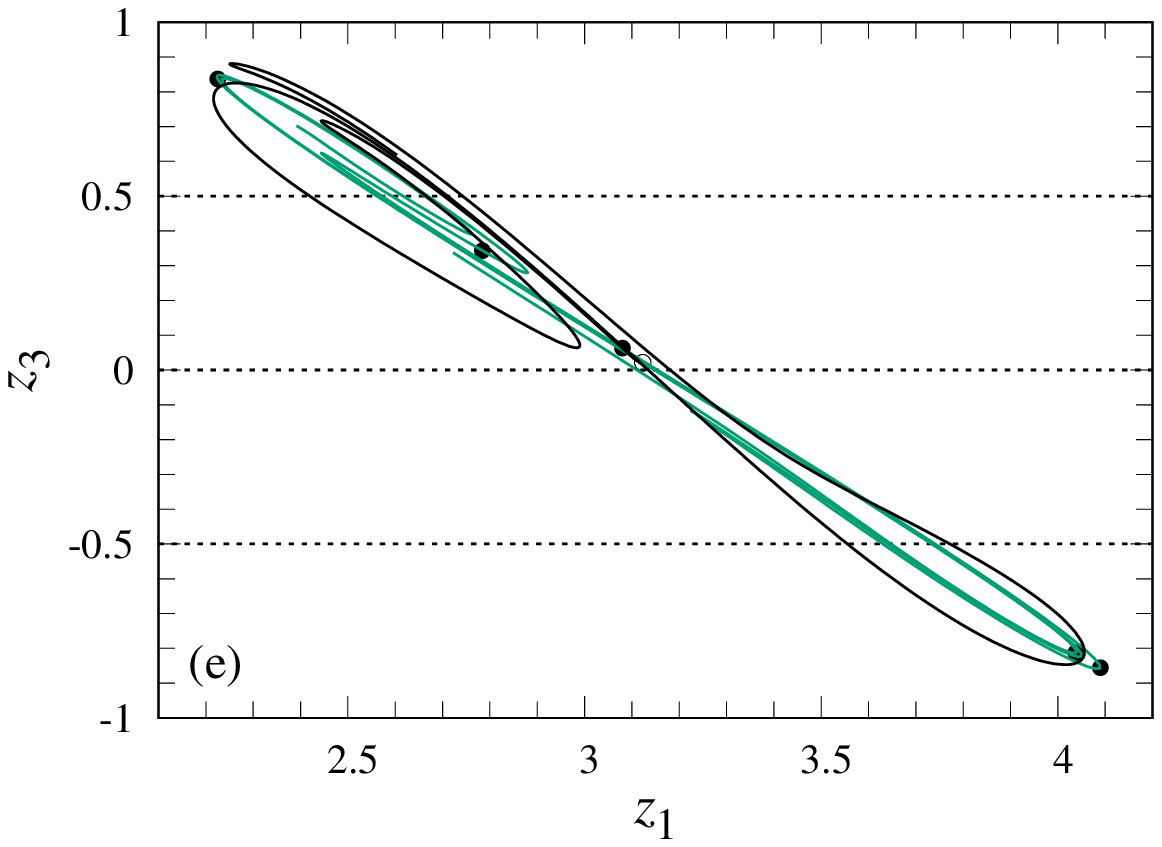}\\[2ex]
\includegraphics[scale=0.51]{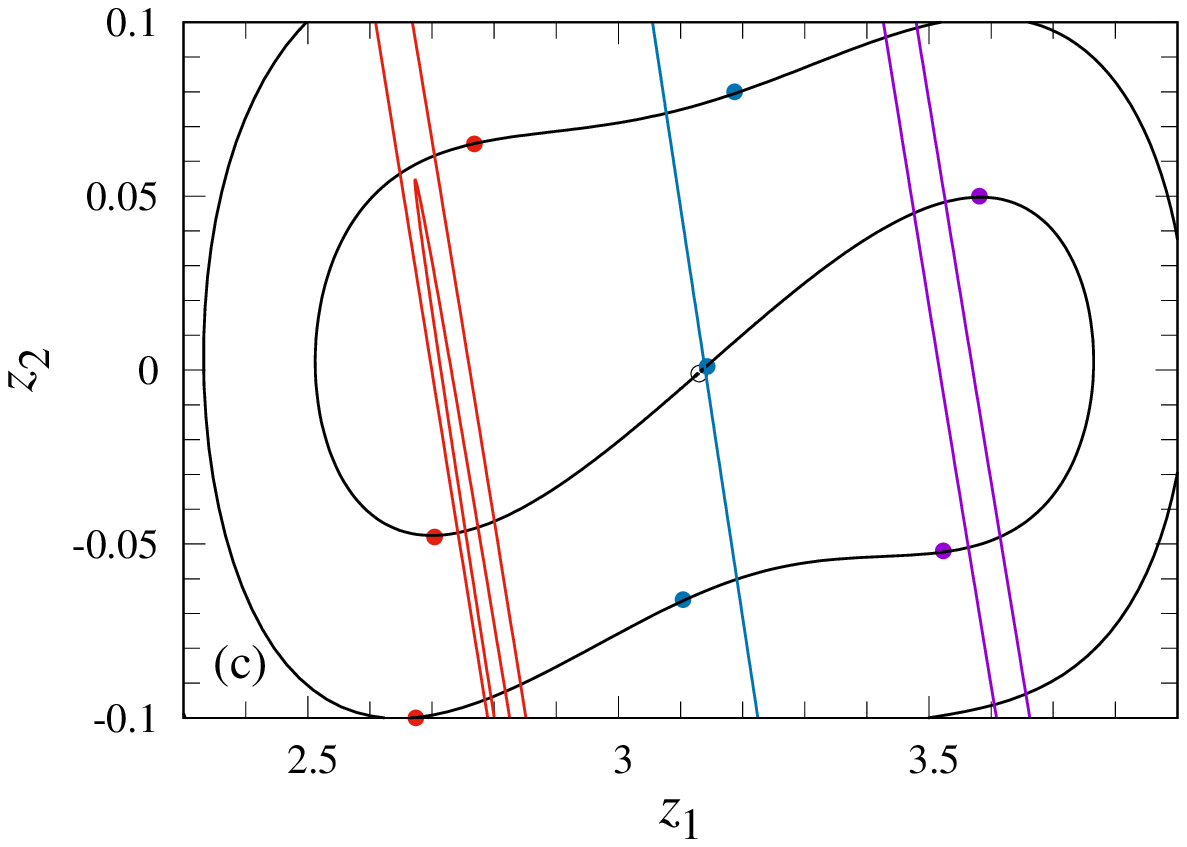}\ 
\includegraphics[scale=0.51]{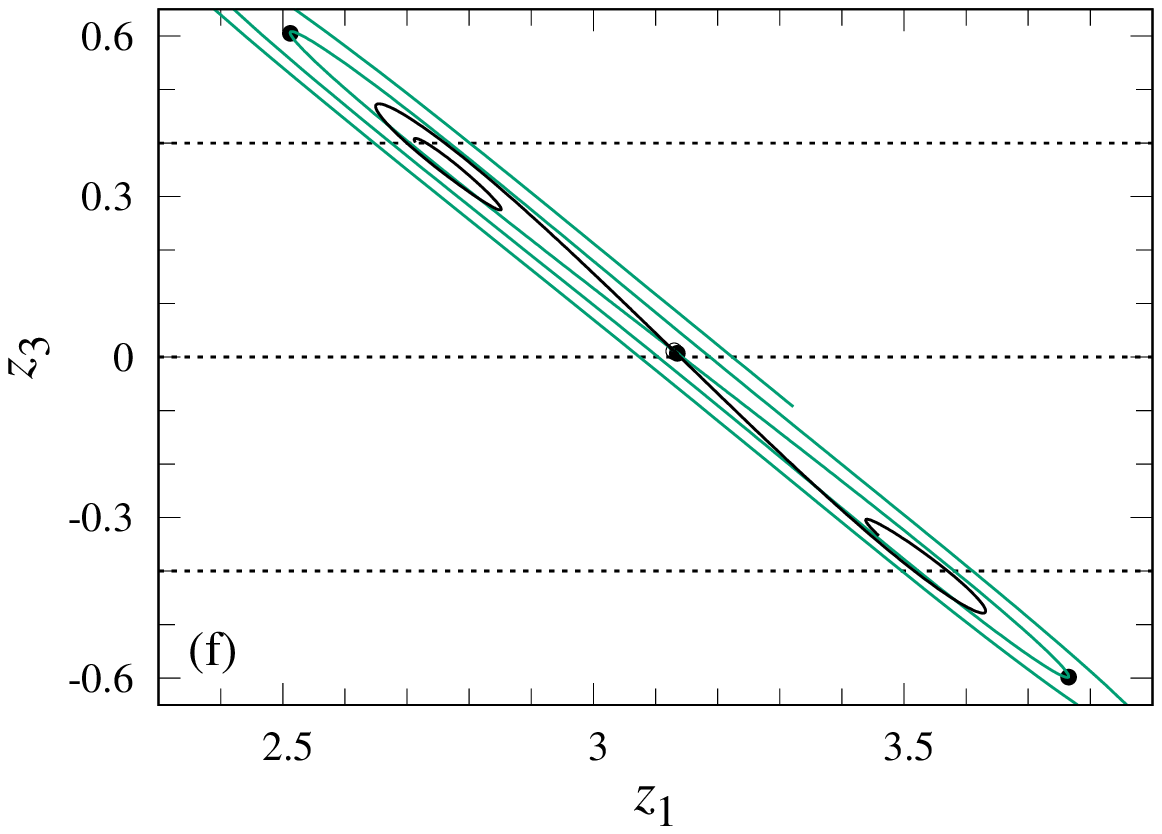}
\end{center}
\caption{Stable and unstable manifolds of a periodic orbit
 in \eqref{eqn:ex} with \eqref{eqn:td} on the Poincar\'e section $\{t=0\mod 2\pi/\omega\}$
 for $\alpha=0.8$, $\theta_0=\pi$, $\delta_0=0.5$, $\delta_1=0.5$,
 $\beta=5$ and $\hat{\omega}=1.4$:
(a) and (d) $\gamma=0.7$;
(b) and (e) $\gamma=0.7471888$;
(c) and (f) $\gamma=0.78$.
In Fig.~(c),
 the small purple, blue and red disks,  respectively, represent
 the intersections of the unstable manifold with the planes $z_3=-0.4$, $0$ and $0.4$,
 and the purple, blue and red lines, respectively, represent
 the intersections of the stable manifolds
 with the planes $z_3=-0.4$, $0$ and $0.4$ onto  the $(z_1,z_2)$-plane.
See the caption of Fig.~\ref{fig:6c} for the meaning of the other lines and symbols
 and those in the other figures.}
\label{fig:6h}
\end{figure}

Figure~\ref{fig:6h} shows the numerically computed stable and unstable manifolds
 of periodic orbits near the origin, the behavior of which are described in Proposition~\ref{prop:6d},
 for $\alpha=0.8$, $\hat{\omega}=1.4$ and $\gamma=0.7$, $0.7471888$ or $0.78$.
It was numerically observed in Fig.~14(b) of \cite{Y99}
 that there exist a pair of homoclinic orbits in \eqref{eqn:ex} with $\theta_\d=\pi$
 (i.e., $\epsilon=0$) for $\gamma\approx0.7471888$.
We observe that these manifolds intersect transversely in Fig.~\ref{fig:6h}(b)
 while they do not in Figs.~\ref{fig:6h}(a) and (c).
Their behavior in Figs.~\ref{fig:6h}(a), (b) and (c) are, respectively, 
 similar to cases~ \textcircled{\scriptsize 1},  \textcircled{\scriptsize 3}
 and \textcircled{\scriptsize 5} of Fig.~\ref{fig:thm4b}(c).

\begin{figure}[t]
\begin{center}
\includegraphics[scale=0.55]{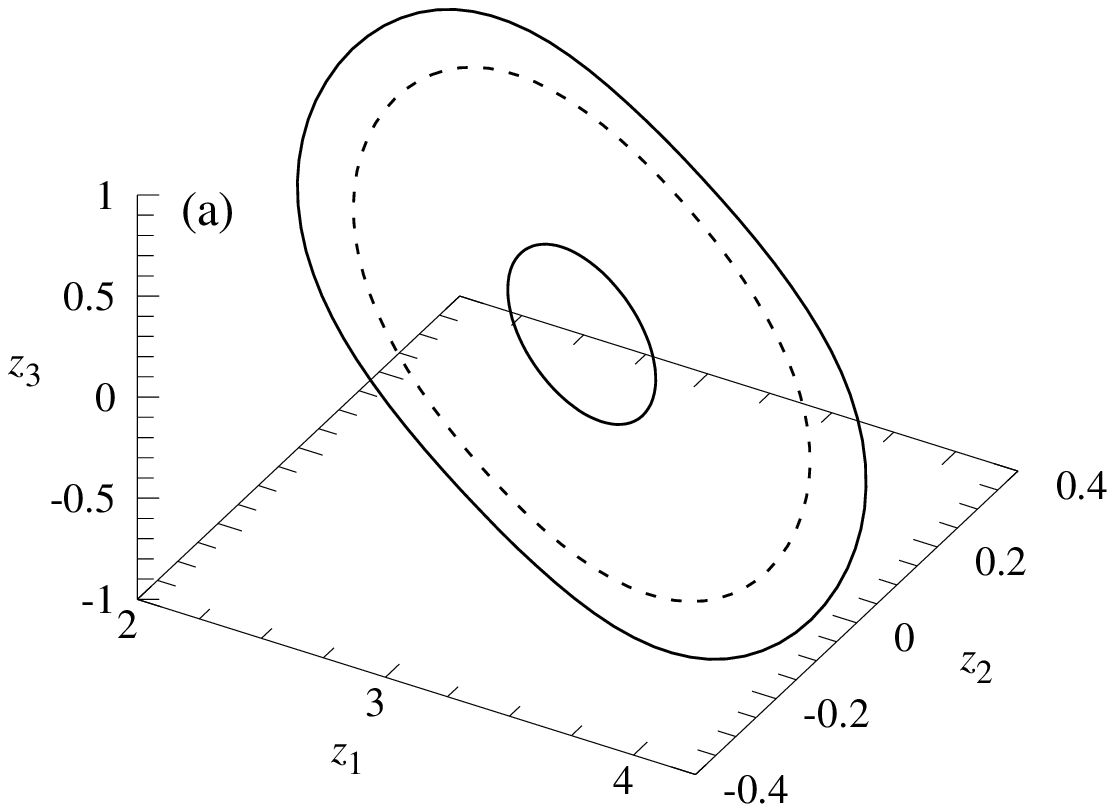}\\
\includegraphics[scale=0.55]{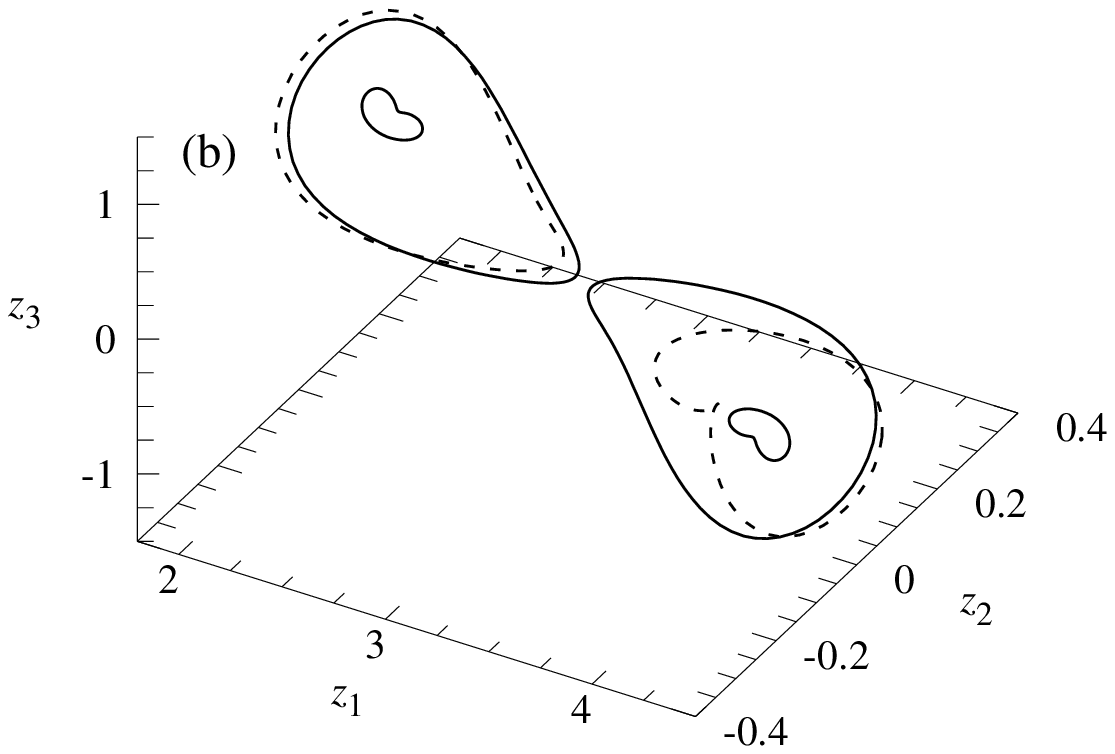}
\includegraphics[scale=0.55]{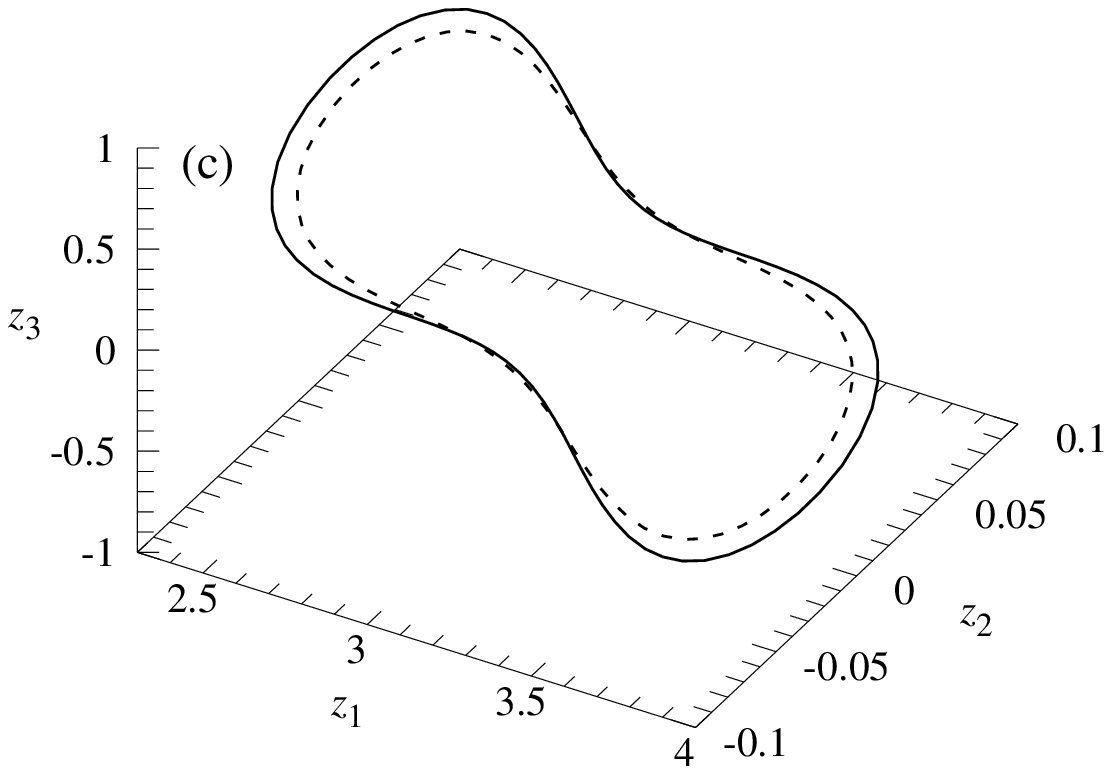}
\end{center}
\caption{Periodic orbits in \eqref{eqn:ex} with \eqref{eqn:td}
 for $\theta_0=\pi$, $\delta_0=0.5$ and $\delta_1=0.5$:
(a) $\alpha=1.2$, $\gamma=1.3$,  $\beta=5$ and $\hat{\omega}=1.4$;
(b) $\alpha=0.8$, $\gamma=0.7$,  $\beta=5$ and $\hat{\omega}=0.8$;
(c) $\alpha=0.9$, $\gamma=0.87$,  $\beta=5$ and $\hat{\omega}=0.8$.
The solid and broken lines, respectively, represent stable and unstable ones.}
\label{fig:6i}
\end{figure}

Figure~\ref{fig:6i} shows numerically computed periodic orbits of period $2\pi/\omega$
 for $\alpha=1.25$ and $\gamma=-1.2$.
Periodic orbits born at saddle-node bifurcations detected
 by Propositions~\ref{prop:6c}, \ref{prop:6e} and \ref{prop:6f}
 are displayed in Figs.~\ref{fig:6i}(a), (b) and (c),  respectively,
 while small periodic orbits, which are continued from the equilibria
 at $z=(\pi,0,0)$ for $(\alpha,\gamma)=(1.2,1.3)$
 and given by \eqref{eqn:a1} for $(\alpha,\gamma)=(0.8,0.7)$ 
 with $\epsilon$ from $\epsilon=0$
 and always exist near them when $\epsilon>0$ is sufficiently small,
 are also plotted in Figs.~\ref{fig:6i}(a) and (b).
One of the periodic orbits bifurcated at the saddle-node bifurcations
 is stable and the other is unstable, as predicted by these propositions.

The values of $\alpha,\gamma$ used in Figs.~\ref{fig:6g}-\ref{fig:6i} are not necessarily
 close to the codimension-two bifurcation point $(\alpha,\gamma)=(1,1)$
 and they are rather far from it in some cases.
Thus, our theoretical results are still valid even in such a situation
 although it does not satisfy our assumption in the theory,
 as in Section~6.2.

\section*{Acknowledgments}
The author would like to thank Takeshi Inoue for his assistance.
This work was partially supported by the JSPS KAKENHI Grant Number
 JP17H02859.


\appendix

\renewcommand{\theequation}{\Alph{section}.\arabic{equation}}

\section{Numerical analyses for \eqref{eqn:ex} with \eqref{eqn:td}} 

In this appendix
 we describe our approach to obtain numerical results of Section~6
  for \eqref{eqn:ex} with \eqref{eqn:td} by the computer software \texttt{AUTO} \cite{DO12}.
The following two cases were considered:
\begin{enumerate}
\setlength{\leftskip}{-1.6em}
\item[(i)]
$\theta_0=0$, $\delta_0=0.2$, $\delta_1=-1.2$, $\beta=5$
 and $\hat{\omega}=1.4$ or $0.8$;
\item[(ii)]
$\theta_0=\pi$, $\delta_0=0.5$, $\delta_1=0.5$, $\beta=5$
 and $\hat{\omega}=1.4$ or $0.8$.
\end{enumerate}
Several values of $\alpha$ and $\gamma$ were taken,
 and the values of $\beta$ and $\hat{\omega}$ were changed
 during some numerical continuations.
We have $(\alpha_0,\gamma_0,\sigma)=(1,-1,1)$ and $(1,1,-1)$
 in cases (i) and (ii), respectively.
We also took $\Delta=1$ in \eqref{eqn:hat}, so that
\[
\omega=\hat{\epsilon}\frac{\alpha_0(\alpha_0+\delta_0)\hat{\omega}}{3\delta_0},\quad
\epsilon=\hat{\epsilon}^4,\quad
\hat{\epsilon}=\frac{3\delta_0 \sqrt{|\tilde{a}_{21}\alpha_1+\tilde{b}_{21}\gamma_1|}}
 {\alpha_0(\alpha_0+\delta_0)},
\]
where $\alpha_1=\alpha-\alpha_0$ and $\gamma_1=\gamma-\gamma_0$.
Numerical results of the codimension-two bifurcations with symmetry
 in \eqref{eqn:ex} with $\theta_\d=\theta_0$ (i.e., $\epsilon=0$ in \eqref{eqn:td})
 for these parameter values are found in \cite{Y99} (see Figs.~10 and 11 there especially).

\subsection{Bifurcations of periodic orbits}

We took the following approach for numerical bifurcation analyses of periodic orbits
 by  the computer tool \texttt{AUTO}
 except for periodic orbits encircling the unperturbed homoclinic orbits in case~(ii),
 such as plotted in Fig.~\ref{fig:6i}(c).
The equilibrium
\begin{align}
z=&(4.024683666180083,0,-0.7727046360165037)\notag\\
& \mbox{or }(2.258501640999503,0,0.7727046360165037)
\label{eqn:a1}
\end{align}
was chosen as the starting solution at $\beta=0$  for $\gamma=0.7$ in case~(ii),
  or $z=(\theta_0,0,0)$ with $\theta_0=0$ or $\pi$ for the other cases,
 and periodic orbits with a period of $2\pi/\omega$
 were numerically continued from the equilibria when $\beta$ was increased.
After that, the periodic orbits with $\beta=5$ were continued with $\gamma$,
 and saddle-node bifurcations were detected.
The saddle-node bifurcations were continued with the two parameters $\alpha,\gamma$,
 and the corresponding bifurcation sets in the $(\alpha,\gamma)$-plane were obtained.
Here the forcing term `$\epsilon\beta\cos\omega t$' was computed
 as a solution to a two-dimensional autonomous system
 as in the \texttt{AUTO} demo \texttt{frc} \cite{DO12}.
 
For periodic orbits encircling  the unperturbed homoclinic orbits in case~(ii),
 such an  orbit at $\beta=0$ was computed in advance by \texttt{AUTO}
 and chosen as the starting solution with $\omega=2\pi/\tilde{T}$,
 where $\tilde{T}$ is its period.
The periodic orbit was numerically continued with $\omega$,
 and the approach stated above was followed after that.
   
\subsection{Computation of stable and unstable manifolds}

As in \cite{Y18},
 we reduce the computation of stable and unstable manifolds of periodic orbits
 to a boundary value problem of differential equations.
See, e.g., Section~5.3 of \cite{Y18} for more details.

\begin{table}[t]
\caption{A pair of saddles in \eqref{eqn:ex} with $\theta_\d=0$
 for $\alpha=1.5$, $\delta_0=0.2$ and $\delta_1=-1.2$.
Here the upper or lower sign is taken simultaneously for each saddle and $z_2=0$.}
\label{tab:a1}
\begin{tabular}{c|c|c}
\hline
$\gamma$ & $z_1$ & $z_3$ \\
\hline
$-1.3$ & $\pm 0.9132926875617018$ & $\pm 0.7915203292201416$ \\
$-1.342915$ & $\pm 0.8056567435744962$ & $\pm 0.7212856838648964$\\
$-1.4$ & $\pm 0.6389467723304353$ & $\pm 0.5963503208417396$\\
\hline
\end{tabular}
\end{table}

From the numerical results of \cite{Y99} (see Fig.~14 there especially)
 we see that there exist pairs of heteroclinic orbits and homoclinic orbits to equilibria
 for $(\alpha,\gamma)\approx(1.5,-1.342915)$ and $(0.8,0.7471888)$
 in cases~(i) and (ii), respectively, when $\beta=0$.
Trajectories on the stable and unstable manifolds of the equilibria for $\beta=0$
 were computed in advance by \texttt{AUTO}
 with the assistance of the \texttt{HomCont} library \cite{CK94,CKS96,DO12}.
Here the equilibria depend on the value of $\gamma$ in case~(i) 
 while the equilibrium does not change at $z=(\pi,0,0)$ in case~(ii).
The loci of the equilibria in case~(i) used for Fig.~\ref{fig:6c}
 are given in Table~\ref{tab:a1}.
The computed trajectories were chosen as the starting solutions
 and continued with $\beta$ or their time length or endpoints.
The saddle-type periodic orbits were also computed at the same time.

The reader may consider that
 we can numerically continue (transverse) homoclinic and heteroclinic orbits
 and compute homoclinic and heteroclinic bifurcation sets
 like the periodic orbits and saddle-node bifurcation sets in Appendix~A.1.
However, we did not succeed in it
 since the period of the saddle-type periodic orbits is very long 
 and the moduli of related positive and negative eigenvalues are very small.
 


\begin{thebibliography}{99}
\bibitem{A72}
V.~I.~Arnold,
Lectures on bifurcations in versal families,
\textit{Russian Math. Surveys}, \textbf{27} (1972), 54–123.
\bibitem{A89}
V.~I.~Arnold,
\textit{Mathematical Method of Classical Mechanics}, 2nd ed.,
Springer, New York, 1989.
\bibitem{BLZ98}
P.W.~Bates, K.~Lu and C.~Zeng,
Existence and persistence of invariant manifolds for semiflows in Banach space,
\textit{Mem. Amer. Math. Soc.}, \textbf{135} (1998), no. 645. 
\bibitem{BLZ99}
P.W.~Bates, K.~Lu and C.~Zeng,
Persistence of overflowing manifolds for semiflow,
\textit{Comm. Pure Appl. Math.}, \textbf{52} (1999), 983--1046.
\bibitem{BRS96}
H.W.~Broer, R.~Roussarie, and C.~Sim\'o,
Invariant circles in the Bogdanov-Takens bifurcation for diffeomorphisms,
\textit{Ergodic Theory Dynam Systems}, \textbf{16} (1996), 1147--1172.
\bibitem{BW95}
K.~Burns and H.~Weiss,
A geometric criterion for positive topological entropy,
\textit{Comm. Math. Phys.}, \textbf{172} (1995), 95--118.
\bibitem{BF54}
P.F.~Byrd and M.D.~Friedman,
\textit{Handbook of Elliptic Integrals for Engineers and Physicists},
Springer, Berlin, 1954.
\bibitem{C81}
J.~Carr,
\textit{Applications of Centre Manifold Theory},
Springer, New York, 1981.

\bibitem{CK94}
A.R.~Champneys and Y.A.~Kuznetsov,
Numerical detection and continuation of codimension-two homoclinic bifurcation analysis,
\textit{Int. J. Bifurcation Chaos}, \textbf{5} (1996), 867--887.
\bibitem{CKS96}
A.R.~Champneys , Y.A.~Kuznetsov and B.~Sandstede,
A numerical toolbox for homoclinic bifurcations,
\textit{Int. J. Bifurcation Chaos}, \textbf{4} (1994), 795--822.

\bibitem{CLW94}
S.-N.~Chow, C.~Li and D.~Wang,
\textit{Normal Forms and Bifurcation of Planar Vector Fields},
Cambridge Univ. Press, Cambridge, 1994.
\bibitem{DO12}
E.~Doedel and B.E.~Oldeman,
\textit{AUTO-07P: Continuation and Bifurcation Software
 for Ordinary Differential Equations},
2012, available online from {\tt http://cmvl.cs.concordia.ca/auto}.
\bibitem{E13}
J.~Eldering,
\textit{Normally Hyperbolic Invariant Manifolds: The Noncompact Case},
Atlantis Press, Paris, 2013.
Springer
\bibitem{F71}
N.~Fenichel,
Persistence and smoothness of invariant manifolds for flows,
\textit{Indiana Univ. Math. J.}, \textbf{21} (1971), 193--225.
\bibitem{F74}
N.~Fenichel,
Asymptotic stability with rate conditions,
\textit{Indiana Univ. Math. J.}, \textbf{23} (1974), 1109--1137.
\bibitem{GH83}
J.~Guckenheimer and P.~Holmes,
\textit{Nonlinear Oscillations, Dynamical Systems, and Bifurcations of Vector Fields},
Springer, New York, 1983.
\bibitem{HI11}
M.~Haragus and G.~Iooss,
\textit{Local Bifurcations, Center Manifolds, and Normal Forms
 in Infinite-Dimensional Dynamical Systems},
Springer, London, 2011.
\bibitem{K04}
Y.A.~Kuznetsov,
\textit{Elements of Applied Bifurcation Theory}, 3rd ed.,
Springer, New York, 2004.
\bibitem{SVM07}
J.A.~Sanders, F.~Verhulst and J.~Murdock, 
\textit{Averaging Methods in Nonlinear Dynamical Systems},
Springer, New York, 2007. 
\bibitem{T74}
F. Takens,
Singularities of vector fields,
\textit{Inst. Hautes \'{E}tudes Sci. Publ. Math.}, \textbf{43} (1974), 47–100.
\bibitem{W94}
S.~Wiggins,
\textit{Normally Hyperbolic Invariant Manifolds in Dynamical Systems},
Springer, New York, 1994.
\bibitem{W03}
S.~Wiggins,
\textit{Introduction to Applied Nonlinear Dynamical Systems and Chaos}, 2nd ed.,
Springer, New York, 2003.
\bibitem{Y94}
K. Yagasaki,
Chaos in a pendulum with feedback control,
\textit{Nonlinear Dynam.}, \textbf{6} (1994), 125--142.
\bibitem{Y96}
K. Yagasaki,
The Melnikov theory for subharmonics and their bifurcations in forced oscillations,
\textit{SIAM J. Appl. Math.}, \textbf{56} (1996), 1720--1765.
\bibitem{Y96b}
K. Yagasaki,
Second-order averaging and Melnikov analyses for forced non-linear oscillators,
\textit{J. Sound Vib.}, \textbf{190} (1996), 587--609.
\bibitem{Y96a}
K. Yagasaki,
A simple feedback control system:
bifurcations of periodic orbits and chaos,
\textit{Nonlinear Dynam.}, \textbf{9} (1996), 391--417.
\bibitem{Y99}
K. Yagasaki,
Codimension-two bifurcations in a pendulum with feedback control,
\textit{Int. J. Non-Linear Mech.}, \textbf{34} (1999), 983--1002.
\bibitem{Y02}
K.~Yagasaki,
Melnikov's method and codimension-two bifurcations in forced oscillations,
\textit{J. Differential Equations}, \textbf{185} (2002), 1--24.
\bibitem{Y18}
K.~Yagasaki,
Heteroclinic transition motions in periodic perturbations of conservative systems
 with an application to forced rigid body dynamics,
\textit{Regul. Chaotic Dyn.}, \textbf{23} (2018), 438--457.
\end{thebibliography}
\end{document}